\documentclass[10pt,intlimits]{amsart}
\usepackage[left=1.4in,right=1.4in,bottom=1in,top=1in]{geometry}

\usepackage{verbatim}
\usepackage{hyperref}
\usepackage[all]{xy}
\numberwithin{equation}{section}
\usepackage{amsmath,amsthm,amssymb,latexsym,amscd,comment}
\usepackage{graphicx}

\newcommand{\udots}{\mathinner{\mskip1mu\raise1pt\vbox{\kern7pt\hbox{.}}
\mskip2mu\raise4pt\hbox{.}\mskip2mu\raise7pt\hbox{.}\mskip1mu}}

\newcommand{\Jblock}[1]
           {
              \ensuremath{              
              \begin{array}{cc}
                   0 & #1  \\
                     -#1 & 0 
              \end{array}
              }
           }
\newcommand{\inverseJblock}[1]
           {
              \ensuremath{              
              \begin{array}{cc}
                   0 & -#1  \\
                   #1 & 0 
              \end{array}
              }
           }  
           
\newcommand{\fourBlock}[4]
           {
               \ensuremath{
               \begin{array}{cccc}
                  0   &    0    & #3 &  0  \\
                  0   &    0    &  0   & #4 \\
                  #1  &    0    &  0   &  0 \\
                  0   &   #2    &  0   &  0 
               \end{array}
               }
           }
\newcommand{\ZeroBlock}[1]
           {
              \ensuremath{              
              \begin{array}{cc}
                   0 & 0 \\
                   0 & 0 
              \end{array}
              }
            }

\newcommand{\Tr}{\mathrm{Tr}}

\newcommand{\C}{\mathbb{C}}

\newcommand{\Z}{\mathbb{Z}}

\newcommand{\SO}{\mathrm{SO}}

\newcommand{\SL}{\mathrm{SL}}

\def\ga{\mathfrak{a}}

\def\gg{\mathfrak{g}}
\def\gh{\mathfrak{h}}

\def\gk{\mathfrak{k}}
\def\gl{\mathfrak{l}}
\def\gm{\mathfrak{m}}
\def\gn{\mathfrak{n}}

\def\gp{\mathfrak{p}}

\def\gt{\mathfrak{t}}
\def\gu{\mathfrak{u}}

\def\gF{\mathfrak{F}}
\def\gB{\mathfrak{B}}

\def\C{\mathbb{C}}

\def\H{\mathbb{H}}

\def\N{\mathbb{N}}

\def\R{\mathbb{R}}

\def\Z{\mathbb{Z}}

\def\cA{\mathcal{A}}
\def\cB{\mathcal{B}}

\def\cH{\mathcal{H}}

\def\cK{\mathcal{K}}

\def\cS{\mathcal{S}}

\newtheorem{theorem}{Theorem}

\newtheorem{lemma}[theorem]{Lemma}

\newtheorem{corollary}[theorem]{Corollary}

\newtheorem{definition}[theorem]{Definition}

\numberwithin{theorem}{section}

\newcommand{\fa}{\mathfrak{a}}

\newcommand{\fg}{\mathfrak{g}}

\newcommand{\fk}{\mathfrak{k}}

\newcommand{\Rank}{\mathrm{Rank}}
\newcommand{\Sp}{\mathrm{Sp}}
\newcommand{\Spin}{\mathrm{Spin}}
\newcommand{\SU}{\mathrm{SU}}
\newcommand{\U}{\mathrm{U}}
\newcommand{\diag}{\mathrm{diag}}
\newcommand{\so}{\mathfrak{so}}
\newcommand{\su}{\mathfrak{su}}
\newcommand{\lsp}{\mathfrak{sp}}
\newcommand{\lsl}{\mathfrak{sl}}

\newcommand{\bmid}{\,\, \right|\,\,}

\newcommand{\Lie}{\mathrm{Lie}}

\def\sideremark#1{\ifvmode\leavevmode\fi\vadjust{\vbox to0pt{\vss% the remark
 \hbox to 0pt{\hskip\hsize\hskip1em%                          will appear only
\vbox{\hsize2cm\tiny\raggedright\pretolerance10000 %          on the side
 \noindent #1\hfill}\hss}\vbox to8pt{\vfil}\vss}}} %          in 2cm
\begin{document}

\title[Conical Representations]{Conical Representations for Direct Limits of Symmetric Spaces}

\author{Matthew Dawson}
\address{Department of Mathematics, Louisiana State University, Baton Rouge, LA
70803, U.S.A.}
\email{mdawso5@math.lsu.edu}
\author{Gestur \'{O}lafsson}
\address{Department of Mathematics, Louisiana State University, Baton Rouge,
LA 70803, U.S.A.}
\email{olafsson@math.lsu.edu}


\begin{abstract}
We extend the definition of conical representations for Riemannian symmetric spaces to a certain class of infinite-dimensional Riemannian symmetric spaces.   Using an infinite-dimensional version of Weyl's Unitary Trick, there is a correspondence between smooth representations of infinite-dimensional noncompact-type Riemannian symmetric spaces and smooth representations of infinite-dimensional compact-type symmetric spaces.  We classify all smooth conical representations which are unitary on the compact-type side.  Finally, a new class of non-smooth unitary conical representations appears on the compact-type side which has no analogue in the finite-dimensional case.  We classify these representations and show how to decompose them into direct integrals of irreducible conical representations.
\end{abstract}

\maketitle
 % Finally, their relationships with spherical representations and admissible representations (in the sense of Olshanskii) is determined.

\pagenumbering{arabic}
%\doublesize %puts doublespace just in text not in the headings
%\input{exabst}

\section{Introduction}

Harmonic analysis and representation theory of topological groups have been very well-studied over the past century and have produced fruitful applications to areas such as PDEs and mathematical physics.  Two broad developments in the theory are brought together in this paper: first, Helgason's theory of conical representations for noncompact-type Riemannian symmetric spaces and second, the more recent study of representation theory and harmonic analysis on infinite-dimensional Lie groups.

In the theory of Riemannian symmetric spaces, there are two crucially important dualities.  One is the duality between compact-type and noncompact-type Riemannian symmetric spaces.  The other is the duality between a noncompact-type Riemannian symmetric space and its horocycle space.  These dualities are intimately connected to the representation theory of their corresponding isometry groups (see \cite{He1970, He2, He3}).  For instance, Weyl's unitary trick sets up a correspondence between finite-dimensional spherical representations for a compact-type symmetric space and finite-dimensional spherical representations for its corresponding noncompact-type symmetric space.  In turn, the finite-dimensional spherical representations for a noncompact-type symmetric space are identical to the conical representations for its corresponding horocycle space.  Furthermore, analysis on a noncompact-type Riemannian symmetric space and its corresponding horocycle space are connected by the famous Radon transform.

More recently, much progress has been made in the theory of infinite-dimensional Lie groups, that is, groups which are modeled on locally convex topological vector spaces in the same way that finite-dimensional Lie groups are modeled on finite-dimensional vector spaces.  The simplest and ``smallest'' infinite-dimensional groups are the \textit{direct-limit groups}, which are constructed by taking unions of increasing chains of finite-dimensional Lie groups.  In a similar way, one can form an infinite-dimensional symmetric space by forming a direct limit of finite-dimensional symmetric spaces.  Representation theory and even harmonic analysis questions for direct-limit groups and direct-limit symmetric spaces have been studied in some depth (e.g., see \cite{Bel, BOl, Ol1978, Ol1984, OW2011, OW2013, W2011, W2009, W2010} for just a few examples).  A good overview of the field may be found in \cite{Ol1990}.

In particular, spherical representations for infinite-dimensional symmetric spaces are well-studied in the literature (e.g., see \cite{DOW, F2008, Ol1984}).  On the other hand, the theory of conical representations for infinite-dimensional Riemannian symmetric spaces appears to have been  largely neglected up to this point.  However, understanding conical representations is a necessary first step to constructing and studying a Radon transform for infinite-dimensional symmetric spaces.  We point the reader to \cite{HO}, for instance, in which the authors define and study a Radon transform between regular functions on certain direct limits of Riemannian symmetric spaces and regular functions on the corresponding horocycle spaces.

In this paper, we classify all of the smooth conical representations for direct limits of noncompact-type Riemannian symmetric spaces that satisfy certain technical conditions.  Combined with the results of \cite{DOW}, we see that for infinite-dimensional symmetric spaces of infinite rank, none of the smooth conical representations are spherical, a situation which is in stark contrast with the classical result of Helgason that all finite-dimensional representations are spherical if and only if they are conical.  We further demonstrate the existence, in certain cases, of nonsmooth unitary conical representations for direct limits of compact-type Riemannian symmetric spaces.  This is a phenomenon which has no analogue for finite-dimensional symmetric spaces.  We also show how these conical representations decompose into direct integrals of irreducible representations.

The outline is as follows.  In Section~2, we quickly review the notation and relevant results on finite-dimensional noncompact-type Riemannian symmetric spaces.  In Section~3, we review the notions of \textit{propagated} direct limits of symmetric spaces (introduced in \cite{OW2013, W2011,W2010}) and \textit{admissible} direct limits of symmetric spaces (introcuced in \cite{HO}), which provide the technical context for our results.  It remains an open question whether every propagated direct limit is admissible, but in the appendix we have included a case-by-case proof that each of the classical examples of propagated direct limits of symmetric spaces are admissible.  In Section~4, we continue by reviewing relevant results on representations of direct-limit groups and prove some new results which we will need later in the paper.  Section~5 contains the main classification theorem for conical representations.  Finally, in Section~6, we prove some decomposition theorems for general (not necessarily irreducible) conical representations.  We end the paper with a summary in Section 7 and the aforementioned appendix containing the proof of admissiblility for the classical direct limits of symmetric spaces.

%Because the class of all topological groups is far too general to admit many useful general theorems, a few simplifying assumptions are usually made when studying them.  The assumptions come in two broad classes:  separability assumptions, which assure that the topology is rich enough to distinguish elements of the group as well as compactness and countability assumptions, which assure that the group is not ``too big.''

%First of all, topological groups are nearly always assumed to be Hausdorff.  This is a very modest assumption that essentially all interesting groups satisfy although the Zariski topology on an algebraic group is a notable exception.   In fact, for a group to be Hausdorff it is sufficient to assume only that the singleton containing the identity element is closed.  In other words, the topology contains enough information to distinguish points on the group from each other.  For instance, one can show that every Hausdorff group is, in fact, completely regular.

%By far the most important assumption commonly made for topological groups, though, especially for the purpose of harmonic analysis, is local compactness.  Locally compact groups (and \textit{only} locally compact groups) admit Borel measures invariant under the group action.  This, in turn, provides the existence of a group $C^*$-algebra that carries all of the information from the representation theory of the group.  Another common assumption is separability.  Groups which are both separable and locally compact admit an abstract Plancherel formula.

\section{Finite-Dimensional Riemannian Symmetric Spaces}
Riemannian symmetric spaces form a class of particularly well-behaved homogeneous spaces with a rich structure theory and relatively well-understood harmonic analysis.  %Among other important properties, they possess a Riemannian metric that is invariant under the action of the translation group. Furthermore, the isotropic subgroup is fixed under an involution on the translation group, which essentially forces the regular representations on Riemannian symmetric spaces to have multiplicity-free direct integral decompositions. 
In particular, there is a beautiful duality between \textit{compact-type} and \textit{noncompact-type} Riemannian symmetric spaces.

In addition, the noncompact-type Riemannian symmetric spaces possess an associated homogeneous space called a \textit{horocycle space}.  The relationship between a Riemannian symmetric space and its horocycle space is analogous to, for instance, the relationship between points and hyperplanes in $\R^n$, or the relationship between points and horocycles of hyperbolic space (it is for this reason that the terminology \textit{horocycle space} was originally chosen).

In the late 1950s, Gelfand and Graev developed a ``horospherical method'' which relates harmonic analysis on the noncompact-type Riemannian symmetric space $\SL(n,\R)/\SU(n)$ and harmonic analysis on its horocycle space (see \cite[p. 283--287]{MackeyBook}).  These ideas were generalized to all noncompact-type Riemannian symmetric spaces and developed quite completely in the pioneering work of Helgason (see \cite{He1970}, for instance).
The relationship between symmetric spaces and horocycle spaces, together with its implications for representation theory, provides the primary context for this paper.

See \cite{He1} for a comprehensive overview of the structure theory for Riemannian symmetric spaces.  See also \cite{He2} and \cite{He3} for applications of representation theory to analysis on Riemmanian symmetric spaces and horocycle spaces, respectively.  An overview of this theory from the perspective of unitary group representations may be found in \cite{OS}.

\subsection{Basic Definitions}
\label{basicStructure}
\label{definitionsSection}
Suppose that $G$ is a semisimple Lie group over $\R$ with finite center and that $K$ is a closed subgroup.  Furthermore, we suppose that there is an involutive automorphism $\theta: G\rightarrow G$ such that
\begin{equation}
\label{symmetricSpaceCondition}
   (G^\theta)_0 \leq K \leq G^\theta,
\end{equation}
where $G^\theta$ is the fixed-point subgroup for $\theta$ and $(G^\theta)_0$ is the connected component of the identity for $G^\theta$.  Then $G/K$ is said to be a \textbf{symmetric space}.

The involution $\theta$ differentiates to an involution $\theta: \gg \rightarrow \gg$ of the Lie algebra $\gg$ of $G$.  By (\ref{symmetricSpaceCondition}), the $+1$-eigenspace for $\theta$ is just $\gk$ (i.e., the Lie algebra for $K$).  We denote the $-1$-eigenspace of $\theta$ by $\gp$.  We may write down the eigenspace decomposition
\[
 \gg = \gk \oplus \gp.
\]
Due to the fact that $\theta$ is also a Lie algebra involution, one easily computes that $[\gk,\gk] \subseteq \gk$, $[\gk,\gp]\subseteq \gp$, and $[\gp,\gp]\subseteq \gk$.  Just as $\gk$ may be naturally identified with the tangent space $T_e K$, there is a natural identification of $\gp$ with the tangent space $T_{eK} G/K$ (see \cite[p. 214]{He1}). 

If $U/K$ is a Riemannian symmetric space with $U$ compact, then it is said to be a \textbf{compact-type Riemannian symmetric space}.  On the other hand, if $G/K$ is a Riemannian symmetric space such that $G$ has no compact factors and has a finite center, then $K$ is compact and $G/K$ is said to be a \textbf{noncompact-type Riemmanian symmetric space}.  

There is a beautiful duality between compact-type and noncompact-type Riemannian symmetric spaces.  Suppose that $U/K$ is a compact-type symmetric space with involution $\theta$.  We make the further simplifying assumption throughout this paper that $U$ is simply-connected whenever we consider a compact-type symmetric space $U/K$.  In particular, this assumption implies that $K = U^\theta$ is connected.  As usual, $\theta$ admits an eigenspace decomposition:
\[
  \gu = \gk \oplus \widetilde{\gp},
\]
where $\gk=\Lie(K)$ and $\widetilde{\gp}$ are the $+1$ and $-1$ eigenspaces, respectively.  We now consider the complexified Lie algebra $\gu_\C = \gu \otimes_\R \C$ and define a new real Lie algebra $\gg$:
\[
   \gg = \gk \oplus i\widetilde{\gp} 
\]  
Next we take the unique connected complex Lie group $U_\C$ with Lie algebra $\gu_\C$ such that $U$ is the analytic subgroup of $U_\C$ corresponding to the Lie algebra $\gg\subseteq\gu_\C$. The Lie algebra involution $\theta$ on $\gu_\C$ integrates to an involution on $U_\C$ by Proposition 7.5 in \cite{Knapp}. We then consider the analytic subgroup $G\leq U_\C$  corresponding to the Lie algebra $\gg\subseteq\gu_\C$.  By Proposition 7.9 in \cite{Knapp}, we see that $G$ is a closed subgroup of $U_\C$ and has a finite center.  Putting everything together, we see that $G/K$ is a \textit{noncompact-type} Riemannian symmetric space, called the \textbf{c-dual} of $U/K$. To emphasize the symmetry between $\gu$ and $\gg$, we note that if we set $\gp=i\widetilde{\gp}$, then
\begin{align*}
\gu & =  \gk\oplus i\gp \\
\gg & =  \gk\oplus\gp.
\end{align*}

We need some well-known definitions from symmetric space theory.  Fix a noncompact-type symmetric space $G/K$.  As before, we can write $\gg = \gk \oplus \gp$.  Let $\ga$ be a maximal abelian subalgebra of $\gp$.  
For each $\alpha\in \ga^*$, let
\[
\gg_\alpha = \{X\in\gg|[H,X] = \alpha(H)X \}.
\]
We write $\Sigma = \Sigma(\gg,\ga) = \{\alpha \in \ga^*\backslash\{0\} | \gg_\alpha \neq 0\}$ for the set of all restricted roots for $(\gg,\gk)$.  We can choose a positive subsystem $\Sigma^+$ of $\Sigma$.  We denote the set of indivisible roots by $\Sigma_0$, and put $\Sigma_0^+ = \Sigma_0\cap \Sigma^+$. 

Next, we define $\gm = Z_\gk(\ga)$ and note that $\gg_0 = \gm \oplus \ga$ by the maximality of $\ga$ in $\gp$.  Thus, if we put 
\begin{align*}
&\gn = \bigoplus_{\alpha\in\Sigma^+} \gg_\alpha 
\text{ and } \overline{\gn} = \bigoplus_{\alpha\in\Sigma^+} \gg_{-\alpha} \text{, then } \\
&\gg = \gn \oplus \gm \oplus \ga \oplus \overline{\gn}.
\end{align*}
Finally, we set $N = \exp(\gn)$ and $M = Z_K(\ga)$.  Note that although $\gm$ is the Lie algebra of $M$, one generally does not have $M = \exp(\gm)$ because $M$ is typically disconnected.

Before proceeding, we lay down some standard notation which will be useful in the future.  Suppose that $\dim \ga = n$.  If $\Sigma(\gg,\ga)$ is a root system of type $A_n$, then we identify $\ga$ with $\R^n$ by setting 
$\ga=\{(x_1,\ldots ,x_{n+1})\in \R^{n+1}\mid x_1+\cdots + x_{n+1} =0\}$.  Otherwise we make the identification $\ga=\R^n$. Set
$e_1=(1,\ldots ,0,0)$, $e_2=(0,1,0,\ldots, 0)$, \ldots ,
$e_r=(0,\ldots , 0,1)$ where $r=n+1$ for $A_n$ and otherwise $n=r$. We view the vectors $e_j$  also
as elements in $\ga^*$ via the standard inner product in $\R^{n+1}$ in the case that  $\Sigma(\gg,\ga)$  is of type $A_n$ and otherwise via the standard inner product in $\R^n$. Note that in the $A_n$-type case this defines a map $\R^{n+1} \to \ga^*$ which is not injective. 

For the purpose of convenience, we choose an ordering on the vectors in $\ga^*$ (that is, a choice of $\Sigma^+$) such that the dominant weights have \textit{increasing} coefficients with respect to the ordered basis $\{e_1,\ldots e_n\}$.  Note that this is the opposite of the conventional choice. That is, we make the following choices for $\Sigma_0^+(\gg,\ga)$ according to the Dynkin diagram $\Psi$ of the root system $\Sigma(\gg, \ga)$:

\begin{center}
\begin{tabular}{|l|l|}
  \hline 
 $\Psi$ & $\Sigma_0^+(\gg,\ga)$ \\
\hline
\hline 
$ A_n$ & $\{ e_j - e_i \mid i < j\}$\\
\hline  
 $B_n$ & $\{ e_j \pm e_i \mid i < j\} \cup \{ e_i \}$\\
\hline 
$ C_n$ & $\{ e_j \pm e_i \mid i < j\} \cup \{ 2e_i \}$\\
\hline 
 $D_n$ & $\{ e_j \pm e_i \mid i < j\}$\\
 \hline
\end{tabular}
\end{center}

\subsection{Finite-Dimensional Conical and Spherical Representations}
In this section, we review the basic results on finite-dimensional conical and spherical representations.  We begin with their definitions:
\begin{definition}
Suppose that $G/K$ is a Riemannian symmetric space (of either compact or noncompact type).  A representation $(\pi,\cH)$ of $\pi$ on a Hilbert space $\cH$ is \textbf{spherical} if there is $e\in\cH\backslash\{0\}$ such that $\pi(K)e = e$, in which case $e$ is said to be a \textbf{spherical vector} for $\pi$.
\end{definition}

\begin{definition}
Suppose that $G/K$ is a Riemannian symmetric space of noncompact type.  Using the notation of Section~\ref{definitionsSection}, we say that a representation $(\pi,\cH)$ of $\pi$ on a Hilbert space $\cH$ is \textbf{conical} if there is $v\in\cH\backslash\{0\}$ such that $\pi(MN)e = e$.
\end{definition}
In harmonic analysis on the homogeneous space $G/MN$ (called the \textbf{horocycle space}), conical representations play a roll which is analogous to the roll played by spherical representations for harmonic analysis on the symmetric space $G/K$.

Note that if $U/K$ is the compact c-dual of $G/K$, then every finite-dimensional representation of $U$ extends by holomorphic continuation to a representation of $G$.  By abuse of notation, we can thus consider a finite-dimensional representation $(\pi,V)$ of $U$ to be also a representation of $G$, and vice versa.  Thus, a finite-dimensional representation of $U$ is said to be conical if the corresponding representation of $G$ is conical.

The problem of determining which finite-dimensional representations are spherical or conical is solved by the Cartan-Helgason Theorem on finite-dimensional spherical representations, which was first stated without proof in \cite{Sug}.

\begin{comment} Similarly, we see that $\gh_c = \gh \oplus i\ga$ is a $\theta$-stable Cartan subalgebra of $\gu$.
\end{comment}

\begin{theorem}[The Cartan-Helgason Theorem]
\label{CartanHelgason}
 (\cite[p. 535]{He2})
Suppose that $U/K$ is a compact-type symmetric space with c-dual $G/K$, and let $\gh$ be a Cartan subalgebra of $\gu$ containing $i\ga$, so that $\gh$ is $\theta$-stable and $\gh=\gt\oplus i\ga$ where $\gt=\gh\cap \gk$. Suppose further that $U$ is simply-connected. If $(\pi,V)$ is an irreducible representation of $U$ with highest-weight $\lambda\in i\gh^*$, then the following are equivalent:
\begin{enumerate}
\item $\pi$ is a spherical representation of $U$
\item $\pi(M)v = v$ for each highest-weight vector $v\in V_\lambda$.
\item $\lambda(\gt) = 0$ and also
\[
  \frac{\langle \lambda,\alpha \rangle}{\langle\alpha,\alpha\rangle} \in \Z^+ \text{ for all } \alpha\in \Sigma^+
\]
\end{enumerate}
\end{theorem}

If $\pi$ is spherical, we say by abuse of notation that $\lambda|_{i\ga}$ is  the \textbf{highest weight} of $\pi$.  Note that there is a natural identification of purely imaginary weights on $i\ga$ with purely real weights on $\ga$.  Thus, the highest restricted roots may be identified with elements of  $\ga^*$.  We set
\[
\Lambda^+ \equiv \Lambda^+(\gg,\ga) \equiv \left\{ \mu\in\ga^* \left| \frac{\langle \lambda,\alpha \rangle}{\langle\alpha,\alpha\rangle} \in \Z^+ \text{ for all } \alpha\in \Sigma^+ \right.\right\}
\]
and write $\pi_\lambda$ for the unique spherical representation with highest weight $\lambda\in\Lambda^+$.

Moreover, $\Lambda^+$ is a semilattice. In fact, define linear functionals  $\xi_{j}\in \ga^*$ by
\begin{equation}\label{fundclass1}
\frac{ \langle \xi_{i},\alpha_{j} \rangle }
{\langle \alpha_{j},\alpha_{j} \rangle} = \delta_{i,j}  \text{ for }
1 \leq j \leq r\ \ ,
\end{equation}
where $\Psi = \{\alpha_1,\ldots,\alpha_r\} \subseteq \Sigma_0^+$ is the set of positive simple roots.
Then $\xi_1,\ldots ,\xi_r\in\Lambda^+$ and
\[\Lambda^+=\Z^+\xi_1+\cdots  + \Z^+\xi_r=\left\{\left. \sum_{j=1}^r n_j\xi_j\bmid n_j\in \Z^+\right\}\, .\]
The weights $\xi_j$ are called the
\textit{fundamental weights} for
$(\gg,\ga)$.  Note that each element of $\Lambda^+$ corresponds to a unique irreducible spherical representation of $U$.

The Cartan-Helgason Theorem also gives a classification of conical representations of $G$ by the following theorem of Helgason:
\begin{theorem}(\cite[p. 79]{He1970})
Suppose that $(\pi,V)$ is an irreducible finite-dimen\-sional representation of $G$.  Then $\pi$ is spherical if and only if it is conical, in which case $V^{MN}$ is one-dimensional and consists of the highest-weight vectors of $\pi$.
\end{theorem}

 Now that the irreducible finite-dimensional spherical and conical representations have been parameterized, one may ask more generally about finite-dimensional spherical and conical representations that may not be irreducible.
 
To that end, suppose that $(\pi_\mu, \cH_\mu)$ is an irreducible $K$-spherical representation of $G$ with highest weight $\mu$ and that $(\sigma, \cH)$  is a unitary primary representation of $G$ consisting of representations of type $\mu$.  By \cite[Lemma 1.5]{GM}, all cyclic primary representations of a compact group are finite-dimensional, and hence $\sigma$ extends uniquely to a holomorphic spherical representation of $G^\C$.  Because it is a finite-dimensional spherical representation, $\sigma$ is automatically a conical representation of $G^\C$.  In fact, as the following result shows, the $MN$-invariant vectors of $\sigma$ are precisely the highest-weight vectors of irreducible subrepresentations of $\sigma$. The lemma is likely known by specialists, but the authors were not able to find an exact citation in the literature. 

\begin{lemma}
\label{irreducible_primary}    
Suppose, as above, that $(\sigma, \cH)$  is a finite-dimensional, unitary primary representation of $U$ consisting of representations with highest weight $\mu$.   If $v\in \cH^{MN}\backslash\{0\}$, then $v$ is a highest-weight vector that generates an irreducible spherical representation of $U$.  Furthermore, if $v,w\in \cH^{MN}\backslash\{0\}$ and $v\perp w$, then $\langle\pi(U)v\rangle \perp \langle \pi(U)w\rangle$.
\end{lemma}
\begin{proof}
Let $v\in \cH^{MN}\backslash\{0\}$, and consider $W = \langle \sigma(G) v \rangle$.  We can write $W = W_1 \bigoplus \cdots \bigoplus W_n$ where each $W_i$ gives an irreducible representation of $G$ that is equivalent to $\cH_\mu$.  It must be a finite direct sum because all cyclic primary representations of compact groups are finite-dimensional (see \cite{GM}).  For each $i$, let $v_i$ be the orthogonal projection of $v$ onto $W_i$.  Then $v = v_1 + \cdots + v_n$.  Since each $W_i$ is a $G$-invariant subspace, it follows that each vector $v_i$ is also invariant under $MN$.  Because $W_i$ is irreducible, we see that $v_i$ must be a (nonzero) highest-weight vector of weight $\mu$ (see \cite[Theorem 12.3.13]{GW}).  Hence $v$ is a weight vector of weight $\mu$. 

   Suppose that $W$ is not irreducible (that is, $n >1$).  Because $W$ is cyclic, there must be $g_1, \ldots, g_k \in G$ and $c_1, \ldots, c_k \in \C$ such that $\sum_{i=0}^k c_i \pi(g_i) v =  v_1$ (it is sufficient to consider finite linear combinations because $W$ is finite-dimensional).  It follows from the invariance of each space $W_k$ that $\sum_{i=0}^k c_i \pi(g_i) v_1 =  v_1$ and $\sum_{i=0}^k c_i \pi(g_i) v_2 = 0$.  Because $W_1$ and $W_2$ give equivalent representations of $G$ and all highest-weight vectors of an irreducible representation are constant multiples of each other, this is a contradiction.  Thus $W$ is irreducible and $v = v_1$ is a highest-weight vector for $W$.

  Now suppose that $v$ and $w$ are nonzero $MN$-invariant vectors in $\cH$ such that $v\perp w$.  Write $V = \langle \pi(G)v\rangle$ and $W = \langle \pi(G)w\rangle$.  By the above, we know that $V$ and $W$ are irreducible representations of $G$ with highest-weight vectors $v$ and $w$, respectively.  Hence, either $V\cap W = \{0\}$ or $V=W$.  Because the space of highest-weight vectors of an irreducible representation of $G$ is one dimensional and $v\perp w$, we cannot have $V=W$.  Thus $V\cap W=\{0\}$.
  
  Now consider the invariant subspace $Z=V+W$ and the corresponding orthogonal projection $p:Z\rightarrow W$, which is an intertwining operator for $\pi$ because $W$ is an invariant subspace of $Z$.  Hence, $p(v)\in \cH^{MN}$ and so $p(v) = cw$ for some $c\in\C$.  Since $v\perp w$, we see that $c=0$ and thus $v\in \ker p$.  Moreover, it is clear that $\ker p$ is a $U$-invariant subspace of $Z$, so it follows that $V=\langle \pi(U)v\rangle \subseteq \ker p$.  Hence $V\perp W$ as we wished to show.
\end{proof}

We end this section with another useful lemma, which we will make use of several times.  Once again, it is probably already known, but the authors were unable to find an exact citation in the literature.
\begin{lemma}
\label{splittinglemma}
%\marginpar{I haven't yet found a citation for this lemma.}
Let $G$ be a topological group and let $(\pi, \cH)$ be a unitary representation of $G$.  Let $\cA$ be a finite or countably infinite index set, and suppose that 
\[
v = \sum_{i\in\cA} v_i,
\] 
where $v_i\in\cH$ for each $i\in\cA$ and where $\langle \pi(G) v_i \rangle$ and $\langle \pi(G) v_j \rangle$ give mutually distinct irreducible representations of $G$ for $i\neq j$.  Then 
\[
\langle \pi(G) v \rangle = \bigoplus_{i\in\cA} \langle \pi(G) v_i \rangle .
\]
\end{lemma}
\begin{proof}
Write $V = \langle \pi(G) v \rangle$.  The fact that $V_i = \langle \pi(G) v_i \rangle$ and $V_j = \langle \pi(G) v_j \rangle$ give disjoint representations of $G$ for $i\neq j$ implies that $V_i \perp V_j$.  It is obvious that 
\[
\langle \pi(G) v \rangle \subseteq \bigoplus_{i\in\cA} \langle \pi(G) v_i \rangle,
\]  
so we prove the opposite containment.  It suffices to show that $v_i\in V$ for all~$i\in\cA$.

Suppose that $v_i \notin V$ for some $i\in\cA$.  Define  
\[
w = \sum_{j\neq i} v_j \text{ and }
W = \langle \pi(G) w \rangle \subseteq \bigoplus_{j\neq i} V_j.
\]
Then $V_i \perp W$  and $v = v_i + w$.  Furthermore, $V_i$ and $W$ give disjoint representations of $G$.  

 Now let $c_1, \ldots c_k \in\C$ and $g_1,\ldots g_k\in G$.  Then 
\[
\sum_{j=1}^k c_j \pi(g_j) v = \left(\sum_{j=1}^k c_j \pi(g_j) v_i\right) + \left(\sum_{j=1}^k c_j \pi(g_j) w\right).
\]
Because $v_i\notin V$ and $V_i$ is irreducible, we see that $V\cap V_i = \emptyset$.  It follows that 
\[
\sum_{j=1}^k c_j \pi(g_j) v_i = 0 \textrm{ if and only if } \sum_{j=1}^k c_j \pi(g_j) w = 0.
\]
Hence there is a well-defined, nonzero intertwining operator $L:V_i\rightarrow W$ such that $L(v_i) = w$, which contradicts the fact that $V_i$ and $W$ give disjoint representations of $G$.
\end{proof}

\section{Direct Limits of Groups and Symmetric Spaces}
%Motivated in part by applications to physics, there has been an increasing amount of work done on infinite-dimensional Lie groups since the 1970s. These are topological groups which are locally modeled on locally convex topological vector spaces over $\R$ (in the same way that finite-dimensional Lie groups are modeled on finite-dimensional vector spaces over $\R$). The simplest infinite-dimensional Lie groups which may be considered are those which are formed by taking direct limits of finite-dimensional Lie groups.  They occupy a sort of ``middle ground'' between finite-dimensional groups and other infinite-dimensional groups with finer topologies, in that they inherit many of the properties of the former but already exhibit some of the pathologies of the latter.

We refer the reader to \cite{DPW} and  \cite{NRW3} for a good overview of the basic properties of direct-limit groups.  See  \cite{NRW1} and \cite{KHN} for some details about the construction of smooth manifold structures on direct-limit groups.  See also \cite{Sa} for an in-depth study of direct limits of abelian and nilpotent groups with applications to physics.

\subsection{Direct Systems of Riemannian Symmetric Spaces}
\label{basicLimRiemannian}
Suppose that $\{G_n / K_n\}_{n\in\N}$ is a sequence of semisimple Riemannian symmetric spaces such that $G_n$ is a closed subgroup of $G_m$ for $n\leq m$ and such that $K_m \cap G_n = K_n$ for $n\leq m$. We label the corresponding involutions by $\theta_n : G_n\rightarrow G_n$ and make the further assumption that $\theta_m|_{G_n} = \theta_n$ for all $n\leq m$. We thus obtain a direct system of homogeneous spaces $\{G_n/K_n\}_{n\in\N}$.  Now construct the direct limits $G_\infty = \varinjlim G_n$, $K_\infty = \varinjlim K_n$, and $G_\infty/K_\infty = \varinjlim G_n/K_n$.  Finally, one can show that $K_\infty = (G_\infty)^{\theta_\infty}$.

We say that $G_\infty/K_\infty$ is a \textbf{lim-Riemannian symmetric space}.  If $G_n/K_n$ is a noncompact-type Riemannian symmetric space for all $n\in\N$, then $G_\infty/K_\infty$ is said to be a \textbf{lim-noncompact Riemannian symmetric space}.  Similarly, if $G_n/K_n$ is a compact-type symmetric space for each $n\in\N$, then $G_\infty/K_\infty$ is said to be a \textbf{lim-compact Riemmanian symmetric space}.  In this paper, all lim-compact Riemannian symmetric spaces will be limits of simply-connected spaces.

We now review how the notion of c-duals may be extended to lim-Riemann\-ian symmetric spaces.  Suppose that $\{U_n/K_n\}_{n\in\N}$ is a direct system of compact-type Riemannian symmetric spaces with involutions $\theta_n:U_n\rightarrow U_n$. One may construct a complexification $(U_\infty)_\C = \varinjlim (U_n)_\C$ for the lim-compact group $U_\infty = \varinjlim U_n$.  To simplify notation we assume that  $(U_n)_\C \subseteq (U_{n+1})_\C$ and therefore $U_n\subseteq U_{n+1}$ for each $n\in\N$.  The involutions $\theta_n:U_n\rightarrow U_n$ extend to holomorphic involutions $\theta_n : (U_n)_\C \rightarrow (U_n)_\C$ such that $\theta_m|_{U_n} = \theta_n$ for $m\leq n$.

We write 
\[
   \gu_n = \gk_n \oplus \widetilde{\gp}_n
\]
for each $n\in\N$, where $\gk_n$ and $\widetilde{\gp}_n$ are the $+1$- and $-1$-eigenspaces of $\theta_n$.  It follows from the fact that $\theta_{n+1}|_{U_n} = \theta_n$ that
\[
\gk_n=\gk_{n+1}\cap \gu_n \text{ and } \widetilde{\gp}_n=\widetilde{\gp}_{n+1}\cap \gu_n
\]
and hence that 
$\gk_n\subseteq \gk_{n+1}$ and $\widetilde{\gp}_n\subseteq\widetilde{\gp}_{n+1}$. 
For each $n$, we construct the c-dual Lie algebra
\[
    \gg_n = \gk_n \oplus i \widetilde{\gp}_n \subseteq (\gu_n)_\C
\]
and note that $\gg_n\subseteq\gg_{n+1}$. Finally, we construct the analytic subgroup $G_n$ of $(U_n)_\C$ which corresponds to the Lie algebra $\gg_n$ and recall that $G_n$ is closed in $(U_n)_\C$. Thus $G_n$ is a closed subgroup of $G_{n+1}$ for each $n$.  It follows that the direct-limit group $G_\infty = \varinjlim G_n$ is a closed subgroup of $(U_n)_\C$ and possesses the direct-limit Lie algebra $\gg_\infty = \varinjlim \gg_n$.

Reviewing the construction of finite-dimensional c-dual spaces, we see that the complexified involution $\theta_n:(\gu_n)_\C \rightarrow (\gu_n)_\C$ restricts to an involution $\theta_n : \gg_n\rightarrow \gg_n$ and that $\gk_n$ and $i\widetilde{\gp}_n$ are the $+1$- and $-1$-eigenspaces of $\theta_n$ in $\gg_n$.  Furthermore, because $\gg_n$ is $\theta_n$-stable, the holomorphic involution $\theta_n:(U_n)_\C \rightarrow (U_n)_\C$ restricts to an involution $\theta_n: G_n\rightarrow G_n$ such that $(G_n)^{\theta_n} = K_n$.  Thus $\{G_n/K_n\}_{n\in\N}$ is a direct system of noncompact-type Riemannian symmetric spaces.  We say that $G_\infty/K_\infty = \varinjlim G_n/K_n$ is the \textbf{c-dual} of $U_\infty/K_\infty$.

We set $\gp_n = i\widetilde{\gp}_n$ for each $n$, so that
\begin{align*}
\gg_n & =  \gk_n\oplus\gp_n \\
\gu_n & =  \gk_n\oplus i\gp_n.
\end{align*}
Finally, we notice that
\begin{align*}
\gg_\infty & =  \gk_\infty\oplus\gp_\infty \\
\gu_\infty & =  \gk_\infty\oplus i\gp_\infty,
\end{align*}
where $\gk_\infty = \varinjlim \gk_n$ and $\gp_\infty = \varinjlim \gp_n$ are the $+1$-and $-1$-eigenspaces of $\theta_\infty$ in $\gg_\infty$.

\subsection{Propagated Direct Limits}

\label{propagationSection}
%In this section we give a short overview of injective limits and propagation
%of compact symmetric spaces, as needed for our considerations on
%conical representations. We refer
%to \cite{OW2013} and \cite{W2010} for details.
%\medskip
As before, we assume that $G_\infty/K_\infty$ is a lim-noncompact Riemannian symmetric space which is the c-dual of a direct limit $U_\infty/K_\infty$ of simply-connected compact Riemannian symmetric spaces.  We need to put some further technical conditions on $G_\infty/K_\infty$ in order to prove our results about conical representations.  The first condition is that of \textit{propagation}, which was introduced in \cite{OW2013, W2011,W2010}.

We begin this section by examining the restricted root data of $G_\infty/K_\infty$, using the notation of Section~\ref{basicLimRiemannian}.
We recursively choose maximal commutative subspaces $\ga_k\subseteq \gp_k$ such
that $\ga_{n} \subseteq \ga_k$ for $n\leq k$ and define $\ga_\infty = \varinjlim \ga_n$.  Note that $\ga_\infty^* \equiv \varprojlim \ga_n^*$, where the projective limit is given by projections  $p_n: \ga_{n+1}^*\rightarrow \ga_n^*$ defined by $p_n:\alpha\mapsto \alpha|_{\ga_n}$.
 We then obtain the restricted root system $\Sigma_n = \Sigma(\gg_n,\ga_n)$ for each $n\in\N$.  Note that 
$
\Sigma_n\subseteq \Sigma_{k}|_{\ga_n}\backslash \{0\}
$ 
whenever $n\leq k$. 

%In other words, $\{\Sigma_n\subseteq \ga^*\}_{n\in\N}$ forms a \textit{projective system} of weights, where the projection maps $p^{n+1}_n:\Sigma_{n+1}\rightarrow\Sigma_n$ are given by restriction: $p^{n+1}_n(\lambda) = \lambda|_{\ga_n}$. Due to the fact that
% \[
%   (\ga_\infty)^* \cong \varprojlim (\ga_n^*),
%\]
%we see that $\varprojlim \Sigma_n \subseteq (\ga_\infty)^*$.  We say that $\Sigma_\infty = \varprojlim \Sigma_n$ is the \textbf{restricted root system} of $(\gg_\infty,\ga_\infty)$.

 Next, we choose positive subsystems $\Sigma_n^+\subseteq \Sigma_n$ so that
$\Sigma_n^+\subseteq \Sigma_k^+|_{\ga_n}\backslash \{0\}$. 
The projective limit $\Sigma_\infty^+=\varprojlim \Sigma_n^+$ plays the role of the positive root subsystem for $(\gg_\infty,\ga_\infty)$.

%Furthermore, we can consider for each $k\in\N$ the Cartan subalgebra $\gh_k = \gt_k \oplus \gt_k$ in $\gg_k$, where $\gt_k$ is a Cartan subalgebra of $\gm_k = Z_{\gk_k}(\ga_k)$.

For each $n\in\N$, we let $(\Sigma_n)_0$ denote the set of nonmultipliable roots in $\Sigma_n$ and set $(\Sigma_n)_0^+ = (\Sigma_n)_0 \cap \Sigma_n^+$. 
Denote  the set of simple roots in $(\Sigma_n)_0^+$
by $\Psi_n =\{\alpha_1,\ldots ,\alpha_{r_n}\}$, where $r_n=\dim \ga_n$. Since we will be dealing with direct limits we may assume that $\Sigma$, and hence $\Sigma_0$, is one of the classical root systems. We number the simple roots in the following way:

\begin{equation}\label{rootorder}
\begin{aligned}
&\begin{tabular}{|c|l|c|}\hline
$\Psi=A_r$&
\setlength{\unitlength}{.5 mm}
\begin{picture}(155,18)
\put(5,2){\circle{4}}
\put(2,5){$\alpha_{r}$}
\put(6,2){\line(1,0){13}}
\put(24,2){\circle*{1}}
\put(27,2){\circle*{1}}
\put(30,2){\circle*{1}}
\put(34,2){\line(1,0){13}}
\put(48,2){\circle{4}}
\put(49,2){\line(1,0){23}}
\put(73,2){\circle{4}}
\put(74,2){\line(1,0){23}}
\put(98,2){\circle{4}}
\put(99,2){\line(1,0){13}}
\put(117,2){\circle*{1}}
\put(120,2){\circle*{1}}
\put(123,2){\circle*{1}}
\put(129,2){\line(1,0){13}}
\put(143,2){\circle{4}}
\put(140,5){$\alpha_1$}
\end{picture}
&$r\geqq 1$
\\
\hline
%\end{tabular}\\
%&\begin{tabular}{|c|l|c|}\hline
$\Psi =B_r$&
\setlength{\unitlength}{.5 mm}
\begin{picture}(155,18)
\put(5,2){\circle{4}}
\put(2,5){$\alpha_{r}$}
\put(6,2){\line(1,0){13}}
\put(24,2){\circle*{1}}
\put(27,2){\circle*{1}}
\put(30,2){\circle*{1}}
\put(34,2){\line(1,0){13}}
\put(48,2){\circle{4}}
\put(49,2){\line(1,0){23}}
\put(73,2){\circle{4}}
\put(74,2){\line(1,0){13}}
\put(93,2){\circle*{1}}
\put(96,2){\circle*{1}}
\put(99,2){\circle*{1}}
\put(104,2){\line(1,0){13}}
\put(118,2){\circle{4}}
\put(115,5){$\alpha_2$}
\put(119,2.5){\line(1,0){23}}
\put(119,1.5){\line(1,0){23}}
\put(143,2){\circle*{4}}
\put(140,5){$\alpha_1$}
\end{picture}
&$r\geqq 2$\\
\hline
%\end{tabular} \\
%&\begin{tabular}{|c|l|c|}\hline
$\Psi=C_r$ &
\setlength{\unitlength}{.5 mm}
\begin{picture}(155,18)
\put(5,2){\circle*{4}}
\put(2,5){$\alpha_{r}$}
\put(6,2){\line(1,0){13}}
\put(24,2){\circle*{1}}
\put(27,2){\circle*{1}}
\put(30,2){\circle*{1}}
\put(34,2){\line(1,0){13}}
\put(48,2){\circle*{4}}
\put(49,2){\line(1,0){23}}
\put(73,2){\circle*{4}}
\put(74,2){\line(1,0){13}}
\put(93,2){\circle*{1}}
\put(96,2){\circle*{1}}
\put(99,2){\circle*{1}}
\put(104,2){\line(1,0){13}}
\put(118,2){\circle*{4}}
\put(115,5){$\alpha_2$}
\put(119,2.5){\line(1,0){23}}
\put(119,1.5){\line(1,0){23}}
\put(143,2){\circle{4}}
\put(140,5){$\alpha_1$}
\end{picture}
& $r\geqq 3$
\\
\hline
%\\
%\end{tabular}\\
%&\begin{tabular}{|c|l|c|}\hline
$\Psi=D_r$ &
\setlength{\unitlength}{.5 mm}
\begin{picture}(155,20)
\put(5,9){\circle{4}}
\put(2,12){$\alpha_{r}$}
\put(6,9){\line(1,0){13}}
\put(24,9){\circle*{1}}
\put(27,9){\circle*{1}}
\put(30,9){\circle*{1}}
\put(34,9){\line(1,0){13}}
\put(48,9){\circle{4}}
\put(49,9){\line(1,0){23}}
\put(73,9){\circle{4}}
\put(74,9){\line(1,0){13}}
\put(93,9){\circle*{1}}
\put(96,9){\circle*{1}}
\put(99,9){\circle*{1}}
\put(104,9){\line(1,0){13}}
\put(118,9){\circle{4}}
\put(113,12){$\alpha_3$}
\put(119,8.5){\line(2,-1){13}}
\put(133,2){\circle{4}}
\put(136,0){$\alpha_1$}
\put(119,9.5){\line(2,1){13}}
\put(133,16){\circle{4}}
\put(136,14){$\alpha_2$}
\end{picture}
& $r\geqq 4$
\\
\hline
\end{tabular}
\end{aligned}
\end{equation}

We are now ready to introduce the definition of propagated direct-limits of symmetric spaces.

\begin{definition}
We say that a lim-noncompact symmetric space $G_\infty/K_\infty$ is \textbf{propagated} if
\begin{enumerate}
 \item For each simple root $\alpha\in\Psi_k$ there is a unique simple root $\widetilde{\alpha}\in\Psi_n$ such that $\widetilde{\alpha}|_{\ga_k} = \alpha$, whenever $k\leq n$.
 \item There is a choice of ordering on the roots in $\Psi_k$ for each $k\in\N$ such that either $\ga_n = \ga_k$ or else $\Psi_k$ extends $\Psi_n$ for $n\leq k$ only by adding simple roots at the left end. (In particular, each $\Psi_k$ has the same Dynkin diagram type.)
 \end{enumerate}
\end{definition}

Let $U_\infty = \varinjlim U_n$ be a direct limit of compact Lie groups.  Each group $U_n$ may be considered as a compact symmetric space $U_n \equiv U_n\times U_n / K_n$, where $K_n := \{(g,g)|g\in U_n\} \cong U_n$ is the fixed-point subgroup of the involution $\theta : (g,h)\mapsto (h,g)$ on $U_n\times U_n$.  In this way, $(U_\infty \times U_\infty) / U_\infty \equiv U_\infty$ becomes a direct limit of symmetric spaces.

\begin{comment}Choose a Cartan subalgebra $\gh_n\subseteq \ug_n$ for each $n$ in such a way that   $\gh_n \subseteq \gh_k$ whenever $n\leq k$.  \marginpar{is the definition of the $\gh_i$'s clear?} One then obtains a root system $\Delta_n = \Delta(\gg_n, \gh_n)$ for each $n$. After recursively choosing positive subsystems  $\Delta_n^+\subseteq \Delta_n$ such that
\[
\Delta_n^+\subseteq \Delta_k^+|_{\gh_n}\backslash \{0\},
\]
for $n \leq k$, we arrive at a set $\Xi_n$ of simple roots in $\Delta_n^+$.
 We order these simple roots the same way as in Table~\ref{rootorder}. 
\end{comment}

\begin{definition}
 We say that the lim-compact group $U_\infty$ is \textbf{propagated} if the associated symmetric space $(U_\infty \times U_\infty) / U_\infty$ is propagated.
% \begin{comment}
%\begin{enumerate}
% \item For each simple root $\alpha\in\Xi_k$ there is a unique simple root $\widetilde{\alpha}\in\Xi_n$ such that $\widetilde{\alpha}|_{\gh_k} = \alpha$, whenever $k\leq n$.
% \item There is a choice of ordering on the roots in $\Xi_k$ for each $k\in\N$ such that either $\gh_n = \gh_k$ or else $\Xi_k$ extends $\Xi_n$ for $n\leq k$ only by adding simple roots at the left end.  
% \end{enumerate}  
% \end{comment}
\end{definition}

%\marginpar{Is this paragraph clear enough?} We note that this definition of propagation of a lim-compact group $U_\infty$ agrees with the notion of the propagation of $U_\infty$ as a lim-compact symmetric space $(U_\infty \times U_\infty) / U_\infty$.  

Suppose that $U_\infty$ is a propagated direct limit of compact, simply-connected semisimple Lie groups.  Then each $U_k$ may be decomposed into a product of compact simple Lie groups, say $U_k = U_k^1\times U_k^2\times \cdots \times U_k^{d_k}$.  We can recursively choose Cartan subalgebras $\gh_k = \gh_k^1\oplus \gh_k^2\oplus\cdots\oplus \gh_k^{d_k}$ where each $\gh_k^i$ is a Cartan subalgebra of $\gu_k^i$. The definition of propagation then implies that $d_n = d_m \equiv d$ for each $n,m\in\N$ and that the indices may be ordered in such a way that $\{U_k^i\}_{n\in\N}$ is a propagated direct system of compact simple Lie groups for each $1\leq i\leq d$.

\subsection{Admissible Direct Limits}
We continue to examine the root data for lim-noncompact symmetric spaces $G_\infty/K_\infty$.
 For each $k\in\N$ and each restricted root $\alpha\in\Sigma_k$, we define as before the root space
\[
\gg_{k,\alpha} =\{Y\in\gg_k \mid [H,Y]=\alpha (H)Y \text{ for all }
H\in \ga_k\}\, .
\]
Next we define the subalgebras
 \[
   \gn_k = \bigoplus_{\alpha\in \Sigma_k^+} \gg_{k,\alpha}
\]
and
\[\gm_k = Z_{\gk_k}(\ga_k)\]
of $\gg_k$. 
Similarly, we define the subgroups $N_k = \exp(\gn_k)$ and $M_k = Z_{K_k}(\ga_k)$ of $G_k$.

For each $k\in K$, the conical representations of $G_k$ are the representations which possess a nonzero vector (or, more generally, distribution vector) which is invariant under the action of the group $M_k N_k$.  Hence, in order to define conical representations of $G_\infty$, one would like to define a subgroup $M_\infty N_\infty = \varinjlim M_n N_n$.  In order for such a group to be well-defined, we need to introduce a technical condition, which was first considered in \cite{HO}.

\begin{definition}
\label{admissibleDefinition}
A lim-noncompact symmetric space $G_\infty / K_\infty$ is said to be \textbf{admissible} if $M_k N_k \leq M_m N_m$ whenever $k\leq m$.  
\end{definition}

As a consequence of the following lemmas, it is sufficient to assume that $\gm_k \subseteq \gm_m$ for $k\leq m$:

\begin{lemma}
If $G_\infty / K_\infty$ is a lim-noncompact symmetric space, then $N_k\leq N_m$ for $k\leq m$.
\end{lemma}
\begin{proof}
We will show that $\gn_k \subseteq \gn_m$.  The result will then follow from the fact that $N_k = \exp \gn_k$ and $N_m = \exp \gn_m$.  

In fact, it suffices to show that $\gg_{k,\alpha} \subseteq \gn_m$ for all  $\alpha\in\Sigma_k^+$.  Suppose that $X\in\gg_{k,\alpha}$.  Consider the decomposition of $X$ into $\ga_m$-root vectors:
\[
   X = \sum_{\beta\in\Sigma_m} X_\beta,
\]
where $X_\beta \in \gg_{m,\beta}$ for each $(\gg_m,\ga_m)$-root $\beta$.  
Because this decomposition is unique and $X$ is a root vector for $\ga_k\subseteq\ga_m$, it follows that $\beta|_{\ga_k} = \alpha$ for all $\beta \in\Sigma_m$ such that $X_\beta \neq 0$.

%Hence, for any $A\in \ga_k \subseteq \ga_m$, we have
%\begin{align*}
%  \sum_{\beta\in\Sigma_m} \beta(A) X_\beta  & = [A,X] \\
%                                              & = \alpha(A) X\\
%                                              & = %\sum_{\beta\in\Sigma_m}\alpha(A) X_\beta
%\end{align*}
%It follows that either $X_\beta = 0$ or $\beta|_{\ga_k} = \alpha$ for all %$\beta\in\Sigma_m$.

Now recall that we have made a consistent choice of positive root subsystems $\Sigma_k^+$ of $\Sigma_k$ and $\Sigma_m^+$ of $\Sigma_m$.  In other words, $\beta\in\Sigma_m$ is positive if $\beta|_{\ga_k}$ is positive.  Since $\alpha\in\Sigma_k^+$ , it follows that $X$ is a sum of $\Sigma_m^+$-root vectors.  Hence, $X\in \gn_m$.
\end{proof}

Due to the fact that $M_k$ is typically a disconnected subgroup of $G_n$, it is not clear 
\textit{a priori} that requiring $\gm_k\subseteq\gm_m$ for $k\leq m$ is sufficient to imply that $M_k\leq M_m$.  However, the following lemma shows that this Lie algebra condition is, in fact, sufficient:
%However, it is still possible to reduce the assumption that $M_k\leq M_m$ to a condition on the Lie algebras, which will make it much easier to prove that the classical symmetric spaces are admissible.

%Suppose that $\gm_k\subseteq\gm_m$ for $k\leq m$.  We recursively choose maximal abelian subalgebras $\gt_k\subseteq\gm_k$ for each $k$ such that $\gt_k\subseteq \gt_m$ whenever $k\leq m$.  For each $k\in\N$ we can form a $\theta_k$-stable Cartan subalgebra of $\gg_k$ by setting $\gh_k = \gt_k \oplus \ga_k$.  Then $\gh_k\subseteq\gh_m$ for $k\leq m$.  

\begin{lemma}
Suppose that $G_\infty / K_\infty$ is a propagated lim-noncompact symmetric space such that $\gm_k\subseteq\gm_m$ for all $k\leq m$.  Then $M_k\leq M_m$ for $k\leq m$.
\end{lemma}
\begin{proof}
By \cite[Theorem 7.53]{Knapp} we see that for each $k\in\N$ there is a finite discrete subgroup $F_k\subseteq \exp(i \ga_k)\cap K_k$ such that $M_k = F_k (M_k)_0$, where $(M_k)_0 = \exp\gm_k$ is the connected component of the identity in $M_k$.   Because $\gm_k\subseteq \gm_m$ for all $k\leq m$, we see that $(M_k)_0 \leq (M_m)_0$.  It is thus sufficient to show that $F_k \leq M_m$ for $k\leq m$.  In fact, since $F_k\subseteq \exp(i\ga_k)\subseteq \exp(i\ga_m)$, it is clear that $F_k$ centralizes $\ga_k$.  Since $F_k\subseteq K_k \subseteq K_m$, we see that $F_k\leq M_m$, and the result follows.
\end{proof}

 It was not clear in \cite{HO} which lim-Riemannian symmetric spaces are admissible.  While we still do not know know whether it is possible to show that all propagated direct systems of Riemannian symmetric spaces are admissible in the sense of~\ref{admissibleDefinition}, we have been able to show that each of the classical direct limits (see Table \ref{classicalSymmetricPairs}) are admissible on a case-by-case basis.  The details of the proof are available in the appendix to this paper.

%%%%%%%%%%%%%%%%%%%%%%%%%%%%%%
%%%%%%%%%%%%%%%%%%%%%%%%%%%%%%
%%%%%%%%%%%%%%%%%%%%%%%%%%%%%%
%%%%%%%%%%%%%%%%%%%%%%%%%%%%%%
%%%%%%%%%%%%%%%%%%%%%%%%%%%%%%
%%%%%%%%%%%%%%%%%%%%%%%%%%%%%%
\section{Representations of Direct-Limit Groups}
In this section we review some important results about representations for direct-limit groups and lim-Riemannian symmetric spaces. See~\cite{Ol1990} for an overview of representation theory for classical direct limits of symmetric spaces.  See also~\cite{DPW} and \cite{NRW3} for many basic results on representations of direct-limit groups.
 
\subsection{Direct Limits of Representations}

Suppose that $\{G_n\}_{n\in\N}$ is an increasing sequence of Lie groups (i.e., $G_n$ is a closed subgroup of $G_m$ for $n\leq m$) and that for each $n$ we are provided with a continuous Hilbert representation $(\pi_n,\cH_n)$ such that $(\pi_n,\cH_n)$ is equivalent (by a unitary intertwining operator) to a subrepresentation of $(\pi_{n+1}|_{G_n},\cH_{n+1})$.  Then one has a direct system of representations and may form a \textbf{direct-limit representation} $(\pi_\infty,\cH_\infty)$ of $G_\infty$ on the direct-limit vector space $\cH_\infty \equiv\varinjlim \cH_n$.  
Now $\cH_\infty$ has a natural pre-Hilbert space structure, and if $\pi_\infty(g)$ is a bounded operator on $\cH_\infty$ for all $g\in G_\infty$, then $\pi_\infty$ extends to a continuous Hilbert representation of $G_\infty$ on the Hilbert-space closure $\overline{\cH_\infty} = \overline{\varinjlim \cH_n}$.

\begin{comment}
One of the key tools in representation theory is the study of intertwining operators for representations.
It is clear that an operator $T\in\cB(\cH)$ is an intertwining operator for a Hibert representation $(\pi,\cH)$ of a direct-limit group $G_\infty = \varinjlim G_n$ if and only if it is an intertwining operator for $\pi|_{G_n}$ for each $n$. If $\pi$ is a direct-limit representation, then we can say more:
\begin{lemma}(\cite{KS})
\label{intertwiningOperators}
If $(\pi,\cH) = (\varinjlim \pi_n,\overline{\varinjlim\cH_n})$ is a direct limit of Hilbert representations, then a bounded operator $T\in\cB(\cH)$ is an intertwining operator for $\pi$ if and only if $T|_{\cH_n}$ is an intertwining operator for $\pi|_{G_n}$ for each $n\in\N$.
\end{lemma}
\end{comment}

Direct-limit representations are the easiest representations to construct for $G_\infty$.  The following theorem shows that they can in fact be used to construct a large class of irreducible unitary representations:
\begin{theorem}(\cite{KS})
\label{irreducibleLimit}
Suppose that $\{G_n\}_{n\in\N}$ is a direct system of locally compact groups and that $\{(\pi_n,\cH_n)\}_{n\in\N}$ is a compatible direct system of irreducible unitary representations of $G_n$ for each $n\in\N$.  Then $(\pi,\cH) \equiv (\varinjlim \pi_n,\overline{\varinjlim \cH_n})$ is an irreducible unitary representation of $G_\infty$.
\end{theorem}

We caution the reader that there are many examples of irreducible representations of direct-limit groups which are not given by direct limits of irreducible representations (see \cite[p. 971]{DPW}).

\subsection{Smoothness and Local Finiteness}
Just as for finite-dimensional Lie groups, it is natural to try to gather information about a representation of a direct-limit group by differentiating it to obtain a representation of its Lie algebra.  We begin by defining a notion of smoothness.

\begin{definition}
\label{smoothnessDefinition}
Suppose that $(\pi,\cH)$ is a continuous representation  of a direct-limit group $G_\infty = \varinjlim G_n$ on a Hilbert Space $\cH$ and that $v\in\cH$.  We say that $v$ is a \textbf{smooth vector} for $\pi$ if it is a smooth vector for the restricted representation $(\pi|_{G_n},\cH)$ of $G_n$ for each $n\in\N$. We denote by $\cH^\infty$ the space of all smooth vectors for $\pi$.

Similarly, we say that $v$ is a \textbf{locally finite vector} for $\pi$ if it is a $G_n$-finite vector for the restricted representation $(\pi|_{G_n},\cH)$ of $G_n$ for each $n\in\N$.  We denote by $\cH^\text{fin}$ the space of locally finite vectors for $\pi$.  Note that $\cH^\text{fin}\subseteq\cH^\infty$.
\end{definition}

We remark that the question of how to put a smooth structure on direct limit groups such as $G_\infty$ has been explored extensively in \cite{Glockner} and \cite{NRW1}, where it is shown that under certain technical growth conditions on the $G_n$'s, it is possible to put a smooth structure on $G_\infty$ that is consistent with Definition~\ref{smoothnessDefinition}.

It is not at all clear from the definition that a representation of $G_\infty$ is guaranteed to possess any smooth vectors or locally-finite vectors.  In fact, the existence of smooth vectors is far more subtle for representations of infinite-dimensional Lie groups than for finite-dimensional Lie groups, where every continuous representation on a Frechet space admits a dense subspace of smooth vectors.  There are examples of unitary representations of Banach-Lie groups which do not possess any $C^1$ vectors, much less any smooth vectors (see \cite{BelK}). 
For direct-limit groups, however, a beautiful theorem of Danilenko shows that unitary representations always admit smooth vectors.

\begin{theorem}(\cite{Da}; see also \cite[Theorem 11.3]{KHN})
%\marginpar{Reading Neeb it's clear that smoothness is meant in the sense of our definition}
Suppose that $(\pi,\cH)$ is a unitary representation of a countable direct limit of Lie groups.  Then $\cH^\infty$ is a dense subspace of $\cH$.  
\end{theorem}

Some representations may consist entirely of smooth vectors:
\begin{definition}
Suppose that $G_\infty$ is a direct-limit Lie group. We say that a continuous representation $(\pi,\cH)$ of $G_\infty$ on a Hilbert space $\cH$ is \textbf{smooth} if $\cH^\infty = \cH$.  Similarly, we say that $\pi$ is \textbf{locally finite} if $\cH^\text{fin} = \cH$.

If $G_\infty$ is a direct limit of complex Lie groups, then a continuous Hilbert representation $(\pi,\cH)$ of $G_\infty$ is \textbf{holmorphic} if $\pi|_{G_n}$ is holomorphic for each $n\in\N$.
\end{definition}
\noindent In fact, we will be primarily concerned with smooth representations in this paper.  They play a role for direct-limit groups that is similar to the role played by finite-dimensional representations for finite-dimensional Lie groups.  There are several conditions which are equivalent to smoothness, as we will soon see.  First we need to reference a well-known lemma on smooth one-parameter semigroups.

\begin{theorem}(\cite[Theorem 13.36]{Rudin}; \cite{Riesz-Nagy})
\label{theoremInRudin}
Let $X$ be a Banach space and let $Q:[0,\infty)\rightarrow \cB(X)$ be a strongly-continuous one-parameter semigroup with the (possibly unbounded) operator $A$ as its infinitesmial generator.  Then the following are equivalent 
\begin{enumerate}
\item The domain of $A$ is all of $X$.
\item $\displaystyle \lim_{\epsilon\rightarrow 0} ||Q(\epsilon)-1|| = 0$
\item $A\in \cB(X)$ and $Q(t)=e^{tA}$ for all $t\geq 0$.
\end{enumerate}
\end{theorem}

This result about one-parameter semigroups straightforwardly generalizes to representations of Lie groups:
\begin{theorem}
\label{boundedDifferential}
Let  $(\pi,\cH)$ be a continuous Hilbert representation of a connected Lie group $G$.  Then the following are equivalent:
\begin{enumerate}
\item $\pi$ is analytic.
\item There is a Lie algebra representation $\mathrm{d}\pi: \gg\rightarrow \cB(\cH)$ (for which $\gg$ acts by bounded operators) such that 
\[
   \pi(\exp X) = \exp(\mathrm{d}\pi(X))
\]
for each $X\in\gg$.
\item $\pi$ is smooth.
\item $\pi$ is norm continuous.
\end{enumerate}
\end{theorem}
\begin{proof}
The result follows from judicious application of Theorem~\ref{theoremInRudin}. The details are left to the reader.
\end{proof}

It is certainly possible to construct continuous unitary representations of direct-limit groups which possess no locally finite vectors.  This behavior is already present for finite-dimensional Lie groups, however:  an irreducible infinite-dimensional representation of a noncompact Lie group $G$ does not possess any $G$-finite vectors.  More surprisingly, it is possible to construct an irreducible unitary representation of a lim-compact group which has no locally finite vectors (\cite{KHNPrivate}). %\marginpar{CITE K-H. Neeb private correspondence?}  
However, Corollary~\ref{smoothClassification} will show that smooth representations of connected lim-compact groups always consist entirely of locally finite vectors.

It is well known that every continuous, finite-dimensional representation of a Lie group is smooth. However, it is also possible to construct infinite-dimensional Hilbert representations which are smooth.  Suppose that $U$ is a compact Lie group and that $(\pi,V)$ is a finite-dimensional representation of $U$.  Without loss of generality, we may assume that $\pi$ is unitary.  Now consider the representation 
\[
\left(\infty\cdot\pi,\infty\cdot V\right) \equiv \left(\bigoplus_{n\in\N}\pi,\bigoplus_{n\in\N}V\right)
\]
constructed by taking a Hilbert space direct sum of countably many copies of $(\pi,V)$.  For each $v\in\infty\cdot V$, we consider the closed invariant subspace \[W=\overline{\langle(\infty\cdot \pi)(U)v\rangle}\] generated by $v$.  Then $W$ gives a cyclic primary representation of $U$ and decomposes into a direct sum of representations equivalent to  $(\pi,V)$. From \cite{GM} we see that every cyclic primary representation of the compact group $U$ is finite-dimensional.  Thus $\dim W<\infty$ and so $v$ is a $U$-finite vector.  

In fact, the next theorem shows that in a certain sense, primary representations (or more precisely, finite direct sums of them) provide the only way to obtain infinite-dimensional smooth representations of $U$:

\begin{theorem}
\label{smoothClassification}
Let $(\pi,\cH)$ be a unitary representation of a compact Lie  group $U$. Then the following are equivalent.
\begin{enumerate}
\item $\pi$ is smooth.
\item $\pi$ decomposes into a finite direct sum of primary representations of $U$.  
\item $\pi$ is locally finite.
\end{enumerate}
\end{theorem}

\begin{proof}
Let $(\pi,\cH)$ be a unitary representation of $U$.  Then we can write 
\[  
    \cH\cong_G \bigoplus_{\delta\in\widehat{G}} \cH_\delta,
\]
where $\cH_\delta$ is the space of $\delta$-isotypic vectors for each $\delta\in\widehat{G}$ (that is, vectors in $\cH_\delta$ generate primary representations that are direct sums of copies of $\delta$). Thus $\pi$ is a finite direct sum of primary representations if and only if $\cH_\delta = \{0\}$ for all but finitely many $\delta\in\widehat{G}$.

 We begin by showing that  (1)$\implies$(2).  That is, suppose that $\pi$ is smooth.  Recall from the highest-weight theorem that irreducible representations of compact connected Lie groups are parametrized by a discrete semilattice  $\Lambda^+(\gg,\gh) \subseteq i\gh^*$ of dominant integral weights, where $\gh$ denotes a Cartan subalgebra of $\gg$.  Let $\cS$ denote the set of all weights $\lambda\in \Lambda^+(\gg,\gh)$ such that $\lambda$ appears as the highest weight of a subrepresentation of $\pi$.  

If $\lambda\in\cS$, then $\lambda(X)$ is an eigenvalue of $d\pi(X)$ for each $X\in\gh$. But since $\pi$ is smooth, Theorem~\ref{boundedDifferential} implies that $||d\pi(X)||<\infty$ for each $X\in\gh$.  Thus, since $\Lambda^+(\gg,\gh)$ is a semilattice, it follows that $\{\lambda(X):\lambda\in\cS\}$ is finite for each $X$.  Hence $\cS$ is finite because $\gh$ is finite-dimensional; that is, $\pi$ decomposes as a direct sum of a finite number of primary representations.

Next we show that (2)$\implies$(3).  Suppose that 
\[  
    \cH\cong_U \bigoplus_{i=1}^n \cH_{\delta_i},
\]
where $\delta_i\in\widehat{U}$ for each $i$.  We will show that $\pi$ is smooth.  For each $v\in\cH$, we can write $v = v_1+\cdots+v_n$, where $v_i\in\cH_{\delta_i}$.  Then
\[
\langle \pi(U)v \rangle \subseteq \bigoplus_{i=1}^n \langle \pi(U)v_i \rangle.
\]
However, because each space $\langle \pi(U)v_i \rangle$ gives a cyclic primary representation of $U$, we see that it is finite-dimensional (see \cite{GM}).  Thus $v$ is $U$-finite.  Because $v\in\cH$ was arbitrary, it follows that $\pi$ is locally finite.

Finally, it is clear that (3)$\implies$(1); that is, if $\cH^\text{fin} = \cH$ then obviously $\cH^\infty = \cH$. 
\end{proof}

\begin{corollary}
Suppose that $U$ is a compact Lie group and that $(\pi,\cH)$ is a holomorphic representation of $U_\C$.  Then $\pi$ is locally-finite.
\end{corollary}
\begin{proof}
First we note that holomorphic representations are in particular smooth, and it follows that $\pi|_U$ is a smooth representation of $U$.  Hence $\pi|_U$ is locally-finite by Theorem~\ref{smoothClassification}.  Let $v\in\cH$ and consider the finite-dimensional subrepresentation of $\pi|_U$ on $V=\langle\pi(U)v\rangle$.  This representation has a unique holomorphic extension to a finite-dimensional representation of $U_\C$, which is thus a finite-dimensional subrepresentation of $\pi$ which contains $v$.
\end{proof}

\begin{corollary}
Suppose that $(\pi,\cH)$ is a continuous Hilbert representation of a connected lim-compact group $U_\infty$.  Then $\pi$ is smooth if and only if it is locally finite.  That is, $\cH^\infty = \cH$ if and only if $\cH^\text{fin} = \cH$.
\end{corollary}
\begin{proof}
Fix $n\in\N$.  Then representation $(\pi|_{U_n},\cH)$ of $U_n$ may be unitarized.  Let $\cH^{\infty,n}$ denote the space of $U_n$-smooth vectors and let $\cH^{\text{fin},n}$ denote the space of $U_n$-finite vectors.  By Theorem~\ref{smoothClassification}, it follows that $\cH = \cH^{\infty,n}$ if and only if $\cH = \cH^{\text{fin},n}$.  The corollary then follows, since a vector in $\cH$ is $U_\infty$-smooth if and only if it is $U_n$-smooth for all $n\in\N$, and it is $U_\infty$-finite if and only if it is $U_n$-finite for all $n\in\N$.
\end{proof}

The following corollaries restate the conclusions of the previous theorem in terms of weights.  Different proofs of these corollaries may be found in Lemma 3.5 and Proposition 3.6 of \cite{NRW3}.
\begin{corollary}
\label{smoothnessFiniteWeightSet}
Fix a Cartan subalgebra $\gh$ in $\gu$, and suppose that $(\pi,\cH)$ is a unitary representation of $U$.  Then $\pi$ is smooth if and only if $\#\Delta(\pi)<\infty$ (that is, $\pi$ has only finitely many weights). 
\end{corollary}
\begin{proof}
This follows immediately from the equivalence of conditions (1) and (2) in Theorem~\ref{smoothClassification}.
\end{proof}

\begin{corollary}
\label{smoothnessFiniteWeightSetLimit}
Suppose that $U_\infty$ is a lim-compact group.  As before, we fix a subalgebra $\gh_\infty = \varinjlim\gh_n$ in $\gu_\infty$, where each $\gh_n$ is a Cartan subalgebra of $\gu_n$.  Suppose that $(\pi,\cH)$ is a unitary representation of $U$.  Then $\pi$ is smooth if and only if $\#\Delta(\pi|_{U_n})<\infty$ (that is, $\pi$ has only finitely many weights) for each $n\in\N$. 
\end{corollary}
\begin{proof}
This result follows immediately from Corollary~\ref{smoothnessFiniteWeightSet} and the fact that $\pi$ is smooth if and only if $\pi|_{U_n}$ is smooth for each $n$.
\end{proof}

Suppose now that $U_\infty$ is a propagated lim-compact group.  We recursively choose a countable orthonormal basis
 $\{e_i\}_{i\in\N}$ for $\gh_\infty$ as in Section~\ref{propagationSection}.  Consider the supremum norm of a weight $\lambda\in i\gh_n^*$, given by 
 \[
    ||\lambda||_\infty = \max_{1\leq i\leq r_n} |\lambda(e_i)|
 \]
 We then obtain the following useful theorem, which is a slight generalization of \cite[Proposition 3.14]{NRW3}:
\begin{theorem}
\label{bounded_coefficients}
%\marginpar{It may be possible to somewhat weaken the condition here that $U_\infty$ be propagated}
A unitary representation $(\pi,\cH)$ of a propagated direct limit $U_\infty$ of simply-connected compact semisimple Lie groups is smooth if and only if there is $M>0$ such that for all $n$ one has $||\lambda||_\infty<M$ for each weight $\lambda\in i\gh_n^*$ that appears as the highest weight for an irreducible subrepresentation of $\pi|_{U_n}$.
\end{theorem}
\begin{proof}

First we prove the theorem in the case that $U_\infty$ is a direct limit of compact simple Lie groups.

Let $(\pi,\cH)$ be a unitary representation of $U_\infty$.  Suppose there is $M\in\N$ such that for all $n$ one has $||\lambda||_\infty<M$ for each weight $\lambda\in i\gh_n^*$ that appears as the highest weight for an irreducible subrepresentation of $\pi|_{U_n}$.  If $\lambda\in i\gh^*$ is a highest weight which appears in $\pi|_{U_n}$, then it has the form
\[
   \lambda = \sum_{i=1}^{r_n} a_i e_i, \text{ where } a_i\in \Z \text{ and } -M\leq a_i\leq M. 
\]
Thus, there are only $(2M)^{r_n}$ possible values for $\lambda$.  In other words, $\pi|_{U_n}$ may be written as a direct sum of finitely many primary representations and is thus smooth by Theorem~\ref{smoothClassification}.  Because $n\in\N$ was arbitrary, we have that $\pi$ is smooth.

To prove the other direction, suppose that for each $M>0$ there is $n\in\N$ and a highest weight $\lambda\in i\gh_n^*$ of an irreducible subrepresentation of $\pi|_{U_n}$ such that $||\lambda||_\infty > M$. Fix $M>0$ and pick $n\in\N$ and $\lambda\in i\gh_n^*$ satisfying those conditions. Then $\lambda = \sum_{i=1}^{r_n} c_i e_i$, where $c_i\in\Z$ for each $i$.  Because $||\lambda||_\infty > M$, we see that there is some index $j$ such that $|c_j|>M$.

By considering the $A_n$, $B_n$, $C_n$, and $D_n$ cases separately, we see that there is a Weyl group element $w\in W(\gg_n,\ga_n) $ such that $w(e_1) = e_i$ and $w(e_i) = e_1$.  Then $|(w\lambda)(e_1)| = |c_j| > M$.   By the Highest-Weight Theorem, we see that $w\lambda\in\Delta(\pi|_{U_n})$; that is, $w\lambda$ is an $\gh_n$-weight for $\pi|_{U_n}$.  It is then clear that $(w\lambda)|_{\gh_k}$ is an $\gh_k$-weight for $\pi|_{U_k}$ whenever $k\leq n$ (since every $w\lambda$-weight vector in $\cH$ is automatically a $(w\lambda)_{\gh_k}$-weight vector).  Furthermore, since $|(w\lambda|_{\gk_n})(e_1)| = |c_j| > M$, we see that $||w\lambda|_{\gk_n}||_\infty > M$.  
 
 Thus, if $k\in\N$ is fixed, then for each $M\in\N$ there is a weight $\lambda\in\Delta(\pi|_{U_k})$ such that $||\lambda||>M$.  Hence $\Delta(\pi|_{U_k})$ is not a finite set and thus by Corollary~\ref{smoothnessFiniteWeightSetLimit} it follows that $\pi$ is not smooth.  
 
 Suppose more generally that $U_\infty$ is a propagated direct limit of semisimple Lie groups.  Then we can write $U_k = U_k^1\times U_k^2 \times \cdots \times U_k^d$ for all $k\in \N$ in such a way that $\{U_n^i\}_{n\in\N}$ is a propagated direct system of compact simple Lie groups for each $1\leq i \leq d$.  We can then recursively choose Cartan subalgebras $\gh_n = \gh_n^i\oplus \gh_n^2 \oplus \cdots \oplus \gh_n^d$, where $\gh_n^i$ is a Cartan subalgebra of $\gu_n^i$ for each $i$ and $n$.  A weight in $\lambda \in i\gh_n^*$ is dominant integral if and only if $\lambda|_{\gh_n^i}$ is dominant integral for each $1\leq i\leq d$.  Since $U_\infty^i$ is a propagated direct limit of compact simple Lie groups, it follows that there is $M_i>0$ such that for all $n\in\N$ one has that $||\lambda||_\infty<M_i$ for each highest weight $\lambda\in\gh_n^*$ appearing in $\pi|_{U_n^i}$.  Since $\displaystyle\max_{1\leq i\leq d} M_i < \infty$, we are done. 
\end{proof}

We end the section with the following remarkable result, which implies that the smoothness of a representation of a direct limit of simple compact groups is controlled by the smoothness of the restriction to any nontrivial one-dimensional analytic subgroup.   

\begin{theorem}
\label{holomorphicControl}
Let $U$ be a compact simple Lie group. Then a unitary representation $(\pi,\cH)$ of $U$ is smooth if and only if there is $X\in \gu\backslash\{0\}$ such that $\mathrm{d}\pi(X)$ is a bounded operator on $\cH$. 
\end{theorem}

\begin{proof}
One direction is obvious.  To show the other direction, suppose that $(\pi,\cH)$ is a non-smooth unitary representation of $U$.  We will show that $\mathrm{d}\pi(X)$ has an unbounded spectrum for any $X\in \gu \backslash\{0\}$. Let $\gh$ be any Cartan subalgebra for $U$.

Because $\pi$ is not smooth, it follows that there is for each $M>0$ weight $\lambda\in\Delta(\pi)$ with $||\lambda||_\infty>M$.  As in the proof of Theorem~\ref{bounded_coefficients}, we see that for each Weyl-group element $w\in W(\gu,\gh)$, the weight $w\lambda$ is in $\Delta(\pi)$.  If we write $\lambda = \sum_{i=1}^{r} a_i e_i$, then there is some $j$ such that $|a_j|>M$.  We can use the Weyl group to permute the basis elements so that $a_j$ appears as the $i^\text{th}$ coefficient of a weight in $\Delta(\pi|_U)$.  Thus we have that the set
\[
\{\langle \lambda,e_i\rangle | \lambda\in\Delta(\pi)\}
\]
 of $i^\text{th}$ coefficients of weights of $\pi$ is unbounded for all $i\leq r$.
 
In other words, one has for each $n\in\N$ that the set of weights in $\Delta(\pi)$ is unbounded in every direction on $\gh$. 
 It follows that $\mathrm{d}\pi(X)$ has an unbounded spectrum for all $X\in\gh$. Because every element of $\gu$ is contained in some Cartan subalgebra, the result follows.
\end{proof}
\begin{corollary}
\label{holomorphic_equivalence}
Let $U_\infty$ be a direct limit of compact simple Lie groups. Then a unitary representation $(\pi,\cH)$ of $U_\infty$ is smooth if and only if there is $X\in\gu\backslash\{0\}$ such that $\mathrm{d}\pi(X)$ is a bounded operator on $\cH$. 
\end{corollary}
\begin{proof}
This corollary follows immediately by applying Lemma~\ref{holomorphicControl} to $U_n$ for each $n$ in $\N$.
\end{proof}

Note that this result is false for non-simple compact groups: suppose that $J$ and $T$ are compact Lie groups, that $(\pi,\cH)$ is a smooth unitary representation of $J$, and that $(\sigma, \cK)$ is a non-smooth unitary representation of $T$.  Then the outer tensor product representation
$(\pi \boxtimes \sigma, \cH\otimes \cK)$ of $J\times T$ has the property that $\pi|_J$ is smooth but $\pi|_T$ is non-smooth.

\subsection{Generalizing Weyl's Unitary Trick}
Weyl's Unitary Trick plays a crucial role in understanding finite-dimensional representations of finite-dimensional Lie groups.  There is a natural extension of Weyl's Unitary Trick to locally-finite representations of direct-limit groups.  The first step is to extend Weyl's unitary trick to smooth representations of finite-dimensional groups.  We begin with a well-known lemma on intertwining operators of smooth representations.
\begin{lemma}
\label{intertwiningDifferential}
Suppose that $(\pi,\cH)$ is a smooth Hilbert representation of a Lie group $G$.  Then the derived representation $\mathrm{d}\pi:\gg\rightarrow\cB(\cH)$ possesses the same algebra of intertwining operators as $\pi$.
\end{lemma}

Now we are ready to extend Weyl's Trick to smooth representations of finite-dimensional groups.

\begin{theorem}
\label{WeylInfiniteTrick}
Suppose that $U$ is a compact Lie group and that $G$ is a (not necessarily compact) closed subgroup of $U_\C$ such that $U_\C$ is a complexification of $G$.  There are one-to-one correspondences between the following categories of representations on $\cH$ which preserve the algebras of intertwining operators:
\begin{enumerate}
\item Locally-finite representations of $G$ on $\cH$
\item Holomorphic representations of $U_\C$ on $\cH$
\item Smooth representations of $U$ on $\cH$
\end{enumerate}

%Furthermore, a smooth representation of $U$ on $\cH$ is unitary if and only if the corresponding representation of $G$ on $\cH$ has the property that $\pi(g)^* = \pi(\theta(g^{-1}))$ for all $g\in G$.
\end{theorem}

\begin{proof}
We begin by reminding the reader that a representation of $U$ is smooth if and only if it is locally-finite (see Theorem~\ref{smoothClassification}) and that every holomorphic representation of $U_\C$ is locally-finite.

We will construct the correspondences $(1)\rightarrow (2)$ and $(2)\rightarrow (1)$.  The proofs for $(2)\rightarrow (3)$ and $(3)\rightarrow (2)$ are identical.

One passes from $(2)$ to $(1)$ quite easily: if $(\pi,\cH)$ is a locally-finite representation of $U_\C$, then it is clear that $\pi|_G$ is a locally-finite representation of $G$.  

To construct $(1)\rightarrow (2)$, we suppose that $(\pi,\cH)$ is a locally finite representation of $G$.  
We wish to construct a holomorphic representation $\pi_\C$ of $U_\C$ on $\cH$ such that $\pi_\C|_G = \pi$.  
First we notice that each vector $v\in\cH$ is contained in a finite-dimensional $G$-invariant subspace $W$.  Write $\pi^W$ for the subrepresentation of $\pi$ corresponding to  $W$. By the finite-dimensional Weyl Trick, we see that $\pi^W$ uniquely extends to a holomorphic representation $\pi^W_\C$ of $U_\C$ on $W$.
We define $\pi_\C(g)v = \pi^W_\C(g)v$ for each $v\in W$ and $g\in U_\C$.  If $V$ and $W$ are finite-dimensional invariant subspaces of $\cH$ and $v\in V\cap W$, then the uniqueness of the holomorphic extension shows that $\pi^W_\C(g)v = \pi^V_\C(g)v$ and thus $\pi_\C$ is well-defined and gives a locally-finite, holomorphic representation of $U_\C$.

That these one-to-one correspondences of representations preserve the algebra of intertwining operators follows from passing to the derived representation and using Lemma~\ref{intertwiningDifferential}.

\end{proof}

Our infinite-dimensional version of Weyl's Trick is then an immediate corollary (see \cite[Proposition 3.6]{NRW3} for a partial version of this result and a different proof):
\begin{corollary}
\label{GeneralizedWeylTrick}
Suppose that $G_\infty/K_\infty$ is a lim-noncompact Riemannian symmetric space which is the c-dual of a lim-compact symmetric space $U_\infty/K_\infty$ where $U_n/K_n$ and $U_n$ are simply-connected for each $n$.  Finally, let $\cH$ be a Hilbert space.  There are one-to-one correspondences between the following categories of representations on $\cH$ which preserve the algebras of intertwining operators:
\begin{enumerate}
\item Locally finite representations of $G_\infty$ on $\cH$
\item Holomorphic representations of $(U_\infty)_\C$ on $\cH$
\item Smooth representations of $U_\infty$ on $\cH$
\end{enumerate}

%Furthermore, a smooth representation of $U$ on $\cH$ is unitary if and only if the corresponding representation of $G$ on $\cH$ has the property that $\pi(g)^* = \pi(\theta(g^{-1}))$ for all $g\in G_\infty$.
\end{corollary}
\begin{proof}
This corollary follows immediately by applying Theorem~\ref{WeylInfiniteTrick} to representations of $G_n$, $(U_n)_\C$, and $U_n$ on $\cH$ for each $n\in\N$.
\end{proof}

\subsection{Highest-Weight Representations}
\label{highestWeightRepSection}
Now suppose that $G_\infty/K_\infty$ is an admissible lim-noncompact symmetric space which is the c-dual of a lim-compact symmetric space $U_\infty/K_\infty$.  We wish to construct irreducible spherical and conical representations for $G_\infty/K_\infty$ and $U_\infty/K_\infty$.
 The most natural way to do this would be to construct a direct limit of spherical/conical representations.  The following lemma provides the foundation for this construction and is a generalization of a result proved by \'{O}lafsson and Wolf in \cite[Lemma~5.8]{OW2013}.

\begin{theorem}
\label{irreduciblerestriction}
Let $U_\infty/K_\infty$ be a propagated lim-compact symmetric space such that $U_n/K_n$ is simply connected for each $n\in\N$. Fix indices $n<m$ and dominant weights $\mu_n\in\Lambda^+(\gu_n,\ga_n)$ and $\mu_m\in\Lambda^+(\gu_m,\ga_m)$ such that $\mu_n|_{\ga_n} = \mu_m$.  Consider the irreducible spherical representations $(\pi_{\mu_m},\cH_{\mu_m})$ and $(\pi_{\mu_n},\cH_{\mu_n})$ of $U_m$ and $U_n$, respectively, with respective highest weights $\mu_m$ and $\mu_n$.  Let $w$ be a highest-weight vector for $\pi_{\mu_m}$.

Then the representation of $U_n$ on 
$W = \langle \pi_{\mu_m}(U_n) w \rangle$
is equivalent to $\pi_{\mu_n}$ (and, in particular, is irreducible).
\end{theorem}
\begin{proof}
We consider the action of $U_n$ on $W$.  For each dominant weight $\nu$ in $\Lambda^+(\gg_n,\ga_n)$, let $w_\nu$ be the orthogonal projection of $w$ onto the space of $\nu$-isotypic vectors.  Then $w = \sum_{\nu} w_\nu$ (note that $w_\nu = 0$ for all but finitely many choices of $\nu$).
We also write $W_\nu = \langle \pi_{\mu_m}(U_n) w_\nu \rangle$ for each $\nu \in \Lambda^+(\gg_n,\ga_n)$.  %Because $W_\nu$ consists of $U_n$-isotypic vectors of type $\nu$, we see that the action of $U_n$ on $W_\nu$ is $U_n$-isomorphic to a direct sum of copies of the irreducible representation $(\pi_\nu,\cH_\nu)$ with highest-weight $\nu$.  

Since $w$ is a $U_m$-highest-weight vector for $\pi_{\mu_m}$, we have that $\pi(M_mN_m)w = w$, and in particular, $\pi(M_nN_n)w = w$.  Since the space of isotypic vectors in $W$ of type $\mu_n$ is invariant under $G_n$, it follows that $w_\nu$ is fixed under $M_nN_n$ for each $\nu \in \Lambda^+(\gg_n,\ga_n)$.  Thus Lemma~\ref{irreducible_primary} shows that if $w_\nu\neq 0$, then $W_\nu$ is a $U_n$-irreducible subspace of $W$ and that $w_\nu$ is a highest-weight vector for $W_\nu$.  In particular, $w_\nu$ is a weight vector of weight~$\nu$.

On the other hand, since $w$ is a $U_m$-weight vector of weight $\mu_m$, it follows that it is a $U_n$-weight vector of weight $\mu_n = \mu_m|_{\ga_n}$.  But we also have that $w = \sum_\nu w_\nu$, where each $w_\nu$ is a weight vector of weight $\nu$.  Hence $w = w_{\mu_n}$ and $W = W_{\mu_n}$, and so we are done.  
\end{proof}

We follow \cite[p. 464--466]{W2011} for the construction of highest-weight representations.
For each $n$, we denote the set of fundamental weights by  $\xi_{n,1},\ldots\xi_{n,{r_n}}$, where $r_n = \dim \ga_n$ and where we have numbered the fundamental weights according to the roots as in Section~\ref{propagationSection}.  Suppose $k\leq n$.  One can show that 
\begin{equation}
\label{fundamentalWeightRestriction}
\xi_{n,i}|_{\ga_k} = \xi_{k,i}
\end{equation}
 for all $n\in\N$ and $i\leq r_k$.  Furthermore, one can check that $\xi_{n,i}|_{\ga_k} = 0$ for $r_k < i\leq r_n$.  Thus 
\begin{equation}
\label{dominantWeightsReprise}
\Lambda^+(\gg_n,\ga_n)=\Z^+\xi_{n,1}+\cdots + \Z^+\xi_{n,r_n}=\left\{\left. \sum_{j=1}^{r_n} c_j\xi_{n,j}\bmid c_j\in \Z^+\right\},
\end{equation}
 and 
\begin{equation}
\label{dominantWeightRestriction}
   \left.\left(\sum_{j=1}^{r_n} c_j\xi_{n,j}\right)\right|_{\ga_k} = \left(\sum_{j=1}^{r_k} c_j\xi_{k,j}\right) \in \Lambda^+(\gg_k,\ga_k)
\end{equation}
whenever $k\leq n$.

  We can thus form a projective limit
\[
\Lambda^+ \equiv \Lambda^+(\gg_\infty,\ga_\infty) = \varprojlim \Lambda^+(\gg_n,\ga_n).
\]
We say that $\Lambda^+(\gg_\infty,\ga_\infty)$ is the set of \textbf{dominant integral weights} for the restricted root system $\Sigma(\gg_\infty,\ga_\infty)$.  That is, $\Lambda^+$ consists of the elements $\lambda$ of $\ga_\infty^* = \varprojlim \ga_n^*$ such that $\lambda|_{\ga_n}$ is dominant and integral for every $n$.  Notice that (\ref{fundamentalWeightRestriction}) implies that
for each $i\in\N$ there is a weight $\xi_i\in\ga_\infty^*$ such that $\xi_i|_{\ga_n} = \xi_{n,i}$ for each $n\in\N$.  

Just as in the finite-dimensional case, weights in $\Lambda^+$ can be used to create highest-weight representations of $U_\infty$.  To see this, fix $\mu\in\Lambda^+$.  For $n$ in $\N$, let $(\pi_{\mu_n},\cH_{\mu_n})$ be the irreducible representation of $U_n$ with highest weight $\mu_n \equiv \mu|_{\ga_n}$, and let $v_n \in \cH_{\mu_n}$ be a nonzero highest-weight vector.  By Theorem~\ref{irreduciblerestriction}, we see that $\pi_{\mu_n}$ may be embedded unitarily into $\pi_{\mu_{n+1}}$ by identifying the respective highest-weight vectors $v_n$ with $v_{n+1}$.  The corresponding unitary representation of $U_\infty$ constructed by the direct limit of $\pi_{\mu_n}$, $n\in\N$ is denoted by
\[
(\pi_\mu, \cH_\mu) = \left(\varinjlim \pi_{\mu_n}, \overline{\varinjlim \cH_{\mu_n}} \right),
\]
where $\cH_\mu = \overline{\varinjlim \cH_{\mu_n}}$ is the Hilbert completion of the algebraic direct limit $\varinjlim \cH_{\mu_n}$  of Hilbert spaces.
  We refer to $\pi_\mu$ as the \textbf{highest-weight representation with highest weight $\mu$}.  Note that a direct limit of irreducible representations of~$U_n$ is an irreducible representation of~$U_\infty$ by~\ref{irreducibleLimit}.

If $\dim \ga_\infty = \infty$, then we can write elements of $\ga^*$ as sequences $(a_i)\in\Z$  of integers, so that a sequence $(a_i)\in\Z$  corresponds to the formal sum $\sum_{i\in\N} a_i e_i \in \ga_\infty^*$.  We now use this notation to write down the fundamental weights for $\Sigma(\gg_\infty,\ga_\infty)$ for some infinite Dynkin-diagram types.

If $\Sigma(\gg_\infty,\ga_\infty)$ has type $A_\infty$, then
\[
\xi_i = (0,\ldots,0,2,2,2,\ldots)
\]
where the first $i$ entries in $\xi_i$ are zeros.  

If $\Sigma(\gg_\infty,\ga_\infty)$ has type $B_\infty$, then
\[
\xi_1 = (1,1,1,\ldots) \text{ and } \xi_i = (0,\ldots,0,2,2,2,\ldots) \text{ for } i>1,
\]
where the first $i-1$ entries in $\xi_i$ are zero for $i>1$.

If $\Sigma(\gg_\infty,\ga_\infty)$ has type $C_\infty$, then
\[
\xi_i = (0,\ldots,0,2,2,2,\ldots),
\]
where the first $i-1$ entries in $\xi_i$ are zero.

If $\Sigma(\gg_\infty,\ga_\infty)$ has type $D_\infty$, then
\[
\xi_1 = (1,1,1,\ldots), \xi_2 = (-1,1,1,\ldots) \text{ and } \xi_i = (0,\ldots,0,2,2,2,\ldots) \text{ for } i\geq 3,
\]
where the first $i-1$ entries in $\xi_i$ are zero for $i\geq 3$.

By examining the fundamental weights in each case and extending them to weights on $\gh_\infty$, it follows from the boundedness condition in Theorem~\ref{bounded_coefficients} that a highest-weight representation $(\pi_\mu,\cH_\mu)$ for $\lambda\in\Lambda^+(\gg_\infty,\ga_\infty)$ will be smooth if and only if we can write $\lambda$ as a finite linear combination 
\[
\lambda= \sum_{i=1}^n c_i \xi_i,
\] 
where $c_i\in\N$ for each $n$.  In particular, if $\dim \ga_\infty < \infty$, then every highest-weight representation  $(\pi_\mu,\cH_\mu)$ for $\lambda\in\Lambda^+(\gg_\infty,\ga_\infty)$ is smooth.

\section{Conical Representations for Admissible Direct Limits}

In this section, we give a natural definition for conical representations of admissible lim-noncompact symmetric spaces $G_\infty/K_\infty$.  
As before, we assume that $G_\infty/K_\infty$ is the c-dual of a propagated lim-compact symmetric space $U_\infty/K_\infty$.  By using the generalization of Weyl's Unitary Trick from the previous chapter, each smooth cyclic representation of $U_\infty$ gives rise to a smooth cyclic representation of $G_\infty$, and it is natural to say that a smooth cyclic  representation of $U_\infty$ is conical if the corresponding representation of $G_\infty$ is conical.  

In fact, we will see that in some cases it is possible to define nonsmooth unitary representations of $U_\infty$ which are conical but do not correspond to continuous Hilbert representations of $G_\infty$.  This is a strange situation which does not occur in the finite-dimensional case.

With these definitions, we classify all of the irreducible cyclic unitary representations of $U_\infty$ which are conical.  
Next we see that smooth conical unitary representations of $U_\infty$ decompose into a discrete direct sum of highest-weight representations.  

It will follow from this classification, together with \cite[Theorem 4.5]{DOW}, that if $\Rank\ U_\infty/K_\infty = \infty$, then there are no smooth unitary representations of $U_\infty$ which are both spherical and conical.  On the other hand, if $\Rank\ U_\infty/K_\infty < \infty$, then we will see that a smooth irreducible unitary representation of $U_\infty$ is spherical if and only if it is conical.  This situation is also in stark contrast to the situation for finite-dimensional symmetric spaces, for which finite-dimensional representations are spherical if and only if they are conical.

\subsection{Definition of Conical Representations}
We begin by presenting our definition of conical representations for lim-Riemannian symmetric spaces.  Let $G_\infty/K_\infty$ be the c-dual of a propagated lim-compact symmetric space $U_\infty/K_\infty$ such that $U_n/K_n$ and $U_n$ are simply-connected for each $n$ and assume that $G_\infty/K_\infty$ is admissible. 

For finite-dimensional symmetric spaces, it is possible to consider a finite-dimen\-sion\-al conical representation to be a representation of either $G$ or $U$ (where $G/K$ is the c-dual of the compact symmetric space $U/K$).  On the one hand, many harmonic analysis applications of conical representations appear on the horocycle space $G/MN$, so in a certain sense it is most natural to speak of conical representations of $G$.  On the other hand, these representations are only unitary if we move to the compact group $U$.   

Similarly, because unitarity is crucially important in the arguments which follow, we will mainly consider unitary conical representations of $U_\infty$.  However, it is important to remember that smooth representations of $U_\infty$ correspond to locally-finite representations of $G_\infty$ under Theorem~\ref{GeneralizedWeylTrick}, and vice versa.

We are now ready to present the definition:

\begin{definition}
A unitary representation $(\pi,\cH)$ of $U_\infty$ is {\bf conical} if there is a nonzero cyclic vector $v\in\cH^\mathrm{fin}$ such that $\pi(M_n N_n)v = v$ for all $n\in\N$. In that case, we say that $v$ is a {\bf conical vector} for $\pi$.
\end{definition}

\noindent \textbf{Remark: } Even though $\pi$ is a representation of $U_\infty$ in the above definition and $M_n N_n$ is a subgroup of the c-dual group $G_\infty$, the fact that $v$ is locally-finite implies that each $\pi|_{U_n}$ extends analytically to a representation of $G_n$ on the finite-dimensional space $\langle \pi(U_n) v\rangle$, and so the condition that $\pi(M_n N_n)v = v$ makes sense.

Notice also that we do not require that conical representations of $U_\infty$  be smooth.  This opens the door to the aforementioned possibility of constructing conical representations of $U_\infty$ which do not correspond to representations of $G_\infty$ under the generalized unitary trick, and indeed we will construct many examples of such representations in Section~\ref{disintegrationSection}.

\subsection{Classification of Conical Representations}
We now begin to classify the unitary conical representations of $U_\infty$.  We determine which representations are irreducible and show how conical representations decompose into subrepresentations.

\begin{theorem}
\label{irreduciblehighestweightprop}
Suppose that $U_\infty/K_\infty$ is a propagated lim-compact symmetric space with $U_n$ and $U_n/K_n$ simply-connected for each $n$ and such that the c-dual $G_\infty/K_\infty$ is admissible. Suppose further that $(\pi, \cH)$ is a conical representation with a conical vector $v$.  For each $n$, write $\Gamma_n(\pi,v)$ for the set of highest weights $\mu$ in $\Lambda^+(\gu_n,\ga_n)$ such that the projection $v_\mu = \operatorname{pr}_{\cH_{\mu}} v$ of $v$ onto the space of $U_n$-isotypic vectors of type $\mu$ is nonzero. Then
\begin{enumerate}
\item For each $n\in\N$ and $\mu \in \Gamma_n(\pi,v)$, the action of $U_\infty$ on
      $
           \overline{\langle \pi(U_\infty) v_\mu \rangle}
      $
      gives a conical representation of $U_\infty$ with conical vector $v_\mu$.
\item $\pi$ decomposes into an orthogonal direct sum of disjoint conical representations as follows:
      \[
           \cH = \overline{\langle \pi(U_\infty) v \rangle} = \bigoplus_{\mu\in\Gamma_n(\pi,v)} \overline{\langle \pi(U_\infty) v_\mu \rangle}
      \]
      
\item If $\pi$ is irreducible, then $\pi$ is equivalent to a highest-weight representation $\pi_\mu$ for some $\mu \in \Lambda^+(\gg_\infty, \ga_\infty)$.
\item If $\pi$ is irreducible, then $\dim \cH^{M_\infty N_\infty} = 1$.
\end{enumerate}  

\end{theorem}

\begin{proof}
   For each $n\in\N$, the set $\Gamma_n(\pi,v)$ is finite because $v$ is $U_n$-finite for all $n$.  Then the decomposition of $v$ into $U_n$-isotypic vectors may be written \[v = \sum_{\mu\in\Gamma_n(\pi,v)} v_\mu,\] where $v_\mu = \operatorname{pr}_{\cH_{\mu}} v$.  Since each isotypic subspace is $U_n$-invariant, it follows that $v_\mu\in\cH^{M_n N_n}$ for each $\mu\in\Gamma_n(\pi,v)$.  Note that $\langle\pi(U_n)v_\mu\rangle$ gives a primary representation of $U_n$ of type $\mu$.  Hence, by Lemma~\ref{irreducible_primary}, it is an irreducible representation with highest-weight vector $v_\mu$.

   We repeat the same process for $U_{n+1}$, writing the decomposition of $v$ into $U_{n+1}$-isotypic vectors as
   \begin{equation} 
   \label{isotypic}
   v = \sum_{\lambda\in\Gamma_{n+1}(\pi,v)} v_\lambda 
   \end{equation}
   By Theorem~\ref{irreduciblerestriction} it follows for each $\lambda\in\Gamma_{n+1}(\pi,v)$ that $\langle \pi(U_n) v_\lambda \rangle$ is a $U_n$-irreducible subspace for which $v_\lambda$ is a highest-weight vector of weight $\lambda|_{\gh_n}$.  In other words, $v_\lambda$ is also a $U_n$-isotypic vector, so $\lambda|_{\gh_n} \in \Gamma_n(\pi,v)$.  Furthermore, since (\ref{isotypic}) is a decomposition of $v$ into $U_n$- and $U_{n+1}$-isotypic vectors, we see that for each $\mu\in\Gamma_n(\pi,v)$ there is $\lambda\in\Gamma_{n+1}(\pi,v)$ such that $\lambda|_{\gh_n} = \mu$. 
   
    In other words, if we consider all the highest weights of irreducible subrepresentations $\pi(U_n)$ and allow $n\in\N$ to vary, then the highest weights may be naturally arranged into a tree, as in Figure~\ref{weighttree}.

  % To illustrate the ``splitting'' behavior just described, we can arrange the $\mu_i^n$'s into a tree structure.  The weights $\mu^n_i$ for each $i$ and $n$ are the nodes of the tree.  We draw an arrow from $\mu^n_i$ to $\mu^{n+1}_j$ if $\mu^{n+1}_j |_{\ga_n} = \mu^n_i$.  Hence every node has exactly one parent node and at least one child node, which gives the structure of an infinite tree of weights. %In Figure~\ref{weighttree}, we show a typical highest-weight tree.

   \begin{figure}
      \caption{Example of a highest-weight tree} 
      \label{weighttree}
      \centerline{
         \xymatrix@R=1em{
            \Gamma_1(\pi,v)\ar@{}[d]|<<{||}             &  \Gamma_2(\pi,v)\ar@{}[d]|<<{||}            &  \Gamma_3(\pi,v)  \ar@{}[d]|<<{||}  &  \ar@{}[d]_{}="b"  \ldots \\
                                                        &                                             &  \mu^3_1         \ar[r]            &  \ar@{}[d]_{}="c"  \ldots \\
                                                        &  \mu^2_1 \ar[r]\ar[ur]                      &  \mu^3_2         \ar[r]            &  \ar@{}[d]_{}="d"  \ldots \\
            \mu^1_1\ar[ur]\ar[r]\ar[ddr]                &  \mu^2_2 \ar[r]                           &  \mu^3_3         \ar[r]            &  \ar@{}[d]_{}="e"  \ldots \\
                                                        &                                             &  \mu^3_4         \ar[r]            &  \ar@{}[d]_{}="f"  \ldots \\
                                                        &  \mu^2_3 \ar[ur]\ar[r]\ar[dr]               &  \mu^3_5         \ar[r]            &  \ar@{}[d]_{}="g"  \ldots \\
                                                        &                                             &  \mu^3_6         \ar[r]            &  \ar@{}[d]_{}="h"  \ldots \\      
            \mu_2^1\ar[r]                               &  \mu^2_4\ar[r]                              &  \mu^3_7         \ar[r]            &  \ldots 
           \ar"3,3";"c" \ar"5,3";"e" \ar"5,3";"f" 
         }
      }
   \end{figure}

   Next we prove (1).  First note that $V_\lambda = \langle\pi(U_\infty)v_\lambda\rangle$ is a $U_\infty$-invariant subspace of $\cH$ for each $\lambda\in\Gamma_n(\pi,v)$.    
   Suppose $m>n$, and write 
\[
u_\lambda = \sum_{\nu\in\Gamma_m(\pi,v) \text{ s.t. }\nu|_{\ga_n} = \lambda} v_\nu
\] 
for each $\lambda\in\Gamma_n(\pi,v)$.  
   Then $u_\lambda$ is a $U_n$-isotypic vector of type $\lambda$.  Because $v = \sum_{\nu\in\Gamma_m(\pi,v)} v_\nu$, we see that $v =  \sum_{\lambda\in\Gamma_n(\pi,v)} u_\lambda$ since every $U_m$-highest-weight vector $v_\nu$ appears as a summand in exactly one $u_\lambda$.  
   Since $v = \sum_{\lambda\in\Gamma_n(\pi,v)} v_\lambda$ is also a decomposition of $v$ into $U_n$-isotypic vectors, it follows that $v_\lambda = u_\lambda$ for each $\lambda\in\Gamma_n(\pi,v)$.  In particular, $v_\lambda$ is $M_m N_m$-invariant for all $m\geq n$.  It follows that $V_\lambda = \overline{\langle \pi(U_\infty) v_\lambda \rangle}$ gives a conical representation of $U_\infty$, proving (1).
   
 To prove (2), we need to show that $V_{\mu_1} \perp V_{\mu_2}$ for all $\mu_1\neq\mu_2$ in $\Gamma_n(\pi,v)$.
   It is sufficient to show that $V_{\mu_1}^m = \langle\pi(U_m)v_{\mu_1}\rangle$ and $V_{\mu_2}^m = \langle\pi(U_m)v_{\mu_2}\rangle$ are orthogonal for all $m$.     
     We apply Lemma~\ref{splittinglemma} to see that 
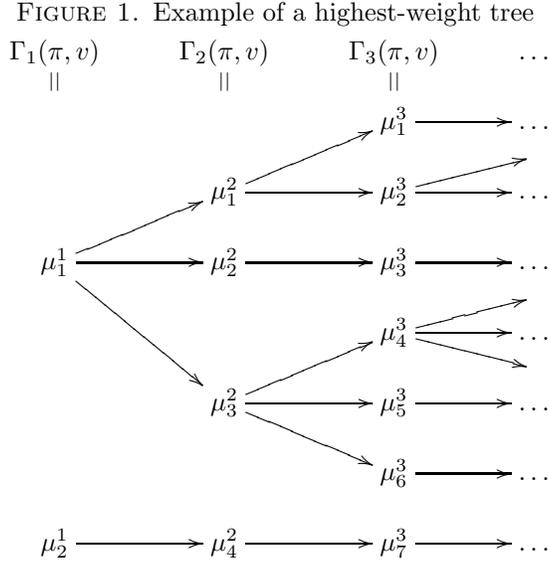
\[
   \langle\pi(U_m)v_\lambda\rangle = \bigoplus_{\nu\in \Gamma_m(\pi,v)\text{ s.t }\nu | \ga_n = \lambda} \langle\pi(U_m)v_\nu\rangle. 
\]  
It follows that $\langle\pi(U_m)v_{\mu_1}\rangle$ and $\langle\pi(U_m)v_{\mu_2}\rangle$ are orthogonal for all $m$ and hence that $V = \bigcup_m \langle\pi(U_m)v_{\mu_1}\rangle$ and $W = \bigcup_m \langle\pi(U_m)v_{\mu_2}\rangle$ are orthogonal $G$-invariant subspaces of $\cH$, proving (2).  Figure~\ref{weightreptree} demonstrates how the decomposition of $U_m$-representations matches the tree structure of the highest weights that was exhibited in Figure~\ref{weighttree}.

To prove (3), we assume that $\pi$ is irreducible. Suppose that there is $n$ such that $\#{\Gamma_n(\pi,v)}>1$ (that is, there is more than one $U_m$-highest weight in $\pi|_{U_m}$). Then (2) produces orthogonal, nonzero invariant subspaces of $\cH$, which contradicts the assumption that $\pi$ is irreducible.  Hence $\#{\Gamma_n(\pi,v)}=1$ for all $m$.
   
   For each $n$, let $\mu_n$ refer to the single element of $\Gamma_n(\pi,v)$.  From this it follows that $v$ is a $U_m$-highest-weight vector of weight $\mu_m$ for each $m$ with the property that $\mu_m |_{\ga_n} = \mu_n$ for $m\geq n$.  Furthermore, $V_n = \langle \pi(U_n)v \rangle$ is a $U_n$-irreducible subspace of $\cH$ for each $n$, and we can write $\displaystyle \pi = \varinjlim \pi_n$, where $\pi_n$ is the representation of $U_n$ on $V_n$ induced by~$\pi$.  Thus $\pi$ is a highest-weight representation and (3) is proved.
   
   To prove that $\dim \cH^{M_\infty N_\infty}=1$, suppose that  $v$ and $w$ are nonzero conical vectors for $\pi$ such that $v\perp w$.  Write $V_n = \langle \pi(U_n)v\rangle$ and $W_n = \langle \pi(U_n) w\rangle$ for each $n$.  We see that $V_n$ and $W_n$ are both equivalent to $\pi_{\mu_n}$ and have $v$ and $w$ as respective highest-weight vectors.  By Lemma~\ref{irreducible_primary}, it follows that $V_n\perp W_n$ for each $n$, contradicting the irreducibility of $\pi$.   
  
   \begin{figure}
      \label{weightreptree}
      \caption{Example of a decomposition of $\langle\pi(U_n)v\rangle$ into $U_n$-isotypic subspaces (direct sums are taken vertically)}
      \centerline{
         \xymatrix@R=1em{
            \langle\pi(U_1)v\rangle \ar@{}[d]|<<{||}                                   &  \langle\pi(U_2)v\rangle  \ar@{}[d]|<<{||}                              &  \langle\pi(U_3)v\rangle   \ar@{}[d]|<<{||}                             &  \ldots \\
                                                                                       &                                                                          &  \langle\pi(U_3)v^3_1\rangle\ar@{}[d]|{\bigoplus} \ar[ur]\ar[r]         &  \ldots \\
                                                                                       &  \langle\pi(U_2)v^2_1\rangle \ar@{}[d]|{\bigoplus} \ar[r]\ar[ur]         &  \langle\pi(U_3)v^3_2\rangle\ar@{}[d]|{\bigoplus} \ar[r]                &  \ldots \\
            \langle\pi(U_1)v_1^1\rangle \ar[ur]\ar[r]\ar[ddr] \ar@{}[dddd]|{\bigoplus} &  \langle\pi(U_2)v^2_2\rangle \ar@{}[dd]|{\bigoplus} \ar[r]               &  \langle\pi(U_3)v^3_3\rangle\ar@{}[d]|{\bigoplus} \ar[lr]\ar[r]         &  \ldots \\
                                                                                       &                                                                          &  \langle\pi(U_3)v^3_4\rangle\ar@{}[d]|{\bigoplus} \ar[r]                &  \ldots \\
                                                                                       &  \langle\pi(U_2)v^2_3\rangle \ar@{}[dd]|{\bigoplus} \ar[ur]\ar[r]\ar[dr] &  \langle\pi(U_3)v^3_5\rangle\ar@{}[d]|{\bigoplus} \ar[r]                &  \ldots \\
                                                                                       &                                                                          &  \langle\pi(U_3)v^3_6\rangle\ar@{}[d]|{\bigoplus} \ar[lr]\ar[r]  &  \ldots \\      
            \langle\pi(U_1)v_2^1\rangle \ar[r]                                         &  \langle\pi(U_2)v^2_4\rangle \ar[r]                                      &  \langle\pi(U_3)v^3_7\rangle                      \ar[r]                &  \ldots
         }
      }
   \end{figure}
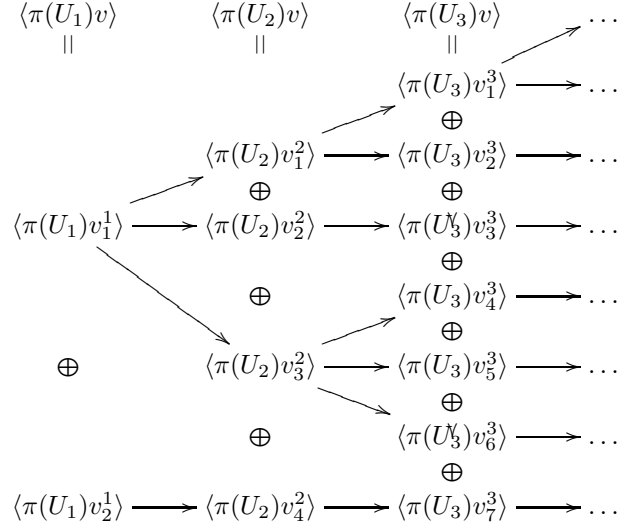

\end{proof}
 
Notice that the maps $p^{n+1}_n : \Gamma_{n+1}(\pi,v) \rightarrow \Gamma_n(\pi,v)$ defined by $p_n(\lambda) = \lambda|_{\ga_n}$ define a projective system.  We refer to the set $\Gamma(\pi,v) = \varprojlim \Gamma_n(\pi,v) \subseteq \Lambda^+(\gu_\infty,\ga_\infty)$ as the \textbf{highest-weight support} of $\pi$.  If we arrange the highest weights in a tree as in Figure~\ref{weighttree}, then we see that elements of $\Gamma(\pi,v)$ correspond to infinite paths.
 
We now examine the connection between conical and spherical representations of $G$.  Recall that for a finite-dimensional Riemannian symmetric space the irreducible finite-dimensional conical and spherical representations are identical.  The situation is much different for infinite-dimensional symmetric spaces, as the following corollary shows.

\begin{corollary}
If $\mathrm{Rank} (U_\infty/K_\infty) < \infty$, then a unitary irreducible representation is spherical if and only if it is conical.  If $\mathrm{Rank} (U_\infty/K_\infty) = \infty$, then no unitary irreducible representation is both spherical and conical.
\end{corollary}

\begin{proof}
By part (3) of Theorem~\ref{irreduciblehighestweightprop}, we see that the irreducible conical representations are precisely the highest-weight representations of $U_\infty$ with highest weight $\mu\in \Lambda^+(U_\infty, K_\infty)$.  By Theorem~4.5 in \cite{DOW}, it follows that these highest-weight representations of $U_\infty$ are spherical if and only if $\mathrm{Rank} (U_\infty/ K_\infty) < \infty$.  Furthermore, if $\mathrm{Rank} (U_\infty/K_\infty) < \infty$, then the spherical representations of $U_\infty$ are exhausted by the irreducible highest-weight representations.  
\end{proof}

\subsection{Highest-Weight Supports of Conical Representations}
In this section we explore some of the properties of the highest-weight trees associated with conical representations.  These trees form an invariant for conical representations, but as we shall see it is possible for two distinct conical representations to possess the same highest-weight tree.

First we show that the tree set of a conical representation is independent of the choice of conical vector:
\begin{theorem}
\label{uniquetree}
Let $(\pi,\cH)$ be a unitary conical representation of $U_\infty$. Then $\Gamma_n(\pi,v) = \Gamma_n(\pi,w)$ for any conical vectors $v,w$ in $\cH$.    
\end{theorem}

\begin{proof}
Suppose that both $v$ and $w$ are conical vectors in $\cH$ and that  $\mu\in\Gamma_n(\pi,w)$ but $\mu\notin\Gamma_n(\pi,v)$.  Write $w_\mu$ for the projection of $w$ onto the $\mu$-isotypic vectors in $\cH$.  Since $\mu\in\Gamma_n(\pi,w)$, it follows that $w_\mu\neq 0$. Define $W = \langle \pi(U_\infty) w_\mu \rangle$ and $V = \langle \pi(U_\infty) v \rangle$.  We claim that $W \perp V$, which will be a contradiction since $V$ is dense in $\cH$.

 Note that $W = \bigcup_{m\geq n} \langle \pi(U_m) w_\mu \rangle$ and $V = \bigcup_{m\geq n} \langle \pi(U_m) v \rangle$.  It is sufficient to show that $\langle \pi(U_m) w_\mu \rangle \perp  \langle \pi(U_m) v \rangle$ for $m\geq n$.  As before, we see from Lemma~\ref{irreducible_primary} and Theorem~\ref{irreduciblerestriction}  that
\[
\langle \pi(U_m) v \rangle = \bigoplus_{\lambda\in\Gamma_m(\pi,v)} \langle \pi(U_m) v_\lambda \rangle 
\cong_{U_m} \bigoplus_{\lambda\in\Gamma_m(\pi,v)} \cH_\lambda
\]
and 
\[
\langle \pi(U_m) w_\mu \rangle = \bigoplus_{\nu\in\Gamma_m^\mu(\pi,w) } \langle \pi(U_m) w_\nu \rangle 
\cong_{U_m} \bigoplus_{\nu\in\Gamma_m(\pi,w)} \cH_\nu,
\]
where $\Gamma_m^\mu(\pi,w) = \{ \nu\in\Gamma_m(\pi,w) \text{ s.t. } \nu|_{\ga_n} = \mu \}$.

Fix $m\geq n$.  Since $\mu\notin\Gamma_n(\pi,v)$, it follows that $\lambda|_{\ga_n}\neq \mu$ for all $\lambda\in\Gamma(\pi,v)$.  Thus $\Gamma_m(\pi,v)$ and $\Gamma_m^\mu(\pi,w)$  are disjoint.  This means that $\langle \pi(U_m)v_\lambda \rangle \perp \langle \pi(U_m) w_\nu \rangle$ for each $\lambda \in \Gamma_m(\pi,v)$ and $\nu\in\Gamma_m^\mu(\pi,w)$.  Hence $\langle \pi(U_m) v \rangle \perp \langle \pi(U_m) w_\mu \rangle$ for all $m$, as we wanted to show.
\end{proof}

From now on, we write $\Gamma_n(\pi) \equiv \Gamma_n(\pi,v)$ and $\Gamma(\pi) = \varprojlim \Gamma_n(\pi,v)$, where $v$ is any conical vector of a conical representation $\pi$ of $U_\infty$.  

\begin{corollary}
Let $(\pi, \cH)$ and $(\rho, \cK)$ be unitary conical representations of $(U_\infty,K_\infty)$.  If there is $n\in\N$ such that $\Gamma_n(\pi)\neq\Gamma_n(\rho)$, then $\pi\not\cong\rho$.
\end{corollary}

In particular, we have shown that having the same highest-weight tree is a necessary condition for two conical representations to be equivalent.  Later we will provide examples of inequivalent conical representations with the same highest-weight trees.  However, two conical representations with the same highest-weight trees are nonetheless \textit{almost} equivalent in a certain sense, as the following theorem shows. 

\begin{theorem}
\label{almostequivalent}
Let $(\pi, \cH)$ and $(\rho, \cK)$ be conical representations of $(U_\infty,K_\infty)$ with respective conical vectors $v$ and $w$ such that $\Gamma_n(\pi) = \Gamma_n(\rho)$ for each $n$.  Consider the algebraic subrepresentations $V = \langle \pi(U_\infty) v \rangle$ and $W = \langle \rho(U_\infty)w\rangle$ generated by the action of $U_\infty$ on $v$ and $w$.  Write $\pi_V$ and $\rho_W$ for the representations of $U_\infty$ given by restricting $\pi$ and $\rho$ to the dense invariant subspaces $V$ and $W$ of $\cH$ and $\cK$, respectively.  Then 
\begin{enumerate}
\item $\pi_V \cong \rho_W$
\item $\pi|_{U_n} \cong \rho|_{U_n}$ for each $n$.
\end{enumerate}
\end{theorem}

\begin{proof}
We begin by proving (1).  We claim that the map $L:V \rightarrow W$ induced by $\pi(g) v \mapsto \rho(g) w$ is a well-defined invertible $U_\infty$-intertwining operator.   

As before, write $V_m = \langle \pi(U_m) v \rangle$ and $W_m = \langle \pi(U_m) w \rangle$, so that $V = \bigcup_{m\geq n} V_m$ and $W = \bigcup_{m\geq n} W_m$.  Then   
\[
 V_m = \bigoplus_{\lambda\in\Gamma_m} \langle \pi(U_m) v_\lambda \rangle
 \cong_{U_m} \bigoplus_{\lambda\in\Gamma_m} \cH_\lambda
\]
and
\[
 W_m = \bigoplus_{\lambda\in\Gamma_m} \langle \rho(U_m) w_\lambda \rangle
 \cong_{U_m} \bigoplus_{\lambda\in\Gamma_m} \cH_\lambda,
\]
where $\Gamma_m = \Gamma_m(\pi) = \Gamma_m(\rho)$.  Thus $V_m$ and $W_m$ are $U_m$-isomorphic.  We must show that there is an invertible $U_m$-intertwining operator $L^m : V_m \rightarrow W_m$ that maps $v$ to $w$.

In %\marginpar{Too obvious?  Remove?}
fact, we note that for each $\lambda\in\Gamma_m$ there is a (not necessarily unitary) $U_m$-intertwining operator $L_\lambda : \langle\pi(U_m)v_\lambda\rangle \rightarrow \langle\rho(U_m)w_\lambda\rangle$ given by $\pi(g)v_\lambda\mapsto\rho(g)w_\lambda$.  We can then define
\[
L^m = \bigoplus_{\lambda\in\Gamma_m} L_\lambda :  V_m = \bigoplus_{\lambda\in\Gamma_m} \langle \pi(U_m) v_\lambda \rangle \rightarrow \bigoplus_{\lambda\in\Gamma_m} \langle \rho(U_m) w_\lambda \rangle = W_m.
\]
Hence $L^m v = L^m(\sum_{\lambda\in\Gamma_m}v_\lambda) = \sum_{\lambda\in\Gamma_m}w_\lambda = w$.  

Since $v$ and $w$ are cyclic vectors in $V_m$ and $W_m$, respectively, $L^m$ is in fact uniquely determined as an intertwining operator by the fact that it maps $v$ to $w$.  In particular, $L^m|V_n = L^n$ for all $n\leq m$.  Thus the family $\{L^m\}_{m\in\N}$ is a direct system of intertwining operators that induces a continuous $U_\infty$-intertwining operator 
$L:V = \varinjlim V_m \rightarrow \varinjlim W_m =W $
such that $Lv = w$.  

Next we prove (2).  Fix $n\in\N$.  Recursively define $\widetilde{V}_n = V_n$ and $\widetilde{V}_m = V_m \ominus V_{m-1}$ for $m > n$, where the orthogonal complement is taken with respect to the Hilbert space structure inherited by $V_n$ as a closed subspace of $\cH$.    
Notice that $\widetilde{V}_m$ is a finite-dimensional $U_n$-invariant subspace of $\cH$ for each $m\geq n$.  We define $U_n$-invariant spaces $\widetilde{W}_m\subseteq \cK$ for each $m\geq n$ in exactly the same way.  

Recall that $V_m$ and $W_m$ give equivalent representations of $U_n$ for each $m\geq n$ under the intertwining operator $L^m$. It follows that $\widetilde{V}_m = V_m \ominus V_{m-1}$ and $\widetilde{W}_m = W_m \ominus W_{m-1}$ are $U_n$-isomorphic for all $m > n$.  Note that
\[
\cH = \bigoplus_{m\geq n} \widetilde{V}_m \text{ and } \cK = \bigoplus_{m\geq n} \widetilde{W}_m, 
\]
 where the direct sums are orthogonal.  Since there is a unitary $U_n$ intertwining operator between $\widetilde{V}_m$ and $\widetilde{W}_m$ for all $m\geq n$, it follows that there is a unitary $U_n$-intertwining operator between $\cH$ and $\cK$.  
\end{proof}

\subsection{Smooth Conical Representations}
\noindent
Next we classify the smooth conical representations of $U_\infty$.  As mentioned previously, these are of interest because, by the generalized Weyl unitary trick,  they are precisely the conical representations which extend to smooth conical representations of the c-dual $G_\infty$. Our next theorem classifies the smooth representations.
\begin{theorem}
\label{holomorphic_decomposition}
Suppose that $(\pi, \cH)$ is a smooth conical representation of $U_\infty$.  Then $\pi$ decomposes into a direct sum of irreducible smooth highest-weight representations.
\end{theorem}

\begin{proof}
Let $v$ be a conical vector for $\pi$.  For each $U_n$, write 
\[v = \sum_{\lambda\in\Gamma_n(\pi)}v_\lambda\] 
as before.
   As in Section~\ref{propagationSection}, we recursively construct a countable basis $\{e_i\}_{n\in\N}$ for $\ga_\infty$ such that $\{e_1,\ldots,e_{r_n}\}$ is a basis for $\ga_n$ for each $n$.  For each $\lambda\in \ga_n^*$, write 
\[
||\lambda||_\infty = \max_{1\leq i \leq n} |\lambda(e_i)|.
\]
In fact, if $\lambda\in\Lambda^+(\gg_n,\ga_n)$ and $\lambda = \sum_{i=1}^{r_n} a_i e_i$, then we see from the data at the end of Section~\ref{highestWeightRepSection} that $a_i \leq a_j$ when $i\leq j$; thus $||\lambda||_\infty = a_{r_n}$.
   
  For each $\mu\in\Gamma_{n}(\pi)$, let $\Gamma^\mu_{n+1}(\pi) = \{\lambda\in \Gamma_{n+1}(\pi) : \lambda|_{\ga_n} = \mu\}$.  Hence we have $||\lambda||_\infty \geq ||\mu||_\infty$ for each $\lambda \in \Gamma^\mu_{n+1}$,.
 
   Now suppose that $\mu\in\Gamma_n(\pi)$ and that there are distinct weights $\lambda_1, \lambda_2\in \Gamma^\mu_{n+1}(\pi)$.  In this case we say that $\mu$ \textit{splits} with respect to $\pi$.  Because  $\lambda_1$ and $\lambda_2$ in $\Lambda^+(\gg_n,\ga_n)$ are by assumption distinct and agree on the first $r_n$ coordinates, we see that they must differ on a coordinate $i$ with $r_n<i\leq r_{n+1}$.  Since the coefficients of dominant weights form an increasing sequence, we see that either  $||\lambda_1||_\infty > ||\lambda_2||_\infty \geq ||\mu||_\infty$ or $||\lambda_2||_\infty > ||\lambda_1||_\infty \geq ||\mu||_\infty$  
   
   In other words, if a highest weight $\mu\in\Gamma_n(\pi)$ splits, then there is a $U_{n+1}$-highest weight in $\Gamma_{n+1}^\mu(\pi)$ with a coefficient which is strictly greater than all the coefficients in $\mu$.  It follows that unless there is a weight $\mu_n\in\Gamma_n(\pi)$ for some $n$ which does not split and such that  each $\lambda \in \Gamma_m^\mu(\pi)$ for any $m\geq n$ does not split, then we can repeat this process to obtain arbitrarily large coefficients of highest weights of representations appearing in $\pi$, contradicting Lemma~\ref{bounded_coefficients}.
   Hence, there is some highest weight $\mu\in\Gamma_n(\pi)$ such that, for each $m\geq n$, the vector $v_\mu$ is a $U_m$-highest-weight vector.  Thus $\overline{\langle \pi(U_\infty)v_\mu\rangle}$ gives a highest-weight representation of $U_\infty$.
 
We have shown that every smooth unitary conical representation possesses an irreducible subrepresentation and that the orthogonal complement is also a smooth unitary conical representation.  A standard Zorn's Lemma argument then shows that $\cH$ decomposes into an orthogonal direct sum of irreducible smooth conical representations.  
\end{proof}

It follows from Theorems~\ref{bounded_coefficients} and \ref{holomorphic_decomposition} that every smooth unitary conical representation $(\pi,\cH)$ of $U_\infty$ is an orthogonal direct sum of smooth highest-weight representations: 
\[
    \pi\cong \bigoplus_{i\in\cA} \pi_{\mu_i},
\]
where $\mu_i\in\Lambda^+ = \varprojlim \Lambda^+(\gu_n, \ga_n)$ for each $i\in\cA$.  We can write each highest weight $\mu_i$ in terms of fundamental weights as in Section~\ref{highestWeightRepSection}:
\[
\mu_i = \sum_{n=1}^{k_i} a_n^i \xi_i,
\]
where $a_n^i\in\N$ for each $i$ and $n$ (each $\mu_i$ is a finite sum over the fundamental weights is finite because $\pi_{\mu_i}$ is a smooth highest-weight representation).  By Theorem~\ref{bounded_coefficients}, the smoothness of $\pi$ is equivalent to the existence of a bound $M>0$ such that $\sum_{n=1}^{k_i} a_n^i <M$ for all $i\in\cA$.

\section{Disintegration of Conical Representations}
\label{disintegrationSection}
\noindent
If we remove the assumption in Theorem~\ref{holomorphic_decomposition} that the conical representation $(\pi, \cH)$ is smooth, then we can no longer be assured that $\pi$ has an irreducible subrepresentation.  However, we would still like to describe general conical representations in terms of the irreducible ones.  This sort of description is possible with a direct-integral decomposition.  

Recall that 
\[
\Lambda^+ \equiv \Lambda^+(\gu_\infty,\ga_\infty) \equiv \varprojlim \Lambda^+(\gu_n,\ga_n) \subseteq \ga_\infty^*
\]
 denotes the set of dominant integral weights for the root system $\Sigma(\gu_\infty,\ga_\infty)$. 
 %Write $\Lambda^+_n = \Lambda^+(U_n,K_n)$ for each $n$.  %If $n<m$, then we can embed $\Lambda^+_n$ into $\Lambda^+_m$ as in \cite{DOW} by identifying the fundamental weights of $(U_n,K_n)$ with the first $n$ fundamental weights of $(U_m,K_m)$ \marginpar{Explain in more detail?}.  Similarly, the restriction from $\ga_m$ to $\ga_n$ gives a natural surjection from $\Lambda_m$ onto $\Lambda_n$ if $m>n$.  These maps allow us to identify   
 %That is, $\Lambda^+$ consists of $\lambda$ in $i\ga_\infty' = \varprojlim i\ga_n$ such that $\lambda|_{\ga_n}$ is dominant and integral for every $n$ \marginpar{This requires a consistent choice of positive systems, etc!}.  
We start by putting a topology on $\Lambda^+$.  
Each lattice $\Lambda^+(\gu_n,\ga_n)$ carries the discrete topology.  We then consider the projective limit topology on $\Lambda^+$.  This topology is defined by a basis consisting of the cylinder sets $B_\lambda  =  \{ \mu\in\Lambda^+ | \mu_{|\ga_n} = \lambda\}$, where $\lambda$ is a dominant integral weight on $\ga_n$. It is clear that $\Lambda^+$ is second-countable under this topology, since there are only countably many dominant integral weights on $i\ga_n$, for each fixed $n\in\N$, so that our basis of cylinder sets is a countable union of countable sets.  This projective-limit topology is also Hausdorff.

Next consider closed subsets $\Gamma$ of $\Lambda^+$ with the property that, for each $n\in\N$, we have $\Gamma\bigcap B_\lambda = \emptyset$ for all but finitely many $\lambda$ in $\Lambda^+_n$.  We will refer to such sets as \textbf{tree sets} because, as we shall soon see, they are in one-to-one correspondence with trees of a certain type.
We give each tree set $\Gamma$ the subspace topology.   Write $\Gamma^\lambda = B_\lambda \bigcap \Gamma = \{\mu\in\Gamma | \mu_{|{\ga_n}} = \lambda \}$ for each $n$ and each $\lambda \in \Lambda^+_n$.  We refer to these sets as \textbf{cylinder sets} for $\Gamma$.  If $\lambda\in\Lambda^+_n$ and $\Gamma^\lambda\neq 0$ (that is, there is $\mu\in\Gamma$ such that $\mu|_{\ga_n} = \lambda$), then we say that $\lambda$ is a \textbf{node} of the tree set $\Gamma$.  We write $\Gamma_n = \{\mu_{|\ga_n} | \mu\in \Gamma\}$ for the set of all nodes of $\Gamma$ that lie in $\Lambda^+_n$.   One quickly sees that $\Gamma = \varprojlim \Gamma_n$ and that $\Gamma$ has the corresponding projective limit topology.  Thus $\Gamma$ is a projective limit of compact topological spaces, from whih it follows that that $\Gamma$ is compact.

It may be readily seen that if $\pi$ is a conical representation of $U_\infty$, then the highest-weight tree $\Gamma(\pi) \subseteq \Lambda^+$ is a tree set.  This follows because each set $\Gamma_n(\pi)\subseteq \Lambda^+(\gu_n,\ga_n)$ is finite.

We spend a few moments explaining our tree-centric choice of terminology.  For each tree set $\Gamma$, we can construct a tree as follows.  Each element of $\Gamma_n$ for each $n\in\N$ forms a node of the tree.  Draw an edge from a node $\lambda$ in $\Gamma_n$ to a node $\mu$ in $\Gamma_{n+1}$ if $\mu|_{\ga_n} = \lambda$.  There is a correspondence between infinite paths in this tree and elements of $\Gamma$. Each infinite path $\{\lambda_n\in\Gamma_n\}_{n\in\N}$ of nodes of the tree defines a dominant weight $\lambda\in\Lambda^+$, since $\lambda_m|_{\ga_n} = \lambda_n$ for $m>n$.  Because $\Gamma$ is closed in the projective limit topology on $\Lambda^+$, it follows that $\lambda\in\Gamma$.  Similarly, each dominant weight $\lambda$ in $\Gamma$ defines a path $\{\lambda|_{\ga_n}\in\Gamma_n\}_{n\in\N}$ in the tree.  Hence, if $\lambda$ is a node of $\Gamma$, then the cylinder set $\Gamma^\lambda$ corresponds to the set of all infinite paths in the tree which pass through the node $\lambda$.

As usual, we give each tree set $\Gamma \subseteq \Lambda^+$ the Borel $\sigma$-algebra, which is generated by node sets.  We can use $\Gamma$ to define a measurable family of Hilbert spaces $\lambda\mapsto\cH_\lambda$  over $\lambda\in\Gamma$. For each $\lambda\in\Gamma$, consider the representation $(\pi_\lambda,\cH_\lambda)$ of $U_\infty$ with highest-weight $\lambda$. For each such representation, pick out a unit highest-weight vector $v_\lambda\in\cH_\lambda$.  

To tie these Hilbert spaces together in a measurable way, we consider the family $\{s_g|g\in U_\infty\}$ of maps $s_g:\Gamma \rightarrow \dot{\bigcup}_{\lambda \in \Gamma} \cH_\lambda$ given by $s_g(\lambda) = \pi_\lambda(g)v_\lambda$.  Now choose a countable dense subset $E\subseteq U_\infty$ (recall that $U_\infty = \varinjlim U_n$ is separable) and consider the countable family
\[
 \{s_g|g\in E\}
\]
of sections.
  We shall use this family as a measurable frame for our family of Hilbert spaces.  Hence, we need to show that 
\begin{equation}
   \label{frameinnerproduct}
   \lambda \mapsto \langle s_g(\lambda), s_h(\lambda)\rangle = \langle \pi_\lambda(g)v_\lambda,\pi_\lambda(h)v_\lambda\rangle 
\end{equation}
is $\gB$-measurable for each $g, h \in E$.  Suppose that $g,h\in U_n$ for some $n\in\N$.  Then the representation of $U_n$ on $\langle\pi_\lambda(U_n)v_\lambda\rangle$ is equivalent to $\pi_{\lambda|_{\ga_n}}$ for each $\lambda$.  Thus the map in (\ref{frameinnerproduct}) is constant on each node set $\Gamma^{\lambda|_{\ga_n}}$ where $\lambda\in\Gamma$ and is hence $\gB$-measurable.  Finally, note that $\langle \{ s_g(\lambda) = \pi_\lambda(g)v_\lambda |g\in E \} \rangle$ is dense in $\cH_\lambda$ since $\pi_\lambda$ is irreducible and $E$ is dense in $U_\infty$.  Thus, $\lambda \mapsto \cH_\lambda$ is a measurable field of Hilbert spaces.

Next, we note that $s_g$ is a measurable section for all $g\in U_\infty$.  In fact, every $g\in U_\infty$ is a limit of a sequence $\{g_i\}_{i\in N} \subseteq E$.  Hence, we have that 
\[
\lambda\mapsto \langle s_g(\lambda),s_h(\lambda) \rangle  = \lim_{i\rightarrow \infty} \langle s_{g_i}(\lambda), s_h(\lambda) \rangle 
\]
is a measurable function for all $h\in E$, so that $s_g$ is a measurable section.

In order to construct a direct integral of representations $(\pi_\lambda, \cH_\lambda)$ over $\lambda\in\Gamma$, we still need a suitable choice of measure on $(\Gamma,\gB)$.  In particular, we need to choose a finite measure whose support is all of $\Gamma$ (we will refer to such measures as having {\bf full support}).

Given a finite Borel measure $\mu$ on $\Gamma$ of full support, we may consider the direct integral $\cH = \int^\oplus_\Gamma \cH_\mu d\mu(\lambda)$.  Elements of this direct integral consist of measurable sections $x : \lambda \mapsto x(\lambda)$ of the field $\lambda\mapsto\cH_\lambda$ such that the norm given by $||x||^2 = \int_\Gamma ||x(\lambda)||^2_{\cH_\lambda} d\mu(\lambda)$ is finite. 

Our next task is to show that $\lambda \rightarrow \pi_\lambda$ is a $\mu$-measurable family of representations.  Let $x\in\cH$, and fix $g$ in $U_\infty$.  We need to show that $\lambda \stackrel{\pi(x)}{\mapsto} \pi_\lambda(g)x(\lambda)$ is in $\cH$. Now 
\begin{eqnarray*}
\lambda \mapsto \langle \pi_\lambda(g)x(\lambda), s_h(\lambda)\rangle & = & \langle \pi_\lambda(g)x(\lambda), \pi_\lambda(h)v_\lambda\rangle \\
          & = & \langle x(\lambda), \pi_\lambda(g^{-1}h)v_\lambda\rangle \\
          & = & \langle x(\lambda), s_{g^{-1}h}(\lambda) \rangle 
\end{eqnarray*}
is measurable for all $h$ in $U_\infty$ since $x$ is a measurable section of $\lambda \mapsto \cH_\lambda$.  Thus $\lambda \stackrel{\pi(g)x}{\mapsto} \pi_\lambda(g)x(\lambda)$ is a measurable section of $\lambda \mapsto \cH_\lambda$.  Furthermore, since each $\pi_\lambda$ is unitary, it follows  that $||\pi(g)x||_\cH=||x||_\cH < \infty$ .  Hence $\pi = \int^\bigoplus_\Gamma \pi_\lambda d\mu(\lambda)$ is a unitary representation of $U_\infty$.  Our next task is to show that $\pi$ is conical and classify all of its conical vectors.

\begin{theorem}
%\marginpar{Is this proposition a special case of 8.8.2 (page 181) in the Dixmier $C^*$-Algebras book?}
\label{conicalvectorsexhausted}
Let $\Gamma$ be a tree set and let $\mu$ be a finite Borel measure of full support on $\Gamma$.  Consider the representation 
\[
(\pi,\cH) \equiv \left(\int^\bigoplus_\Gamma \pi_\lambda d\mu(\lambda), \int^\bigoplus_\Gamma \cH_\lambda d\mu(\lambda)\right)
\]
and suppose that $w$ is any nonzero vector in $\cH$.  Then $w$ generates a unitary conical representation of $U_\infty$ if and only if there is $f\in L^2(\Gamma,\mu)$ such that $w = \int_\Gamma^\oplus f(\lambda)v_\lambda d\mu(\lambda)$. 

In particular, $\pi$ is a conical representation with conical vector $v=\int_\Gamma^\oplus v_\lambda d\mu(\lambda)$.
\end{theorem}

\begin{proof}
($\Rightarrow$) Suppose that $w$ is a conical vector for a subrepresentation of $\pi$ and fix $n$ in $\N$.  Because conical representations are by definition locally finite, we see that $V_n \equiv \langle\pi(U_n)w\rangle$ is finite-dimensional, say with dimension~$d$.  We must show that $w(\lambda)$ is a conical vector in $\cH_\lambda$ for almost all $\lambda\in\Gamma$.  Our first task is to show that $V_n(\lambda) = \langle \pi(U_n)w(\lambda)\rangle$ is finite-dimensional for almost all $\lambda\in\Gamma$. 

Write $d = \dim V_n$. Fix an orthonormal basis $w_1,\ldots w_d$ for $V_n$ and write \[W(\lambda) = \langle w_1(\lambda),\ldots w_d(\lambda)\rangle.\]  We will show that $W(\lambda) = V_n(\lambda)$ (and hence $\operatorname{dim} V_n(\lambda) \leq d$) for almost all $\lambda$, and it will follow that $\dim V_n(\lambda) \leq d$ for almost all $\lambda$.  
Apply a Gram-Schmidt orthonormalization process to the collection $w_1(\lambda),\ldots,w_d(\lambda)$ for each $\lambda$.  We then obtain a collection $\widetilde{w}_1(\lambda),\ldots,\widetilde{w}_d(\lambda)$ with the property that $\langle \widetilde{w}_i(\lambda), \widetilde{w}_j(\lambda)\rangle = 0$ for $i\neq j$ and $\langle \widetilde{w}_i(\lambda),\widetilde{w}_i(\lambda) \rangle \in \{0,1\}$.  %\marginpar{Is this readable?}
 One can show that $\lambda \mapsto \widetilde{w}_i(\lambda)$ is measurable and thus that $\widetilde{w}_i\in\cH$ for each $i$.

Now $W(\lambda) = V_n(\lambda)$ if and only if $\pi(g)w(\lambda) \in W(\lambda)$ for all $g$ in $U_\infty$.  Choose a countable dense subset $\{g_n\}_{n\in\N}$ in $U_\infty$ (one notes that $U_\infty$ is separable because it is a countable direct union of separable spaces). By the strong continuity of $\pi$, we see that $W(\lambda) = V_n(\lambda)$ if and only if $\pi(g_m)w(\lambda) \in W(\lambda)$ for all $m$ in $\N$ (recall that $W(\lambda)$ is closed because it is finite-dimensional).  
In turn, this happens exactly when $\pi(g_m)w(\lambda)$ is equal to its orthogonal projection onto $W(\lambda)$.  In other words, $W(\lambda) = V_n(\lambda)$ if and only if $F_m(\lambda) = 0$ for all $m\in\N$, where $F_m$ is the non-negative measurable function on $\Gamma$ defined by 
\begin{equation*}
F_m: \lambda \mapsto ||\pi(g_m)w(\lambda)||^2 - \sum_{i=1}^d | \langle \pi(g_m)w(\lambda),\widetilde{w}_i(\lambda) \rangle|^2.
\end{equation*}
for all $m\in\N$. 

Write $A = \{\lambda\in\Gamma | W(\lambda)\neq V_n(\lambda)\}$ and $A_m = \{\lambda\in\Gamma | \pi(g_m) w(\lambda) \notin W(\lambda)\}$.  Then $A = \bigcup_{m\in\N} A_m$.  Furthermore, $A_m$ is measurable for each $m$ since $A_m = F_m^{-1}(0)$ and $F_m$ is a measurable function.

   Suppose that it is not true that $W(\lambda) = V_n(\lambda)$ for almost all $\lambda$ in $\Gamma$.  Then $\mu(A)>0$.  Since $A = \bigcup_{m\in\N} A_m$, it follows that $\mu(A_m) > 0$ for some $m$.  Since $\pi(g_m)w(\lambda)\notin W(\lambda)$ for all $\lambda\in A_m$, we see that $\pi(g_m)w\notin\langle w_1,\ldots,w_d \rangle$, which contradicts the assumption that $w_1,\ldots, w_d$ is a basis for $V_n = \langle \pi(g_m) w \rangle$.  Therefore, $W(\lambda) = V_n(\lambda)$  (and, in particular, $\operatorname{dim}V_n(\lambda)\leq d$) for almost all $\lambda$.   
In particular, $w(\lambda)$ is $U_n$-finite for almost all $\lambda\in\Gamma$.

Fix $n\in\N$.  Since $\pi(M_n) w = w$, it follows that $\pi(M_n) w(\lambda) = w(\lambda)$ for almost all $\lambda$. Next, $\pi(\gn_n)w = w$  because $\pi(N_n)w=w$.  In fact, $\pi(X)w = \int_\Gamma^\oplus \pi(X)w(\lambda)d\mu(\lambda)$ for $X \in \gu_n^\C$ by \cite{A}.  Thus $\pi(\gn_n)w(\lambda) = w(\lambda)$ for almost all $\lambda$, from which it follows that $\pi(N_n)w(\lambda) = w(\lambda)$ for almost all $\lambda$.

Since $\pi(M_n N_n) w(\lambda) = w(\lambda)$ for all $n$ and almost all $\lambda\in\Gamma$, it follows from part (4) of  Theorem~\ref{irreduciblehighestweightprop} that for almost all $\lambda$ there is $f(\lambda)\in\C$ such that $w(\lambda) = f(\lambda)v_\lambda$.  Since $\lambda \mapsto f(\lambda) = \langle w(\lambda), v_\lambda \rangle$ is measurable and 
\[
||f||^2 = \int_\Gamma |f(\lambda)|^2 d\mu(\lambda) = \int_\Gamma ||w(\lambda)||^2 d\mu(\lambda) = ||w||^2,
\]
we see that $f\in L^2(\Gamma,\mu)$, as was to be shown.

($\Leftarrow$) Now suppose that $w = \int^\oplus_\Gamma f(\lambda)v_\lambda d\mu(\lambda)$, where $f\in L^2(\Gamma,\mu)$.  We show that $w$ generates a conical representation of $U_\infty$ with highest-weight support $\operatorname{ess\ supp} f$.

Consider $V_n = \langle \pi(U_n) w \rangle$.  We will show that $V_n$ is finite-dimensional.  As before, 
\[
\pi \cong \bigoplus_{\mu\in\Gamma_n} \left( \int^\bigoplus_{\Gamma^\mu} \pi_\lambda d\mu(\lambda) \right).
\]
Write $\displaystyle w = \sum_{\mu\in\Gamma_n} w_\mu$, where $w_\mu = 1_{N_\mu}w \in \int^\bigoplus_{\Gamma^\mu} \cH_\lambda d\mu(\lambda)\subseteq\cH_\Gamma$
for each $\mu$.  

Of course, if $f|_{\Gamma^\mu} =0$, then $w_\mu = 0$.  On the other hand, we claim that if $f|_{\Gamma^\mu}\neq 0$, then $\langle \pi(U_n) w_\mu \rangle \cong_{U_n} \pi_\mu$.  
In fact, 
\[
\sum_{i=1}^k c_i\pi(g_i)w_\mu = \int_{\Gamma^\mu} \sum_{i=1}^k c_i\pi(g_i)f(\lambda)v_\lambda d\mu(\lambda).
\]
where $c_i\in\C$ and $g_i\in U_n$.
Fix $\lambda\in \Gamma^\mu$ such that $f(\lambda)\neq 0$.  Since $\lambda|_{\ga_n}=\mu$, we see % by the definition of $\pi_\lambda = \varinjlim \pi_{\lambda|_{\ga_m}}$ 
that $\langle\pi(U_n)f(\lambda)v_\lambda\rangle$ is $U_n$-isomorphic to $\pi_\mu$.%  by identifying $w(\lambda)=f(\lambda)v_\lambda$ with a highest-weight vector $v_\mu$ for $\pi_\mu$.

Now $\sum_{i=1}^k c_i\pi(g_i)w_\mu  = 0$ in $\cH$ if and only if $\sum_{i=1}^k c_i\pi(g_i)f(\lambda)v_\lambda = 0$ in $\cH_\lambda$ for $\mu$-almost all $\lambda$ in $\Gamma^\mu$.  For any $\lambda$ in $\Gamma^\mu$ such that $f(\lambda) = 0$, it follows automatically that $\sum_{i=1}^k c_i\pi(g_i)f(\lambda)v_\lambda = 0$.  But for any fixed $\lambda$ in $\Gamma^\mu$ such that $f(\lambda)\neq 0$, we see that $\sum_{i=1}^k c_i\pi(g_i)f(\lambda)v_\lambda = 0$ in $\cH_\lambda$ if and only if $\sum_{i=1}^k c_i\pi(g_i)v_\mu = 0$ in $\cH_\mu$.  

Since $f$ is not almost-everywhere zero on $\Gamma^\mu$, we see that  $\sum_{i=1}^k c_i\pi(g_i)w_\mu  = 0$ in $\cH$ if and only if  $\sum_{i=1}^k c_i\pi(g_i)v_\mu = 0$ in $\cH_\mu$.  Hence there is an injective $U_n$-intertwining operator $L:\langle\pi(U_n)w_\mu\rangle\rightarrow \cH_\mu$ with the property that $Lw_\mu = v_\mu$.  Since $\pi_\mu$ is irreducible, it follows that $\langle \pi(U_n) w_\mu \rangle \cong_{U_n} \pi_\mu$, as we wanted to show.

It follows from Lemma~\ref{splittinglemma} that 
\[
\langle \pi(U_n) w \rangle \cong_{U_n} \bigoplus_{\mu\in\Gamma_n \text{ s.t. } w_\mu\neq 0} \langle \pi(U_n) w_\mu \rangle.
\]
Furthermore, since $w = \sum_{\mu\in\Gamma_n} w_\mu$ and each $w_\mu$ is $M_n N_n$-invariant, we see that $w$ is $M_nN_n$-invariant.  Since this holds for all $n$, it follows that $w$ generates a conical subrepresentation of $\pi$.  The fact that this subrepresentation has highest-weight support $\operatorname{ess\ supp} f$ follows from the fact that $w_\mu = 0$ if and only if $f|_{\Gamma^\mu} = 0$ (recall that $w_\mu$ is the projection of $w$ onto the $\mu$-isotypic vectors in $\cH$).
\end{proof}

In fact, for each conical vector identified by the previous theorem, it is possible to describe the subrepresentation that it generates.  We first remind the reader that the \textbf{essential support} of a function $f:\Gamma\rightarrow \C$ is defined to be the complement in $\Gamma$ of the union of all open sets on which $f$ vanishes $\mu$-almost everywhere.  That is, $\operatorname{ess\ supp} f = \Gamma \backslash \bigcup \{A \subseteq \Gamma | A \text{ is open and } f|A = 0 \text{ a.e.}\}$.  

\begin{theorem}
\label{theoremConicalSpan}
 As before, let $\Gamma$ be a tree set, let $\mu$ be a finite Borel measure of full support on $\Gamma$, and consider the representation 
$(\pi,\cH) \equiv \left(\int^\bigoplus_\Gamma \pi_\lambda d\mu(\lambda), \int^\bigoplus_\Gamma \cH_\lambda d\mu(\lambda)\right)$.

Suppose that $f\in L^2(\Gamma,\mu)$, and put $w = \int_\Gamma^\oplus f(\lambda)v_\lambda d\mu(\lambda)$.  Then the conical representation generated by $w$ has highest-weight support equal to $\operatorname{ess\ supp} f$ and
\[
\overline{\langle \pi(U_\infty) w \rangle} = \int^\bigoplus_{\Gamma\backslash f^{-1}(0)} \cH_\lambda d\mu(\lambda)
\]
\end{theorem}

\begin{proof}
It suffices to show that 
\[
\overline{\langle \pi(U_\infty) w \rangle}^\perp = \int^\bigoplus_{f^{-1}(0)} \cH_\lambda d\mu(\lambda).
\]
One direction of containment is clear: for any  $x\in\overline{\langle \pi(U_\infty) w \rangle}$, we see that $x(\lambda) = 0$ for almost all $\lambda$ such that $f(\lambda)=0$ (since $w(\lambda)\equiv f(\lambda)v_\lambda=0$ if and only if $f(\lambda)=0$).  Hence, if $y\in\cH$ such that $y|_{\Gamma\backslash f^{-1}(0)} = 0$, then $\langle x,y \rangle = \int_\Gamma \langle x(\lambda),y(\lambda) \rangle d\mu(\lambda) = 0$.  In other words, $\int^\bigoplus_{f^{-1}(0)} \cH_\lambda d\mu(\lambda) \subseteq \overline{\langle \pi(U_\infty) w \rangle}^\perp$.

Now suppose that $x\perp\overline{\langle\pi(U_\infty)w\rangle}$.  In Lemma~\ref{lemmaMultiplication}, we will show that $hw\in\overline{\langle \pi(U_\infty) w \rangle}$ for all $h\in L^\infty(\Gamma,\mu)$. We define $h\in L^\infty(\Gamma,\mu)$ by \[h(\lambda) = \frac{\overline{\langle x(\lambda),\pi_\lambda(g)f(\lambda)v_\lambda \rangle}}{|\langle x(\lambda),\pi_\lambda(g)f(\lambda)v_\lambda \rangle|}.\]    Then

\[
0   = \langle x,  \pi(g) hw\rangle 
    = \int_\Gamma |\langle x(\lambda),\pi_\lambda(g)f(\lambda)v_\lambda \rangle| 
 d\mu(\lambda).
\]
for all $g$.  Hence, for almost all $\lambda$, $\langle x(\lambda),\pi_\lambda(g)f(\lambda)v_\lambda \rangle=0$ for all $g\in U_\infty$.  It follows that, for almost all $\lambda$, either $x(\lambda)=0$ or $f(\lambda)=0$.  Hence, $x(\lambda)=0$ for almost all $\lambda$ such that $f(\lambda)\neq 0$.  In other words, $x\in \int^\bigoplus_{f^{-1}(0)} \cH_\lambda d\mu(\lambda)$, and we are therefore done.
\end{proof}

We now prove the lemma that we used in the proof of Theorem~\ref{theoremConicalSpan}:
\begin{lemma}
\label{lemmaMultiplication}
Suppose that $f\in L^2(\Gamma,\mu)$ and put $w = \int_\Gamma^\oplus f(\lambda)v_\lambda d\mu(\lambda)$.  Then  $hw\in\overline{\langle \pi(U_\infty) w \rangle}$ for all $h\in L^\infty(\Gamma,\mu)$. 
\end{lemma}

\begin{proof}
We begin by showing that $1_{\Gamma^\mu}w\in\langle\pi(U_\infty)w\rangle$ for every node set $\Gamma^\mu$.  As before, we choose $c_1,\ldots,c_d\in\C$ and $g_1,\ldots,g_d\in U_\infty$ such that 
$\sum_{i=1}^k c_i \pi_\mu(g_i) v_\mu = v_\mu$ and $\sum_{i=1}^k c_i \pi_\nu(g_i) v_\nu = 0$ for all $\nu\neq\mu$ in $\Gamma_n$.  We claim that $1_{\Gamma^\mu}w = \sum_{i=1}^k c_i\pi_\mu(g_i) w$.  If $f(\lambda) = 0$, then $w(\lambda) =0$ and hence equality holds automatically.  On the other hand, if $f(\lambda)\neq 0$, then recall that $\langle\pi(U_n) w\rangle$ is equivalent to $\pi_{\lambda|_{\ga_n}}$ by identifying $w(\lambda)=f(\lambda)v_\lambda$ with $v_{\lambda|_{\ga_n}}$.  Hence $\sum_{i=1}^k c_i \pi_\mu(g_i)v_\lambda = v_\mu$ if $\lambda|_{\ga_n} = \mu$ (i.e., if $\lambda\in \Gamma^\mu$) and $\sum_{i=1}^k c_i \pi_\mu(g_i)v_\lambda = 0$ otherwise.  Thus $1_{\Gamma^\mu}w = \sum_{i=1}^k c_i\pi_\mu(g_i) w$ and so $1_{\Gamma^\mu}w\in\langle\pi(U_\infty)w\rangle$.

%%%%%%%%%%%%%
Next we see that $1_A w\in \overline{\langle\pi(U_\infty)w\rangle}$ for all open sets $A$ in $\Gamma$.  Every open set $A$ can be written as a disjoint union $A = \bigcup_{i=1}^\infty N_i$ of node sets.  Write $A_n = \bigcup_{i=1}^n N_i$ for each $n$ and note that $1_{A_n} = \sum_{i=1}^k 1_{N_i}$ is in $\langle\pi(U_\infty)v\rangle$ by the previous paragraph.  One then sees that 
\[\textstyle
\int^\oplus_\Gamma 1_{A_n}(\lambda)f(\lambda) v_\lambda d\mu(\lambda) = 1_{A_n} w \rightarrow 1_A w = \int^\oplus_\Gamma 1_{A}(\lambda)f(\lambda) v_\lambda d\mu(\lambda)\] in $\cH$ since $1_{A_n} f\rightarrow 1_A f$ in $L^2(\Gamma,\mu)$.
  Thus $1_A v\in \overline{\langle\pi(U_\infty)v\rangle}$.

Next we show that $1_B v\in \overline{\langle\pi(U_\infty)v\rangle}$ for every Borel set $B$ in $\Gamma$.  This follows since
\[
\begin{split}
\mu(B) & = \inf \left\{\left.\mu\left(\bigcup_{i=1}^\infty F_i\right) \right| B\subseteq\bigcup_{i=1}^\infty F_i\text{ and } F_i\in\gF \right\} \\
       & = \inf \{\mu(A)|B\subseteq A \text{ and $A$ open}\}. \\
\end{split}
\]
Thus $1_Bf$ can be approximated in $L^2(\Gamma,\mu)$ by a sequence $1_{A_n}f$ given by open sets $A_n$, so that $1_{A_n}w\rightarrow 1_Bw$ in $\cH$.  Hence $1_B w\in \overline{\langle\pi(U_\infty)w\rangle}$.

Finally, note that if $h_n\rightarrow h$ in $L^\infty(\Gamma,\mu)$, then $h_n f\rightarrow hf$ in $L^2(\Gamma,\mu)$ and hence $h_n w\rightarrow hw$ in $\cH_\Gamma$.  Because the measurable simple functions are dense in $L^\infty(\Gamma,\mu)$ (recall that $\mu$ is a finite measure), we see that $hw\in\overline{\langle\pi(U_\infty)w\rangle}$ for all $h\in L^\infty(\Gamma,\mu)$.  

%%%%%%%%%%%%%

\end{proof}

Finally, we show that every unitary conical representation of $U_\infty$ disintegrates into a direct integral of highest-weight representations.

\begin{theorem}

\label{disintegration}
%\marginpar{Is this theorem a special case of 8.8.6 (page 187) in the Dixmier $C^*$-Algebras book?}
Suppose that $(\pi, \cH)$ is a unitary conical representation of $U_\infty$ and $w\in\cH\backslash\{0\}$ is a conical vector.  Then there is a unique Borel measure $\mu$ on its highest-weight support $\Gamma(\pi)$ such that
there is a unitary intertwining operator \[U: \cH \rightarrow \int_{\Gamma(\pi)}^\oplus \cH_\lambda d\mu(\lambda)\] defined by $U w = \int_{\Gamma(\pi)}^\oplus v_\lambda d\mu(\lambda)$.
\end{theorem}

\begin{proof}
Without loss of generality, suppose that $||w||=1$.  We begin by constructing a suitable measure $\mu$.  For each $\lambda$ in $\Gamma_n(\pi)$, define $\mu(\Gamma^\lambda) = ||w_\lambda||^2$.  Observe that $w_\lambda = \sum_{\nu\in\Gamma_m^\lambda} w_\nu$ and hence
\[
\mu(\Gamma^\lambda) = ||w_\lambda||^2 = \sum_{\nu\in\Gamma_m^\lambda(\pi)} ||w_\nu||^2 = \sum_{\nu\in\Gamma_m^\lambda(\pi)} \mu(\Gamma^\lambda).
\]
Similarly, 
\[
\sum_{\nu\in\Gamma_n(\pi)} \mu(\Gamma^\nu) = \sum_{\nu\in\Gamma_n(\pi)} ||w_\nu||^2 = ||w||^2 = 1
\]
Thus $\mu$ extends uniquely to a Borel measure on $\Gamma(\pi)$.  

Consider the representation
\(
(\widetilde{\pi},\widetilde{\cH}) \equiv \left(\int_{\lambda\in\Gamma(\pi)}^\oplus \pi_\lambda d\mu(\lambda), \int_{\lambda\in\Gamma(\cH)}^\oplus \cH_\lambda d\mu(\lambda)\right)
\)
and let  $\widetilde{w} \equiv \int_{\Gamma(\pi)} v_\lambda d\mu(\lambda)$.  Then $\widetilde{\pi}$ is conical with conical vector $\widetilde{w}$ and highest-weight support $\Gamma(\pi)$. We construct a unitary intertwining operator 
$U:\cH\rightarrow\widetilde{\cH}$ such that $Uw = \widetilde{w}$.  

By Theorem~\ref{almostequivalent} (i), there is a $U_\infty$-intertwining operator $L:\langle\pi(U_\infty)w\rangle \rightarrow \langle\widetilde{\pi}(U_\infty)\widetilde{w}\rangle$ given by $Lw = \widetilde{w}$. For each $n$ and each $\nu\in\Gamma_n(\pi)$,  $L$ restricts to an intertwining operator between $\langle\pi(U_n)w_\nu\rangle$ and $\langle\widetilde{\pi}(U_n)\widetilde{w}_\nu\rangle$ such that  $L(w_\nu) = \widetilde{w}_\nu$.  Furthermore,
\[
||\widetilde{w}_\nu||^2 = \int_{\Gamma^\nu} ||\widetilde{w}_\lambda||^2 d\mu(\lambda) = \int_{\Gamma^\nu}1d\mu(\lambda) =  \mu(\Gamma^\nu) = ||w_\nu||^2.
\] 
Hence, $L$ restricts to a unitary operator on $\langle\pi(U_n)w_\nu\rangle$ for every $n$ and every $\nu\in\Gamma_n(\pi)$. Because $\langle\pi(U_\infty)w_\nu\rangle$ and $\langle\pi(U_\infty)\widetilde{w}_\nu\rangle$ are dense in $\cH$ and $\widetilde{\cH}$, respectively, $L$ extends to a unitary intertwining operator from $\cH$ to $\widetilde{\cH}$.

The uniqueness of the measure follows from standard results from direct-integral theory (see \cite{Dix}, for instance).
%Now suppose that $\mu'$ is any Borel measure on $\Gamma(\pi)$ such that the representation
%\(
%(\pi',\cH') \equiv \left(\int_{\lambda\in\Gamma(\pi)}^\oplus \pi_\lambda d\mu'(\lambda), \int_{\lambda\in\Gamma(\cH)}^\oplus \cH_\lambda d\mu'(\lambda)\right)
%\)
%is equivalent to $(\pi,\cH)$ via a unitary intertwining operator $U:\cH\rightarrow\cH'$ such that $Uw = w'$, where $w'=\int_{\Gamma(\pi)} v_\lambda d\mu'(\lambda)$.  Then $Uw_\nu = w'_\nu$ for all $\nu\in\Gamma_n(\pi)$ and all $n\in\N$ by Theorem~\ref{almostequivalent}.  In particular, $||w_\nu|| = ||w'_\nu||$ and so we have that 
%\[
%\mu'(\Gamma^\nu) = \int_{\Gamma^\nu} ||v_\lambda||^2 d\mu'(\lambda) = ||w'_\nu||^2 = ||w_\nu||^2 = \mu(\Gamma^\nu).
%\]
%Since $\mu$ and $\mu'$ agree on all node sets, it follows that $\mu = \mu'$.
\end{proof}

A corollary of Theorems~\ref{conicalvectorsexhausted} and~\ref{disintegration} is that unitary conical representations of $U_\infty$ are multiplicity-free (and thus of type I). 

\begin{corollary}
Every unitary conical representation of $U_\infty$ is multiplicity-free.
%\marginpar{Be slightly careful here because we are dealing with the representation of a group and not directly that of a $C^*$-algebra; but can we just take the $C^*$-algebra generated by $\pi(U_\infty)$ and apply the same machinery?}
\end{corollary}
\begin{proof}
Let
\(
(\pi,\cH) \equiv \left(\int\nolimits^\bigoplus_\Gamma \pi_\lambda d\mu(\lambda), \int\nolimits^\bigoplus_\Gamma \cH_\lambda d\mu(\lambda)\right)
\)
 be a conical representation and suppose that $L:\cH\rightarrow\cH$ is a $U_\infty$-intertwining operator.  Consider the conical vector $v = \int_\Gamma^\oplus v_\lambda d\mu(\lambda)$.  Then $Lv$ is a conical vector for a subrepresentation of $\pi$ and can thus be written $Lv = fv$ for some $f\in L^2(\Gamma,\mu)$.  It follows that 
\[
L(\pi(g)v) = \pi(g) (fv) = \int^\oplus_\Gamma \pi(g)f(\lambda) v_\lambda d\mu(\lambda) = f\pi(g)v
\]
for all $g\in U_\infty$ and hence $Ly=fy$ for all $y\in \cH$. In other words, intertwining operators for $\pi$ may be identified with multiplier operators, and thus the ring of intertwining operators for $\pi$ is commutative.   Hence $\pi$ is multiplicity-free.
\end{proof}

As promised before, we now show that there are typically a very large number of inequivalent conical representations of $U_\infty$ with a given highest-weight support $\Gamma$.  By Theorem~\ref{conicalvectorsexhausted}, this problem is equivalent to finding Borel measures with full support on $\Gamma$ that are absolutely discontinuous with respect to each other. 

One way to define a measure $\mu_{\text{rec}}$ on $\Gamma$ is as follows: we assign $\mu_\text{rec}(\Gamma^\nu) = \frac{1}{\#\Gamma_1}$ for each $\nu$ in $\Gamma_1$.  Then, for $\nu\in\Gamma_{n+1}$, recursively define $\mu_\text{rec}(\Gamma^\nu) = \frac{1}{\#\Gamma_{n+1}^\lambda}\mu_\text{rec}(\Gamma^\lambda)$ if $\lambda\in\Gamma_n$ and $\nu\in\Gamma^\lambda$.  One can see quite easily that the atoms of $\mu_\mathrm{rec}$ are precisely the isolated points of the topological space $\Gamma$; all other singleton sets will have measure zero.
  
We now show that for any point $x$ in $\Gamma$ we can construct a Borel measure $\mu_x$ of full support on $\Gamma$ whose atoms are precisely the isolated points of $\Gamma$ and $x$.  Thus, if $x\neq y$ are non-isolated points in $\Gamma$, then $\mu_x$, $\mu_y$, and $\mu_\mathrm{rec}$ lie in distinct measure classes since their null sets do not agree: 
\[
\begin{array}{ll}
\mu_x(\{x\}) > 0, & \mu_x(\{y\}) = 0 \\
\mu_y(\{x\}) = 0, & \mu_y(\{y\}) > 0 \\
\mu_\mathrm{rec}(\{x\}) = 0, & \mu_\mathrm{rec}(\{y\}) = 0
\end{array}
\]

There are many ways to construct $\mu_x$ given $x\in\Gamma$, but we shall use the following method, which involves a simple modification to the recursively uniform measure.  For $\lambda\in\Gamma_1$, define $\mu_x(\Gamma^\lambda)=\frac{3}{4}$ if $x|_{\ga_n} = \lambda$ and $\mu_x(\Gamma^\lambda) = \left(\frac{1}{\#\Gamma_1-1}\right)\frac{1}{4}$ otherwise.  
Next suppose that $\mu_x(\Gamma^\nu)$ has been defined for all $\nu\in\Gamma_{n}$.  For $\lambda\in\Gamma_{n+1}$, we define
\[
\mu_x(\Gamma^\lambda) = \left\lbrace
                    \begin{array}{ll}
                      \frac{1}{2}+\frac{1}{2^{n+1}} 
                            & \textrm{if } x\in\Gamma^\lambda\\
                      \left(\frac{1}{2}-\frac{1}{2^{n+1}}\right)
                            \frac{1}{(\#\Gamma_n^{\lambda|_{\ga_n}})-1} 
                            & \textrm{if } x\notin \Gamma^\lambda
                            \textrm{ and } x \in \Gamma^{\lambda|_{\ga_n}}\\
                      \frac{1}{\#\Gamma_n^{\lambda|_{\ga_n}}}\mu(\Gamma^{\lambda|_{\ga_n}}) & \textrm{otherwise,}
                    \end{array}
                   \right.
\]
where, as before, $\Gamma_n^\nu = \{\gamma\in \Gamma_{n} |\textrm{ } \gamma|_{\ga_{n}}=\nu\}$. We have thus recursively defined a countably additive Borel measure $\mu_x$ on $\Gamma$.  Note that $\mu_x$ has full support on $\Gamma$ because $\mu_x(\Gamma^\lambda)>0$ for every open basis set $\Gamma^\lambda\subseteq \Gamma$.   Furthermore, one can easily check that $\mu_x(\{x\}) = \frac{1}{2}$ and that $\mu_x(\{y\}) = 0$ if $y\neq x$ and $y$ is not an isolated point of $\Gamma$.

\section{Closing Remarks and Further Research}
We have managed to prove several results for the unitary conical representations of $U_\infty$, including the classification of unitary smooth conical representations, which generalize the finite-dimensional conical representations of finite-dimensional symmetric spaces.  However, the question remains of whether it is possible to construct unitary conical representations of $G_\infty$.  The most likely approach would be to construct a sort of unitary spherical principal series representation, perhaps by a direct limit of unitary principal series representations.  See also \cite{W2012} for one approach to constructing an analogue of the principal series for direct-limit groups.

Several questions about harmonic analysis on the symmetric space $G_\infty/K_\infty$ and $G_\infty/M_\infty N_\infty$ remain.  While neither of these infinite-dimensional spaces possess $G_\infty$-invariant measures, there is a possibility of constructing $G_\infty$-invariant measures on larger spaces.  We briefly overview this construction now.

Consider a direct system $\{G_n\}_{n\in\N}$ of Lie groups and suppose that there are measurable (not necessarily continuous) projections $p_n:G_{n+1} \rightarrow G_n$ such that $p_n$ is $G_n$-equivariant and $p_n(g) = g$ for $g\in G_n$.  In other words, one has a projective system of $\sigma$-algebras dual to the direct system of groups.  The resulting projective-limit space $\overline{G_\infty} = \varprojlim G_n$ is acted on by the direct-limit group $G_\infty = \varinjlim G_n$.  Each group $G_n$ possesses a $G_n$-quasi-invariant probability measure $\mu_n$.  

It is then possible to define a projective-limit probability measure $\mu_\infty = \varprojlim \mu_n$ on $\overline{G_\infty}$ using Kolmogorov's theorem.  If this measure is quasi-invariant under the action of $G_\infty$ on $\overline{G_\infty}$ then it is possible to define a unitary ``regular representation'' of $G_\infty$ on $L^2(\overline{G_\infty},\mu_\infty)$.  This ``regular representation'' can then be decomposed into irreducible representations.

In fact, precisely this scheme was used by Doug Pickrell in \cite{Pi1987a} to study analysis on an infinite-dimensional Grassmannian space and later by Olshanski and Borodin in \cite{BOl} to develop a theory of harmonic analysis on the infinite-dimensional unitary group $U(\infty)$.  The role played by probability theory in the latter context was crucial.  In fact, the problem was shown to be related to the study of infinite point processes.  Most intriguingly, probabilistic models from statistical mechanics appeared in the decomposition.

It would be interesting to consider a similar  analysis on the infinite-dimensional symmetric space $G_\infty/K_\infty$ and the horocycle space $G_\infty/M_\infty N_\infty$.  That is, one would construct projective-limit spaces $\overline{G_\infty/K_\infty}$ and $\overline{G_\infty/M_\infty N_\infty}$ which possess $G_\infty$-quasi-invariant measures.  
The problem, then, would be to decompose the corresponding unitary representations of $G_\infty$ on $L^2(\overline{G_\infty/K_\infty})$ and $L^2(\overline{G_\infty/M_\infty N_\infty})$ into irreducible subrepresentations.  One interesting question is whether those representations decompose into direct integrals of  unitary spherical and conical representations of $G_\infty$, respectively.

Also of interest is whether a sort of Radon transform may be constructed between functions on $G_\infty/K_\infty$ and functions on $G_\infty/M_\infty N_\infty$.  In fact, for spaces of regular functions this has been done in the recent paper \cite{HO}.  However, it would be interesting if it were possible to develop a Hilbert space analogue of the Radon transform, perhaps mapping between functions in $L^2(\overline{G_\infty/K_\infty})$ and functions  in $L^2(\overline{G_\infty/M_\infty N_\infty})$.
\normalsize

\section{Appendix:Admissiblility of Classical Direct Limits}
The aim of this appendix is to show that each classical example is admissible.  For the explicit matrix realizations of the compact-type Riemannian symmetric spaces, see \cite[p. 446, 451--455]{He1}.

The classical propagated direct systems of Riemannian symmetric spaces are listed in Table~\ref{classicalSymmetricPairs}, where each row gives a noncompact-type symmetric space $G_n/K_n$ and its simply-connected compact dual space $U_n/K_n$, and where the restricted roots exhibit the Dynkin diagram $\Psi_n$.  For each row, the limit $G_\infty/K_\infty = \varinjlim G_n/K_n$ is propagated and also that it is possible to choose Cartan subalgebras of $U_n$ for each $n\in\N$ so that $U_\infty = \varinjlim U_n$ is a propagated direct-limit group  (see, for instance, \cite[Section 2]{OW2013} or \cite[Section 3]{W2011}).

Note that in each row of Table~\ref{classicalSymmetricPairs}, the symmetric space $U_n/K_n$ is simply-connected.  However, in certain rows the group $U_n$ is not simply-connected.  We may remove this obstruction simply by passing to the universal cover $\widetilde{U_n}$ of $U_n$.  In fact, that the involution $\theta_n$ on $\gu_n$ integrates to an involution $\widetilde{\theta_n}$ on $\widetilde{U_n}$.  Denote the fixed-point subgroup for $\widetilde{\theta_n}$ in $\widetilde{U_n}$ by  $\widetilde{K_n}$.  By simply-connectedness all of the inclusions on the Lie algebra level integrate to inclusions on the group level, so that $\widetilde{U_n}/\widetilde{K_n}$ forms a propagated direct system of compact-type symmetric spaces.  Furthermore, one sees that if $p:\widetilde{U_n}\rightarrow U_n$ is the covering map, then $p\left(\widetilde{K_n}\right)\subseteq K_n$.  Hence $p$ factors to a covering map from $\widetilde{U_n}/\widetilde{K_n}$ to $U_n/K_n$ (see \cite[p. 213]{He1}).  Since $U_n/K_n$ is already simply-connected, we see that $\widetilde{U_n}/\widetilde{K_n}$ is diffeomorphic to $U_n/K_n$.  

{\footnotesize
\begin{equation}\label{classicalSymmetricPairs}
\begin{tabular}{|c|l|l|l|c|} \hline
% \multicolumn{5}{| c |}{Table I}\\
\multicolumn{5}{|c|}{}\\
\multicolumn{5}{| c |}
{Classical direct systems of irreducible Riemannian symmetric spaces}\\
\multicolumn{5}{|c|}{}\\
\hline \hline
\multicolumn{1}{|c}{} &  \multicolumn{1}{|c}{$G_n$} &
        \multicolumn{1}{|c}{$U_n$} &
        \multicolumn{1}{|c}{$K_n$} &     
        \multicolumn{1}{|c|}{$\Psi_n$}  \\ \hline \hline
$1$ &  $\SL(n,\C)$           & $\SU (n)\times \SU(n)$                                        & $\diag\, \SU(n)$                   & $A_{n-1}$  \\ \hline
$2$ &  $\Spin	(2n+1,\C)$         & $\begin{matrix}\Spin (2n+1)\times\\ \Spin (2n+1)\end{matrix}$ & $\diag\, \Spin (2n+1)$             & $B_n$      \\ \hline
$3$ &  $\Spin(2n,\C)$          & $\begin{matrix}\Spin (2n)\times\\ \Spin (2n)\end{matrix}$     & $\diag\, \Spin(2n)$                & $D_n$      \\ \hline
$4$ &  $\Sp(n,\C)$           & $\Sp (n)\times \Sp (n)$                                       & $\diag\, \Sp (n)$                  & $C_n$      \\ \hline
$5_1$ &  $\SU(p,n-p)$            & $\SU(n)$                                                      & $\mathrm{S}(\U (p)\times \U (n-p))$ & $C_p$      \\ \hline
$5_2$ &  $\SU(n,n)$            & $\SU(2n)$                                                      & $\mathrm{S}(\U (n)\times \U (n))$ & $C_n$      \\ \hline
$6_1$ &  $\SO_0(p,n-p)$            & $\SO (n)$                                                     & $\SO (p) \times \SO (n-p)$           & $B_p$      \\ \hline
$6_2$ &  $\SO_0(n,n)$            & $\SO (2n)$                                                     & $\SO (n) \times \SO (n)$           & $B_n$      \\ \hline
$7_1$ & $\Sp(p,n-p)$            & $\Sp (n)$                                                     & $\Sp (p) \times \Sp (n-p)$           & $C_p$      \\ \hline
$7_2$ & $\Sp(n,n)$            & $\Sp (2n)$                                                     & $\Sp (n) \times \Sp (n)$           & $C_n$      \\ \hline

$8$ &  $\SL(n,\R)$           & $\SU (n)$                                                     & $\SO (n)$                          & $A_{n-1}$  \\ \hline
$9$ &  $\SL(n,\H)$            & $\SU (2n)$                                                    & $\Sp (n)$                          & $A_{n-1}$  \\ \hline
$10_1$& $\SO^*(4n)$           & $\SO (4n)$                                                    & $\U (2n)$                          & $C_{n}$    \\ \hline
$10_2$& $\SO^*(2(2n+1))$       & $\SO (2(2n+1))$                                               & $\U (2n+1)$                        & $C_{n}$    \\ \hline
$11$ & $\Sp(n,\R)$           & $\Sp (n)$                                                     & $\U (n)$                           & $C_n$      \\ \hline
\end{tabular}
\end{equation}
}

%We recall from Section 2 the Cartan decompositions $\gg_n = \gk_n \oplus \gp_n$ and $\gu_n = \gk_n \oplus i\gp_n$ and the maximal abelian subalgebra $\ga_n$ of $\gp_n$.  Because each row gives a propagated direct system, we known that $\ga_n\subseteq \ga_{n+1}$ and that either $\Psi_{n+1}=\Psi_n$ or $\Psi_{n+1}$ extends $\Psi_n$ by adding roots only on the lefthand side of the Dynkin diagram.

\subsection{A General Strategy for Proving Admissibility}
The embedding $G_n\hookrightarrow G_{n+1}$ takes the form
\begin{equation}
\label{inclusionTypeOne}
A \mapsto \left(\begin{array}{ccc}
      I &   & \\
        & A & \\
        &   & I
                \end{array}\right)
\end{equation}
for the systems in rows $5_2$, $6_2$, and $7_2$. 
In all other cases in Table~\ref{classicalSymmetricPairs}, the embedding $G_n\hookrightarrow G_{n+1}$ takes the form 
\begin{equation}
\label{inclusionTypeTwo}
A \mapsto \left(\begin{array}{cc}
          A & \\
            & I
                \end{array}\right),
\end{equation}
where $I$ is a $1\times 1$, $2\times 2$, or $4 \times 4$ identity matrix. 

Suppose we can choose  $\ga_n$ for each $n$ in such a way that 
\begin{equation}
\label{embedFormOne}
\ga_{n+1}\subseteq \left(\begin{array}{ccc}
                          * &   0   & * \\
                          0 & \ga_n & 0\\
                          * &   0   & *
                         \end{array}\right)
\end{equation}
or
\begin{equation}
\label{embedFormTwo}
\ga_{n+1}\subseteq \left(\begin{array}{cc}
                          \ga_n &  0\\
                            0   & *
                         \end{array}\right)
\end{equation}
(depending on the type of embedding $G_n\hookrightarrow G_{n+1}$).  In this case, since $\ga_n$ commutes with $M_n = Z_{K_n}(\ga_n)$ by definition, it follows from (\ref{embedFormOne}) and (\ref{embedFormTwo}) that $\ga_{n+1}$ commutes with
\[
  M_n\cong         \left(\begin{array}{ccc}
                          I &   0   & 0 \\
                          0 &  M_n & 0\\
                          0 &   0   & I
                         \end{array}\right)  
\]
or
\[
   M_n\cong              \left(\begin{array}{cc}
                            M_n &  0\\
                            0   &  I
                         \end{array}\right),
\]
respectively, depending on the type of embedding $G_n\hookrightarrow G_{n+1}$.  In other words, $M_n \leq Z_{K_{n+1}}(\ga_{n+1}) = M_{n+1}$

Hence, in order to prove that a propagated direct limit is admissible, it is sufficient to show that either (\ref{embedFormOne}) or (\ref{embedFormTwo}) holds.  In most cases, our proof of admissibility will take this form.

\subsection{$U_n = L_n \times L_n$ and $K_n = \diag\ L_n$} 
This case corresponds to the first four rows in Table~\ref{classicalSymmetricPairs}.  In this case, one sees that 

\begin{align*}
  \gu_n & = \gl_n \times \gl_n \\
  \gk_n & = \{(X,X)\in \gu_n |X\in\gl_n \} \\
  i\gp_n & = \{(X,-X)\in\gu_n | X\in\gl_n\}.
\end{align*}  Furthermore, if we fix a Cartan subalgebra $\gh_n\subseteq\gl_n$ for each $n$, then we can choose 
\[
   i\ga_n = \{(X,-X)\in\gu_n|X\in\gh_n\}.
\]

   Now suppose that $g\in L_n$ and that $(g,g)\in M_n = Z_{K_n}(\ga_n)$.  Then $g\in Z_{L_n}(\gh_n)$; that is, $g$ centralizes the Cartan subalgebra $\gh_n$ of $\gl_n$.  Since $K_n$ is connected, it follows that $g\in H_n \equiv \exp(\gh_n)$.  Thus $M_n = \diag\ H_n$ for each $n$.  It follows that $M_k\leq M_n$ for $k\leq n$.

\subsection{$\Rank(G_\infty/K_\infty) \equiv \dim \ga_\infty< \infty$} 
This case was already discussed in \cite{HO} and corresponds to rows $5_1$, $6_1$, and $7_1$ in Table~\ref{classicalSymmetricPairs}. If $\dim \ga_\infty < \infty$, then for $k$ large enough, one has $\ga_k = \ga_\infty$.  Suppose $k\leq n$ and $g\in M_k$.  That is, $g\in K_k$ and $g$ centralizes $\ga_k$.  But $\ga_k=\ga_n = \ga_\infty$ and $K_k\leq K_n$.  Thus $g\in M_n$.

\subsection{$\Rank(G_n / K_n) = \Rank(G_n)$ for all $n\in\N$}
This case corresponds to rows 8 and 11 in Table~\ref{classicalSymmetricPairs}.   One has that $\ga_n$ is a Cartan subalgebra for $\gg_n$.  In particular, $Z_{\gg_n}(\ga_n) = \ga_n$ and so $\gm_n = \{0\}$ for all $n\in\N$. 

For example, if we let $G_n = \SL(n,\R)$ and $K_n=\SO(n)$ and make the standard choice of $\ga_n = \{\diag(a_1,\ldots,a_n) |a_i\in\R\}$, then one has $M_n = \{\diag(\pm 1, \ldots, \pm 1)\}$.  Thus $M_k\leq M_n$ for $k\leq n$.

\subsection{$U_n / K_n = \SU(2n)/S(\SU(n) \times \SU(n))$}
This case corresponds to row $5_2$ in Table~\ref{classicalSymmetricPairs}.  One has $\gg_n = \su(n,n)$, $\gu_n = \su(2n)$, and $\gk_n = \mathfrak{s}(\su(n)\oplus\su(n))$.  The involution is given by $\theta_n: A \mapsto J_nAJ_n^{-1}$, where 
\[
J_n =\left(
           \begin{array}{cc}
             I_n & \\
                 & -I_n
            \end{array} 
   \right).
\]

More explicitly, one has
\begin{align*}
          \gu_n &  = \left\{\left.
                           \left(\begin{array}{cc}
                            A & B \\
                            -B^* & D
                           \end{array}\right)
                           \in\mathrm{M}(2n,\C)
                           \right|
                           \begin{array}{l}
                           A^* = -A\text{, } D^* = -D \text{,} \\
                            \text{and } \Tr(A) + \Tr(D) = 0
                           \end{array}
                      \right\}
         \\
           \gk_n & = \left\{\left.
                           \left(\begin{array}{cc}
                            A & 0 \\
                            0 & D
                           \end{array}\right)
                           \in\mathrm{M}(2n,\C)
                           \right|\begin{array}{l}
                           A^* = -A\text{, } D^* = -D \text{,} \\
                            \text{and }\Tr(A) + \Tr(D) = 0
                           \end{array}
                      \right\}
          \\       
          i\gp_n  & =  \left\{%\left.
                           \left(\begin{array}{cc}
                            0 & B \\
                            -B^* & 0
                           \end{array}\right)
                           \in\mathrm{M}(2n,\C)
                           %\right|
                           %B\in\mathrm{M}(n,\C)
                      \right\}.
\end{align*}          

We choose 
\[
   i\ga_n =\left\{\left.\left(
           \begin{array}{ccc|ccc}
                  & & &      &        & a_n \\ 
                  & & &      & \udots &   \\
                  & & &  a_1 & &\\
             \hline
                           &        & -a_1 & & &\\
                           & \udots &      & & & \\   
             -a_n          &        &      & & & 
            \end{array} 
   \right)\right|a_i\in\R \right\}
\]
Thus condition~(\ref{embedFormOne}) is satisfied and so $G_\infty/K_\infty$ is admissible.

\subsection{$U_n / K_n = \SO(2n)/(\SO(n) \times \SO(n))$}
This case corresponds to row $6_2$ in Table~\ref{classicalSymmetricPairs}.  One has $\gg_n = \so(n,n)$, $\gu_n = \so(2n)$, and $\gk_n = \so(n)\oplus\so(n)$.  The involution is given by $\theta_n: A \mapsto J_nAJ_n^{-1}$, where 
\[
J_n =\left(
           \begin{array}{cc}
             I_n & \\
                 & -I_n
            \end{array} 
   \right).
\]

More explicitly, one has
\begin{align*}
          \gu_n &  = \left\{\left.
                           \left(\begin{array}{cc}
                            A & B \\
                            -B^\mathrm{T} & D
                           \end{array}\right)
                           \in\mathrm{M}(2n,\R)
                           \right|
                           A^\mathrm{T} = -A\text{ and } D^\mathrm{T} = -D                          
                      \right\}
         \\
           \gk_n & = \left\{\left.
                           \left(\begin{array}{cc}
                            A & 0 \\
                            0 & D
                           \end{array}\right)
                           \in\mathrm{M}(2n,\R)
                           \right|
                           A^\mathrm{T} = -A\text{ and } D^\mathrm{T} = -D                       
                      \right\}
          \\       
          i\gp_n  & =  \left\{%\left.
                           \left(\begin{array}{cc}
                            0 & B \\
                            -B^\mathrm{T} & 0
                           \end{array}\right)
                           \in\mathrm{M}(2n,\R)
                           %\right|
                           %B\in\mathrm{M}(n,\R)
                      \right\}.
\end{align*}   

We choose 
\[
   i\ga_n =\left\{\left.\left(
           \begin{array}{ccc|ccc}
                  & & &      &        & a_n \\ 
                  & & &      & \udots &   \\
                  & & &  a_1 & &\\
             \hline
                           &        & -a_1 & & &\\
                           & \udots &      & & & \\   
             -a_n          &        &      & & & 
            \end{array} 
   \right)\right|a_i\in\R\right\}.
\]
Thus condition (\ref{embedFormTwo}) is satisfied and so $G_\infty/K_\infty$ is admissible.

\subsection{$U_n / K_n = \Sp(2n)/(\Sp(n) \times \Sp(n))$}
This case corresponds to row $7_2$ in Table~\ref{classicalSymmetricPairs}.  One has $\gg_n = \lsp(n,n)$, $\gu_n = \lsp(2n)$, and $\gk_n = \lsp(n)\oplus\lsp(n)$.  The involution is given by $\theta_n: A \mapsto J_nAJ_n^{-1}$, where 
\[
J_n =\left(
           \begin{array}{cc}
             I_n & \\
                 & -I_n
            \end{array} 
   \right).
\]

More explicitly, one has
\begin{align*}
          \gu_n &  = \left\{\left.
                           \left(\begin{array}{cc}
                            A & B \\
                            -B^* & D
                           \end{array}\right)
                           \in\mathrm{M}(2n,\H)
                           \right|
                           A^* = -A\text{ and } D^* = -D                              
                      \right\}
         \\
           \gk_n & = \left\{\left.
                           \left(\begin{array}{cc}
                            A & 0 \\
                            0 & D
                           \end{array}\right)
                           \in\mathrm{M}(2n,\H)
                           \right| 
                           A^* = -A\text{ and } D^* = -D                          
                      \right\}
          \\       
          i\gp_n  & =  \left\{%\left.
                           \left(\begin{array}{cc}
                            0 & B \\
                            -B^* & 0
                           \end{array}\right)
                           \in\mathrm{M}(2n,\H)
                           %\right|
                           %B\in\mathrm{M}(n,\H)
                      \right\}.
\end{align*}          

We choose 
\[
   i\ga_n =\left\{\left.\left(
           \begin{array}{ccc|ccc}
                  & & &      &        & a_n \\ 
                  & & &      & \udots &   \\
                  & & &  a_1 & &\\
             \hline
                           &        & -a_1 & & &\\
                           & \udots &      & & & \\   
             -a_n          &        &      & & & 
            \end{array} 
   \right)\right|a_i\in\R\right\}.
\]
Thus condition~(\ref{embedFormOne}) is satisfied and so $G_\infty/K_\infty$ is admissible.

\subsection{$U_n / K_n = \SU(2n)/\Sp(n)$}
This case corresponds to row $9$ in Table~\ref{classicalSymmetricPairs}. One has $\gg_n = \lsl(n,\H)$, $\gu_n = \su(2n)$ and $\gk_n = \lsp(n)$.  The involution is given by $\theta_n: A \mapsto J_nAJ_n^{-1}$, where $J_n$ is given by 
\begin{equation}
\label{J}
J_n = \left(
           \begin{array}{ccc}
             \inverseJblock{1} & &\\
             & \ddots & \\
             & & \inverseJblock{1}
            \end{array} 
   \right).
\end{equation}

One can also obtain the same symmetric space by using the involution $\widetilde{\theta}_n: A \mapsto \widetilde{J}_n A\widetilde{J}_n^{-1}$, where
\begin{equation}
\label{JTilde}
\widetilde{J}_n = \left(
           \begin{array}{ccc|ccc}
               & & & -1 &        &    \\ 
               & & &    & \ddots &    \\
               & & &    &        & -1 \\
             \hline
               1 &        &   & & &\\
                 & \ddots &   & & &\\
                 &        & 1 & & & 
            \end{array} 
   \right).
\end{equation}
The calculations will be easier if we use $\widetilde{\theta}_n$ instead of $\theta_n$.  However, we must use $\theta_n$ in order for the inclusions $U_n\rightarrow U_{n+1}$ to take the form of (\ref{inclusionTypeTwo}).   We can move freely between these pictures, however, because $J_n = E_{\sigma_n}J_n E_{\sigma_n}^{-1}$, where $E_{\sigma_n}\in M(2n,\C)$ is the permutation matrix corresponding to the permutation 
\[\sigma = (1\ n) (2\ (n+1)) \cdots ((n-1)\ 2n)\in S_{2n}.\]
In other words, the rows and columns are interwoven, so that the first $n$ basis elements of $\C^{2n}$ are mapped to odd-numbered basis elements and the final $n$ basis elements of $\C^{2n}$ are sent to even-numbered basis elements.

We proceed by using $\widetilde{\theta}_n$.  We have
\begin{align*}
          \su(2n) = \gu_n &  = \left\{\left.
                           \left(\begin{array}{cc}
                            A & B \\
                            -B^* & D
                           \end{array}\right)
                           \in\mathrm{M}(2n,\C)
                           \right|
                           \begin{array}{l}
                           A^* = -A\text{, } D^* = -D \text{, and} \\
                            \Tr(A) + \Tr(D) = 0
                           \end{array}
                      \right\}
         \\
          \lsp(n) \cong \gk_n & = \left\{\left.
                           \left(\begin{array}{cc}
                            A & B \\
                            -\overline{B} & \overline{A}
                           \end{array}\right)
                           \in\mathrm{M}(2n,\C)
                           \right|
                           \begin{array}{l}
                           A^* = -A \\
                           \text{and } B^\mathrm{T} = B
                           \end{array}
                      \right\}
          \\       
          i\gp_n  & =  \left\{\left.
                           \left(\begin{array}{cc}
                            A & B \\
                            \overline{B} & -\overline{A}
                           \end{array}\right)
                           \in\mathrm{M}(2n,\C)
                           \right|
                           \begin{array}{l}
                           A^* = -A, B^\mathrm{T} = -B,\\
                           \text{and } \Tr(A)=0
                           \end{array}
                      \right\}.
\end{align*}          

There is a $\widetilde{\theta}_n$-stable Cartan subalgebra
\[
   \widetilde{\gh}_n =\left\{\left.\left(
           \begin{array}{cccc}
             ia_1 &        & \\
                 & \ddots & \\
                 &        & ia_{2n}
             
            \end{array} 
   \right)\right|a_i\in\R \text{ and }\sum_{i=1}^{2n} a_i = 0\right\}
\]
for $\gg_n = \so^*(4n)$, and we can choose
\[
   i\widetilde{\ga}_n =\left\{\left.\left(
           \begin{array}{ccc|ccc}
             ia_1 &        &      & & &\\ 
                  & \ddots &      & & &\\
                  &        & ia_n & & &\\
             \hline
             & & & ia_{1} &        &\\
             & & &         & \ddots & \\
             & & &         &        & ia_n
            \end{array} 
   \right)\right|a_i\in\R \text{ and }\sum_{i=1}^n a_i = 0\right\}.
\]

We now proceed to the $\theta_n$ picture.  Conjugation of $\widetilde{\gh}_n$  by 
$E_{\sigma_n}$ (followed by renumbering the indices) yields the $\theta_n$-stable Cartan subalgebra
\[
   \gh_n = \widetilde{\gh}_n =\left\{\left.\left(
           \begin{array}{cccc}
             ia_1 &        & \\
                 & \ddots & \\
                 &        & ia_{2n}
             
            \end{array} 
   \right)\right|a_i\in\R \text{ and }\sum_{i=1}^{2n} a_i = 0\right\}.
\]
Finally, conjugation of $\widetilde{\ga}_n$ by $E_{\sigma_n}$ yields
\[
 i\ga_n =\left\{\left.\left(
           \begin{array}{ccccccc}
             i a_1 &        &      &          & & & \\ 
                  &  i a_1  &      &          & & & \\
                  &        & i a_2 &          & & & \\
                  &        &      & i a_2 & & & \\
             & & & & \ddots &          & \\
             & & & &        & i a_n & \\
             & & & &        &          & i a_n
            \end{array} 
   \right)\right|a_i\in\R \text{ and } \sum_{i=1}^n a_i = 0\right\}.
\]
While condition (\ref{embedFormTwo}) is not quite satisfied, we do have that
\begin{equation}
\label{embedFormThree}
\ga_{n+1}\subseteq \left(\begin{array}{cc}
                          \ga_n + \C \mathrm{Id} &  0\\
                            0   & *
                         \end{array}\right).
\end{equation}
Since $\gm_n$ centralizes $\ga_n$, it follows that $\gm_n$ commutes with $\ga_n + \C \mathrm{Id}$.  Thus by (\ref{embedFormThree}), it follows that $\gm_n$ commutes with $\ga_{n+1}$.
Thus $\gm_m\subseteq\gm_n$ for $m\leq n$, and it follows that $G_\infty/K_\infty$ is admissible.

\begin{comment}
One then has
\[
    \gm_n =\left\{\left.\left(
           \begin{array}{ccc|ccc}
             ia_1 &        &      & & &\\ 
                  & \ddots &      & & &\\
                  &        & ia_n & & &\\
             \hline
             & & & ia_{1} &        &\\
             & & &        & \ddots & \\
             & & &        &        & ia_n
            \end{array} 
   \right)\right|a_i\in\R \text{, and } \sum_{i=0}^{2n} a_i = 0 \right\}
\]
\end{comment}

\subsection{$U_n / K_n = \SO(4n)/\U(2n)$}        
This case corresponds to row $10_1$ in Table~\ref{classicalSymmetricPairs}. One has $\gg_n = \so^*(4n)$, $\gu_n = so(4n)$ and $\gk_n = \gu(2n)$.  The involution is given by $\theta_n: A \mapsto J_nAJ_n^{-1}$, where $J_n$ is given by~(\ref{J}).  
As in the previous example, one can also obtain the same symmetric space by using the involution $\widetilde{\theta}_n: A \mapsto \widetilde{J}_n A\widetilde{J}_n^{-1}$, where $\widetilde{J}_n$ is given by~(\ref{JTilde}).

We work first on the $\widetilde{\theta}_n$-side.  We have
\begin{align*}
          \so(4n) = \gu_n &  = \left\{\left.
                           \left(\begin{array}{cc}
                            A & B \\
                            -B^\mathrm{T} & D
                           \end{array}\right)
                           \in\mathrm{M}(4n,\R)
                           \right|
                           A^\mathrm{T} = -A\text{ and } D^\mathrm{T} = -D 
                      \right\}
         \\
          \gu(2n) \cong \gk_n & = \left\{\left.
                           \left(\begin{array}{cc}
                            A & B \\
                            -B & A
                           \end{array}\right)
                           \in\mathrm{M}(4n,\R)
                           \right|
                           \begin{array}{l}
                           A^\mathrm{T} = -A \\
                           \text{and } B^\mathrm{T} = B
                           \end{array}
                      \right\}
          \\       
          i\gp_n  & =  \left\{\left.
                           \left(\begin{array}{cc}
                            A & B \\
                            B & -A
                           \end{array}\right)
                           \in\mathrm{M}(4n,\R)
                           \right|
                           \begin{array}{l}
                           A^\mathrm{T} = -A \\
                           \text{and } B^\mathrm{T} = -B
                           \end{array}
                      \right\}.
\end{align*}          

There is a $\widetilde{\theta}_n$-stable Cartan subalgebra
\[
   \widetilde{\gh}_n =\left\{\left.\left(
           \begin{array}{cccc}
             \Jblock{a_1} & & &\\ 
             & \Jblock{a_2} & &\\
             & & \ddots & \\
             & & & \Jblock{a_{2n}}
            \end{array} 
   \right)\right|a_i\in\R\right\}
\]
and we can choose
\[
  i\widetilde{\ga}_n = \scalebox{0.95}{$\left\{\left.\left(
           \begin{array}{cccccc|cccccc}
               0  & a_1 &        &      &     & & & & & & \\ 
             -a_1 & 0   &        &      &     & & & & & & \\
                  &     & \ddots &      &     & & & & & & \\
                  &     &        & 0    & a_n & & & & & & \\
                  &     &        & -a_n & 0   & & & & & & \\    
               \hline
 & & & & &   &    0  & -a_1 &        &     &      \\ 
 & & & & &   &   a_1 & 0    &        &     &      \\
 & & & & &   &       &      & \ddots &     &      \\
 & & & & &   &       &      &        & 0   & -a_n \\
 & & & & &   &       &      &        & a_n & 0    \\  
            \end{array} 
   \right)\right|a_i\in\R\right\}
$}.\]

\begin{comment}
One then has
\[
    \gm_n =\left\{\left.\left(
           \begin{array}{cccccc|cccccc}
               0  & a_1 &        &      &     &     & b_1  &  c_1 &        &     & \\ 
             -a_1 & 0   &        &      &     &     & c_1  & -b_1 &        &     & \\
                  &     & \ddots &      &     &     &      &      & \ddots &     & \\
                  &     &        & 0    & a_n &     &      &      &        & b_n &  c_n\\
                  &     &        & -a_n & 0   &     &      &      &        & c_n & -b_n\\    
               \hline 
          -b_1 & -c_1 &        &      &      &   &   0  & a_1 &        &      &     \\ 
          -c_1 &  b_1 &        &      &      &   & -a_1 & 0   &        &      &     \\
               &      & \ddots &      &      &   &      &     & \ddots &      &     \\
               &      &        & -b_n & -c_n &   &      &     &        & 0    & a_n \\
               &      &        & -c_n & b_n  &   &      &     &        & -a_n & 0   
            \end{array} 
   \right) \right|  \begin{array}{l}
                     a_i\in\R \\
                     b_i\in\R \\
                     c_i\in\R \\
                     d\in\R
                    \end{array}
                \right\}
\]
\end{comment}

Moving to the $\theta_n$-picture, we conjugate everything by $E_{\sigma_n}$ and renumber the indices to arrive at the $\theta_n$-stable Cartan algebra
\[
\gh_n =
\scalebox{0.70}{$
   \left\{\left.\left(
           \begin{array}{cccc}
         \fourBlock{-a_1}{-a_2}{a_1}{a_2} & & &  \\
             & \fourBlock{-a_3}{a_4}{a_3}{a_4} & & \\
             & & \ddots & \\
             & & & \fourBlock{-a_{2n-1}}{-a_{2n}}{a_{2n-1}}{a_{2n}}
            \end{array} 
   \right)\right|a_i\in\R\right\}
   $},
\]
and finally
\[
  i\ga_n =
  \scalebox{0.75}{$\left\{\left.\left(
           \begin{array}{cccc}
         \fourBlock{-a_1}{a_1}{a_1}{-a_1} & & &  \\
             & \fourBlock{-a_2}{a_2}{a_2}{-a_2} & & \\
             & & \ddots & \\
             & & & \fourBlock{-a_n}{a_n}{a_n}{-a_n}
            \end{array} 
   \right)\right|a_i\in\R\right\}
   $}.
\]
Hence $\ga_n$ is block-diagonal, and moving from $\ga_n$ to $\ga_{n+1}$ is simply a matter of adding another $4\times 4$ block.  Thus we see that condition (\ref{embedFormTwo}) is satisfied and hence $G_\infty/K_\infty$ is admissible.

\subsection{$U_n / K_n = \SO(2(2n+1))/\U(2n+1)$}
This case corresponds to row $10_2$ in Table~\ref{classicalSymmetricPairs}. One has $\gg_n = \so^*(2(2n+1))$, $\gu_n = so(4n)$ and $\gk_n = \gu(2n)$. As in the previous example, one can also obtain the same symmetric space by using the involution $\widetilde{\theta}_n: A \mapsto \widetilde{J}_n A\widetilde{J}_n^{-1}$, where $\widetilde{J}_n$ is given by~(\ref{JTilde}).

We first work on the $\widetilde{\theta}_n$ side.  We then have
\begin{align*}
          \so(2(2n+1)) = \gu_n &  = \left\{\left.
                           \left(\begin{array}{cc}
                            A & B \\
                            -B^\mathrm{T} & D
                           \end{array}\right)
                           \in\mathrm{M}(2(2n+1),\R)
                           \right|
                              \begin{array}{l}
                           A^\mathrm{T} = -A \\
                          \text{and }D^\mathrm{T} = -D 
                              \end{array}
                      \right\}
         \\
          \gu(2n+1) \cong \gk_n & = \left\{\left.
                           \left(\begin{array}{cc}
                            A & B \\
                            -B & A
                           \end{array}\right)
                           \in\mathrm{M}(2(2n+1),\R)
                           \right|
                           \begin{array}{l}
                           A^\mathrm{T} = -A \\
                           \text{and } B^\mathrm{T} = B
                           \end{array}
                      \right\}
          \\       
          i\gp_n  & =  \left\{\left.
                           \left(\begin{array}{cc}
                            A & B \\
                            B & -A
                           \end{array}\right)
                           \in\mathrm{M}(2(2n+1),\R)
                           \right|
                           \begin{array}{l}
                           A^\mathrm{T} = -A \\
                           \text{and } B^\mathrm{T} = -B
                           \end{array}
                      \right\}.
\end{align*}          

There is a $\widetilde{\theta}_n$-stable Cartan subalgebra
\[
\widetilde{\gh}_n = 
\scalebox{0.74}{$
   \left\{\left.\left(
           \begin{array}{cccccc|cccccc}
         0 &        &     &        &          &         & a_1 & & & & & \\ 
           &     0  & a_2 &        &          &         &     & & & & & \\ 
           &   -a_2 & 0   &        &          &         &     & & & & & \\
           &        &     & \ddots &          &         &     & & & & & \\
           &        &     &        & 0        & a_{n+1} &     & & & & & \\
           &        &     &        & -a_{n+1} & 0       &     & & & & & \\    
               \hline
     -a_1 & & & & & & 0 &          &         &        &           &     \\                    
          & & & & & &   &    0     & a_{n+2} &        &           &     \\ 
          & & & & & &   & -a_{n+2} & 0       &        &           &     \\
          & & & & & &   &          &         & \ddots &           &     \\
          & & & & & &   &          &         &        & 0         & a_{2n+1} \\
          & & & & & &   &          &         &        & -a_{2n+1} & 0   \\
            \end{array} 
   \right)\right|a_i\in\R\right\}
   $}.
\]
 and we can choose
\[
   i\widetilde{\ga}_n = \scalebox{0.9}{$\left\{\left.\left(
           \begin{array}{cccccc|cccccc}
           0 &       &     &        &      &     & & & & & & \\  
             &    0  & a_1 &        &      &     & & & & & & \\ 
             &  -a_1 & 0   &        &      &     & & & & & & \\
             &       &     & \ddots &      &     & & & & & & \\
             &       &     &        & 0    & a_n & & & & & & \\
             &       &     &        & -a_n & 0   & & & & & & \\    
               \hline
 & & & & & & 0 &       &      &        &     &      \\                
 & & & & & &   &    0  & -a_1 &        &     &      \\ 
 & & & & & &   &   a_1 & 0    &        &     &      \\
 & & & & & &   &       &      & \ddots &     &      \\
 & & & & & &   &       &      &        & 0   & -a_n \\
 & & & & & &   &       &      &        & a_n & 0    \\  
            \end{array} 
   \right)\right|a_i\in\R\right\}
$}.\]

\begin{comment}
One then has
\[
    \gm_n =\left\{\left.\left(
           \begin{array}{cccccc|cccccc}
         0 &        &     &        &      &     & d   &      &      &        &     & \\ 
           &     0  & a_1 &        &      &     &     & b_1  &  c_1 &        &     & \\ 
           &   -a_1 & 0   &        &      &     &     & c_1  & -b_1 &        &     & \\
           &        &     & \ddots &      &     &     &      &      & \ddots &     & \\
           &        &     &        & 0    & a_n &     &      &      &        & b_n &  c_n\\
           &        &     &        & -a_n & 0   &     &      &      &        & c_n & -b_n\\    
               \hline 
       -d &      &      &        &      &      & 0 &      &     &        &      &     \\                    
          & -b_1 & -c_1 &        &      &      &   &   0  & a_1 &        &      &     \\ 
          & -c_1 &  b_1 &        &      &      &   & -a_1 & 0   &        &      &     \\
          &      &      & \ddots &      &      &   &      &     & \ddots &      &     \\
          &      &      &        & -b_n & -c_n &   &      &     &        & 0    & a_n \\
          &      &      &        & -c_n & b_n  &   &      &     &        & -a_n & 0   
            \end{array} 
   \right) \right|  \begin{array}{l}
                     a_i\in\R \\
                     b_i\in\R \\
                     c_i\in\R \\
                     d\in\R
                    \end{array}
                \right\}
\]
\end{comment}

Moving to the $\theta_n$-picture, we conjugate everything by $E_{\sigma_n}$ and renumber the indices to arrive at the $\theta_n$-stable Cartan algebra
\[\gh_n =\scalebox{.79}{$
   \left\{\left.\left(
           \begin{array}{cccc}
           \Jblock{a_1} & & & \\
      &  \fourBlock{-a_2}{-a_3}{a_2}{a_3} & &  \\
             & & \ddots & \\
             & & & \fourBlock{-a_{2n-1}}{-a_{2n}}{a_{2n-1}}{a_{2n}}
            \end{array} 
   \right)\right|a_i\in\R\right\}
$},\]
and finally
\[
   i\ga_n =\scalebox{0.90}{$\left\{\left.\left(
           \begin{array}{cccc}
           \ZeroBlock & & & \\
       &   \fourBlock{-a_1}{a_1}{a_1}{-a_1} & &  \\
             & & \ddots & \\
             & & & \fourBlock{-a_n}{a_n}{a_n}{-a_n}
            \end{array} 
   \right)\right|a_i\in\R\right\}
$}.\]
Hence $\ga_n$ is block-diagonal, and moving from $\ga_n$ to $\ga_{n+1}$ is simply a matter of adding another $4\times 4$ block.  Thus we see that condition (\ref{embedFormTwo}) is satisfied and hence $G_\infty/K_\infty$ is admissible.

\end{document}